\documentclass[a4paper,10pt,final ]{amsart}
\usepackage[utf8x]{inputenc}
\usepackage{fouriernc, amsmath, amsthm, amsfonts, amssymb,epsfig,enumerate, wrapfig, scalerel, tikz,  color, accents, fancybox, rotating, comment, afterpage, changepage, colortbl,  float, xlop, blkarray}
\usepackage{hyperref}
\hypersetup{
    pdftoolbar=true,        
    pdfmenubar=true,        
    pdffitwindow=false,     
    pdfstartview={FitH},    
    pdftitle={A quantum categorification of the Alexander polynomial},    
    pdfauthor={Louis-Hadrien Robert and Emmanuel Wagner},     
    pdfsubject={A quantum categorification of the Alexander polynomial},   
    pdfcreator={Louis-Hadrien Robert and Emmanuel Wagner},   
    pdfproducer={Louis-Hadrien Robert and Emmanuel Wagner}, 
    pdfkeywords={}, 
    pdfnewwindow=true,      
    colorlinks=true,       
    linkcolor=red,          
    citecolor=teal,        
    filecolor=magenta,      
    urlcolor=violet,          
    linkbordercolor=red,
    citebordercolor=teal,
    urlbordercolor=violet,  
    linktocpage=true
}

\usepackage{setspace}
\usetikzlibrary{arrows}
\usetikzlibrary{decorations.markings, decorations.pathreplacing, decorations.pathmorphing}
\usetikzlibrary{decorations}
\usetikzlibrary{patterns}
\usetikzlibrary{positioning}
\usepackage{tikz-3dplot}
\newcounter {res}[section]
\numberwithin{res}{section}
\newtheorem{thm}[res]{Theorem}
\newtheorem*{theo}{Theorem}

\newtheorem{lem}[res]{Lemma}
\newtheorem{prop}[res]{Proposition}
\newtheorem{cor}[res]{Corollary}

\theoremstyle{definition}
\newtheorem{notation}[res]{Notation}
\newtheorem{dfn}[res]{Definition}
\newtheorem{rmk}[res]{Remark}
\newtheorem{obs}[res]{Observation}
\newtheorem{exa}[res]{Example}

\newtheorem{cjc}[res]{Conjecture}

\newcommand{\NB}[1]{\ensuremath{\vcenter{\hbox{#1}}}}

\newcommand{\NN}{\ensuremath{\mathbb{N}}} 
\newcommand{\ZZ}{\ensuremath{\mathbb{Z}}} 
\newcommand{\CC}{\ensuremath{\mathbb{C}}} 
\newcommand{\QQ}{\ensuremath{\mathbb{Q}}}
\newcommand{\RR}{\ensuremath{\mathbb{R}}}

\renewcommand{\SS}{\ensuremath{\mathbb{S}}}

\renewcommand {\Im}{\operatorname{Im}}
\newcommand {\Id}{\operatorname{Id}}
\newcommand{\id}{\mathrm{Id}}

\newcommand{\Ker}{\mathop{\mathrm{Ker}}\nolimits}

\newcommand{\End}{\mathop{\mathrm{End}}}

\newcommand{\gll}{\ensuremath{\mathfrak{gl}}}
\newcommand{\sll}{\ensuremath{\mathfrak{sl}}}

\newcommand{\ie}{i.~e.{} }

\newcommand{\resp}{resp.{} }
\newcommand\eqdef{\ensuremath{\stackrel{\textrm{def}}{=}}}

\newcommand{\listk}[1]{\ensuremath{\underline{#1}}}

\newcommand{\degext}{\ensuremath{\mathrm{deg}^{\Lambda}}}

\newcommand{\degT}{\ensuremath{\mathrm{deg}^{T}}}

\newcommand\kup[1]{\left\langle #1 \right\rangle}
\makeatletter
\DeclareFontFamily{OMX}{MnSymbolE}{}
\DeclareSymbolFont{MnLargeSymbols}{OMX}{MnSymbolE}{m}{n}
\SetSymbolFont{MnLargeSymbols}{bold}{OMX}{MnSymbolE}{b}{n}
\DeclareFontShape{OMX}{MnSymbolE}{m}{n}{
    <-6>  MnSymbolE5
   <6-7>  MnSymbolE6
   <7-8>  MnSymbolE7
   <8-9>  MnSymbolE8
   <9-10> MnSymbolE9
  <10-12> MnSymbolE10
  <12->   MnSymbolE12
}{}
\DeclareFontShape{OMX}{MnSymbolE}{b}{n}{
    <-6>  MnSymbolE-Bold5
   <6-7>  MnSymbolE-Bold6
   <7-8>  MnSymbolE-Bold7
   <8-9>  MnSymbolE-Bold8
   <9-10> MnSymbolE-Bold9
  <10-12> MnSymbolE-Bold10
  <12->   MnSymbolE-Bold12
}{}

\let\llangle\@undefined
\let\rrangle\@undefined
\DeclareMathDelimiter{\llangle}{\mathopen}
                     {MnLargeSymbols}{'164}{MnLargeSymbols}{'164}
\DeclareMathDelimiter{\rrangle}{\mathclose}
                     {MnLargeSymbols}{'171}{MnLargeSymbols}{'171}
\makeatother
 \newcommand\kups[1]{\left\llangle #1 \right\rrangle}

\newcommand{\F}{\mathcal{F}}

\newcommand{\BS}{\ensuremath{\mathcal{B}}}

\newcommand{\TL}{\ensuremath{\mathsf{TLF}}}

\newcommand{\Foam}{\ensuremath{\mathsf{Foam}}}
\newcommand{\qvg}{\ensuremath{\QQ\mathsf{-vect}_{\mathrm{gr}}}}

\newcommand{\syf}{\ensuremath{\mathcal{S}}}

\newcommand{\Xing}{\NB{\scalebox{1.3}{\ensuremath{\times}}}}

\newcommand{\HHH}{\ensuremath{H\!\!H\!\!H}}
\newcommand{\HH}{\ensuremath{H\!\!H}}

\newcommand{\HFK}{\ensuremath{H\!\!F\!\!K}}

\newcommand{\D}{\ensuremath{\mathcal{D}}} 

\newcommand{\ann}{\ensuremath{\mathcal{A}}}
\newcommand{\Vin}{\ensuremath{\mathcal{V}}}
\newcommand{\R}{\ensuremath{R}} 
\newcommand{\egd}{\ensuremath{\mathrm{egd}}} 
\newcommand{\gd}{\ensuremath{\mathrm{gd}}} 

\newcommand{\Hecke}{\ensuremath{H}}
\newcommand{\TLv}{\ensuremath{{T\!L\!F}}}

\newcommand{\skeinf}{\ensuremath{\mathrm{ann}^{\mathrm{full}}_q}}
\newcommand{\skein}{\ensuremath{\mathrm{ann}_q}}
\newcommand{\skeinp}{\ensuremath{\mathrm{ann}_q^\star}}
\newcommand{\wH}{\ensuremath{\widetilde{H}_{\gll_0}}}
\newcommand{\Hglo}{\ensuremath{H_{\gll_0}}}

\newcommand{\bd}{\ensuremath{\mathcal{BD}}}

\newcommand{\domp}{\ensuremath{\NB{\tikz[scale=0.15]{\draw (-1, -1) -- (-1, 1) -- (1,1) -- (1, -1) -- (-1, -1); \node[scale=0.6] at (0,0) {$+$}; }}}}
\newcommand{\domm}{\ensuremath{\NB{\tikz[scale=0.15]{\draw (-1, -1) -- (-1, 1) -- (1,1) -- (1, -1) -- (-1, -1); \node[scale=0.6] at (0,0) {$-$}; }}}}
\newcommand{\dom}[1]{\ensuremath{\NB{\tikz[scale=0.15]{\draw (-1, -1) -- (-1, 1) -- (1,1) -- (1, -1) -- (-1, -1); \node[scale=0.7] at (0,0) {$#1$}; }}}}
\newcommand{\domoo}{\ensuremath{\NB{\tikz[scale=0.15]{\draw (-2, -1) -- (-2, 1) -- (2,1) -- (2, -1) -- (-2, -1); \draw (0,-1) -- (0,1);
\node[scale=0.7] at (-1,0) {$0_a$};
\node[scale=0.7] at (1,0) {$0_b$};
}}}}
\newcommand{\Dom}{\ensuremath{\mathrm{Dom}}}
\newcommand{\dry}{\ensuremath{\mathrm{dry}}}

\newcommand{\imagesfolder}{.}

\def\co{\colon\thinspace}
\title{A quantum categorification of the Alexander polynomial}
\allowdisplaybreaks
 \author{Louis-Hadrien Robert}
 \author{Emmanuel Wagner}
\tikzset{->-/.style={decoration={
  markings,
  mark=at position .5 with {\arrow{>}}},postaction={decorate}}}
\tikzset{-<-/.style={decoration={
  markings,
  mark=at position .5 with {\arrow{<}}},postaction={decorate}}}

\newcommand{\digona}{\ensuremath{\vcenter{\hbox{\tikz[scale=0.4]{
\coordinate (B) at (0,0);
\coordinate (V1) at (0,0.5);
\coordinate (V2) at (0,2.5);
\coordinate (T) at (0,3);
\draw[white] (0, -0.5) -- (0, 3.5);
\draw[->] (B) -- (V1) node[midway, left] {\tiny{$m+n$}};
\draw[->] (V2) -- (T) node[midway, left] {\tiny{$m+n$}};
\draw[->] (V1) .. controls +(+0.5, +0.5) and +(+0.5, -0.5).. (V2) node[midway, right] {\tiny{$n$}};
\draw[->] (V1) .. controls +(-0.5, +0.5) and +(-0.5, -0.5).. (V2) node[midway, left] {\tiny{$m$}};
}}}}}

\newcommand{\verta}{\ensuremath{\vcenter{\hbox{\tikz[scale=0.4]{
\coordinate (B) at (0,0);
\coordinate (T) at (0,3);
\draw[white] (0, -0.5) -- (0, 3.5);
\draw[->] (B) -- (T) node[midway, right] {\tiny{$m+n$}};
}}}}}

\newcommand{\stgamma}{\ensuremath{\vcenter{\hbox{\tikz[scale=0.3]{
\coordinate (B) at (0,0);
\coordinate (V1) at (0,1);
\coordinate (V2) at (1,2);
\coordinate (T1) at (-2,3);
\coordinate (T2) at (0,3);
\coordinate (T3) at (2,3);
\draw[>-] (B) -- (V1) node [at start, below] {\tiny{$i+j+k$}};
\draw[->] (V1) -- (T1) node [at end, above] {\tiny{$i$}};
\draw[->] (V1)  -- (V2) node[midway, right] {\tiny{$j+k$}};
\draw[->] (V2) -- (T2) node[at end, above] {\tiny{$j$}};
\draw[->] (V2) -- (T3) node[at end, above] {\tiny{$k$}};
}}}}}

\newcommand{\stgammaprime}{\ensuremath{\vcenter{\hbox{\tikz[scale=0.3]{
\coordinate (B) at (0,0);
\coordinate (V1) at (0,1);
\coordinate (V2) at (-1,2);
\coordinate (T1) at (-2,3);
\coordinate (T2) at (0,3);
\coordinate (T3) at (2,3);
\draw[>-] (B) -- (V1) node [at start, below] {\tiny{$i+j+k$}};
\draw[->] (V1) -- (T3) node [at end, above] {\tiny{$k$}};
\draw[->] (V1)  -- (V2) node[midway, left] {\tiny{$i+j$}};
\draw[->] (V2) -- (T1) node[at end, above] {\tiny{$i$}};
\draw[->] (V2) -- (T2) node[at end, above] {\tiny{$j$}};
}}}}}

\newcommand{\stgammar}{\ensuremath{\vcenter{\hbox{\tikz[scale=0.3, yscale = -1]{
\coordinate (B) at (0,0);
\coordinate (V1) at (0,1);
\coordinate (V2) at (1,2);
\coordinate (T1) at (-2,3);
\coordinate (T2) at (0,3);
\coordinate (T3) at (2,3);
\draw[<-] (B) -- (V1) node [at start, above] {\tiny{$i+j+k$}};
\draw[-<] (V1) -- (T1) node [at end, below] {\tiny{$i$}};
\draw[-<] (V1)  -- (V2) node[midway, right] {\tiny{$j+k$}};
\draw[-<] (V2) -- (T2) node[at end, below] {\tiny{$j$}};
\draw[-<] (V2) -- (T3) node[at end, below] {\tiny{$k$}};
}}}}}

\newcommand{\stgammaprimer}{\ensuremath{\vcenter{\hbox{\tikz[scale=0.3, yscale = -1]{
\coordinate (B) at (0,0);
\coordinate (V1) at (0,1);
\coordinate (V2) at (-1,2);
\coordinate (T1) at (-2,3);
\coordinate (T2) at (0,3);
\coordinate (T3) at (2,3);
\draw[<-] (B) -- (V1) node [at start, above] {\tiny{$i+j+k$}};
\draw[-<] (V1) -- (T3) node [at end, below] {\tiny{$k$}};
\draw[-<] (V1)  -- (V2) node[midway, left] {\tiny{$i+j$}};
\draw[-<] (V2) -- (T1) node[at end, below] {\tiny{$i$}};
\draw[-<] (V2) -- (T2) node[at end, below] {\tiny{$j$}};
}}}}}

\newcommand{\squarec}{\ensuremath{\vcenter{\hbox{\tikz[xscale=0.65, yscale=0.55]{
\coordinate (B1) at (-1,0);
\coordinate (B2) at (1,0);
\coordinate (C1) at (-1,1.1);
\coordinate (D1) at (-1,1.9);
\coordinate (C2) at (1,0.9);
\coordinate (D2) at (1,2.1);
\coordinate (T1) at (-1,3);
\coordinate (T2) at (1,3);
\draw[->] (B1) -- (C1) node[at start, below] {\tiny{$n$}};
\draw[->] (C1) -- (D1) node[midway, left   ] {\tiny{$n+k $}};
\draw[->] (D1) -- (T1) node[at end , above ] {\tiny{$m$}};
\draw[->] (B2) -- (C2) node[at start, below] {\tiny{$m+l$}};
\draw[->] (C2) -- (D2) node[midway, right] {\tiny{$m+l-k$}};
\draw[->] (D2) -- (T2) node[at end, above] {\tiny{$n+l$}};
\draw[->] (D1) -- (D2) node[midway, above] {\tiny{$n+k-m$}};
\draw[->] (C2) -- (C1) node[midway, below] {\tiny{$k$}};
}}}}}

\newcommand{\squarecc}{\ensuremath{\vcenter{\hbox{\tikz[xscale=-0.65, yscale=0.55]{
\coordinate (B1) at (-1,0);
\coordinate (B2) at (1,0);
\coordinate (C1) at (-1,1.1);
\coordinate (D1) at (-1,1.9);
\coordinate (C2) at (1,0.9);
\coordinate (D2) at (1,2.1);
\coordinate (T1) at (-1,3);
\coordinate (T2) at (1,3);
\draw[->] (B1) -- (C1) node[at start, below] {\tiny{$n$}};
\draw[->] (C1) -- (D1) node[midway, right   ] {\tiny{$n+k $}};
\draw[->] (D1) -- (T1) node[at end , above ] {\tiny{$m$}};
\draw[->] (B2) -- (C2) node[at start, below] {\tiny{$m+l$}};
\draw[->] (C2) -- (D2) node[midway, left] {\tiny{$m+l-k$}};
\draw[->] (D2) -- (T2) node[at end, above] {\tiny{$n+l$}};
\draw[->] (D1) -- (D2) node[midway, above] {\tiny{$n+k-m$}};
\draw[->] (C2) -- (C1) node[midway, below] {\tiny{$k$}};
}}}}}

\newcommand{\squared}{\ensuremath{\vcenter{\hbox{\tikz[yscale=0.55, xscale=0.65]{
\coordinate (B1) at (-1,0);
\coordinate (B2) at (1,0);
\coordinate (C1) at (-1,0.9);
\coordinate (D1) at (-1,2.1);
\coordinate (C2) at (1,1.1);
\coordinate (D2) at (1,1.9);
\coordinate (T1) at (-1,3);
\coordinate (T2) at (1,3);
\draw[->] (B1) -- (C1) node[at start, below] {\tiny{$n$}};
\draw[->] (C1) -- (D1) node[midway, left   ] {\tiny{$m-j$}};
\draw[->] (D1) -- (T1) node[at end , above ] {\tiny{$m$}};
\draw[->] (B2) -- (C2) node[at start, below] {\tiny{$m+l$}};
\draw[->] (C2) -- (D2) node[midway, right] {\tiny{$n+l+j$}};
\draw[->] (D2) -- (T2) node[at end, above] {\tiny{$n+l$}};
\draw[->] (D2) -- (D1) node[midway, above] {\tiny{$j$}};
\draw[->] (C1) -- (C2) node[midway, below] {\tiny{$n+j-m$}};
}}}}}

\newcommand{\squaredd}{\ensuremath{\vcenter{\hbox{\tikz[yscale=0.55, xscale=-0.65]{
\coordinate (B1) at (-1,0);
\coordinate (B2) at (1,0);
\coordinate (C1) at (-1,0.9);
\coordinate (D1) at (-1,2.1);
\coordinate (C2) at (1,1.1);
\coordinate (D2) at (1,1.9);
\coordinate (T1) at (-1,3);
\coordinate (T2) at (1,3);
\draw[->] (B1) -- (C1) node[at start, below] {\tiny{$n$}};
\draw[->] (C1) -- (D1) node[midway, right   ] {\tiny{$m-j$}};
\draw[->] (D1) -- (T1) node[at end , above ] {\tiny{$m$}};
\draw[->] (B2) -- (C2) node[at start, below] {\tiny{$m+l$}};
\draw[->] (C2) -- (D2) node[midway, left] {\tiny{$n+l+j$}};
\draw[->] (D2) -- (T2) node[at end, above] {\tiny{$n+l$}};
\draw[->] (D2) -- (D1) node[midway, above] {\tiny{$j$}};
\draw[->] (C1) -- (C2) node[midway, below] {\tiny{$n+j-m$}};
}}}}}

\begin{document}
\maketitle

\begin{abstract}
Using a modified foam evaluation, we give a categorification of the Alexander polynomial of a knot.
We also give a purely algebraic version of this knot homology which makes it appear as the infinite page of a spectral sequence starting at the reduced triply graded link homology of Khovanov--Rozansky.
\end{abstract}

\tableofcontents

\section*{Introduction}
\label{sec:introduction}
\addtocontents{toc}{\protect\setcounter{tocdepth}{1}}
\subsection*{Context}
\label{sec:context}

The one variable Alexander polynomial $\Delta(q)\in \mathbb{Z}[q^{\pm 1}]$ plays a special and central role in knot theory due to its myriad of definitions. Among them, the ones originating in topology are probably the most studied. The Alexander polynomial can be defined as the order of the  Alexander module, which is the first homology of the infinite cyclic cover of the knot complement. It has been reinterpreted by Milnor as the Reidemeister torsion of the infinite cyclic cover and it can be computed using a Seifert matrix.\\

In an another direction, the Alexander polynomial can be given a completely combinatorial definition through a skein relation first discovered by Conway \cite{Conway}. This approach can be reinterpreted through the more general HOMFLY--PT link polynomial $P(a,q)$ \cite{HOMFLY, PT} discovered  after the development of quantum invariants: $P(1,q)=\Delta(q)$.\\

Moreover, the Alexander polynomial was given an even more quantum flavor by the works of Kauffman--Saleur \cite{KauffSa}, Rozansky--Saleur \cite{RS} and Viro \cite{Viro}. The Alexander polynomial can be obtained from the representation theory of $U_q(\gll(1|1))$ and $U_{\xi}(\sll_2)$. This point view  was further explored by Geer--Patureau-Mirand \cite{GPM} in their development of modified trace and non-semi-simple invariants.\\

Categorification sheded a new light on link polynomial invariants, and reinterpreted them as graded Euler caracteristic  of link homologies. The first categorification $\HFK$ of the Alexander polynomial was constructed by Ozsv\'a{}th--Szab\'o{} \cite{OS} and Rasmussen \cite{Rasmussen}, and it is called the knot Floer homology. It categorifies the topological flavor of the Alexander polynomial.\\

Petkova--Vertesi \cite{PV} and Ellis--Petkova--Vertesi \cite{EPV} provided an extension of this invariant to tangles and explained in this framework how to recover the underlying $U_q(\gll(1|1))$ structure.\\

The HOMFLY--PT polynomial is categorified by the triply graded link homology $\HHH$ of Khovanov--Rozansky \cite{MR2421131}. The relation between the HOMFLY-PT and the Alexander polynomials is conjectured by Dunfield--Gukov--Rasmussen \cite{DGR} to have a categorified version. They propose the existence of a spectral sequence from $\HHH$ to $\HFK$. However, they suggest that this spectral sequence could eventually converge to a different categorification of the Alexander polynomial.\\

\subsection*{Main result}
\label{sec:result}
In this paper, we provide a positive (partial) answer to this conjecture and construct a link homology theory for knots categoryfying the Alexander polynomial together with a spectral sequence from the reduced triply graded Khovanov--Rozansky knot homology :\\

\begin{theo}[Theorems \ref{thm:main1} and \ref{thm:spectral-seq}]
There exists a categorification of the Alexander polynomial of knots, together with a spectral sequence from the reduced triply graded Khovanov--Rozansky knot homology to this knot homology.
\end{theo}

This construction is as close as possible to categorifications of $\sll_n$ link polynomials by Khovanov and Khovanov--Rozansky and follows the framework of our previous works on these links homologies using foams and trivalent TQFT's.\\

Before discussing some details of the construction of the present paper, recall that various steps have been achieved to provide a complete positive answer to Dunfield--Gukov--Rasmussen conjecture. All of them start from the knot Floer homology. Ozsv\'a{}th--Szab\'o{} \cite{MR2574747} gave a version of knot Floer homology through a cube of resolution construction completely analogous to the one of the triply graded link homology. This construction was studied further by Gilmore \cite{Gilmore} in an twisted version and Manolescu \cite{Manolescu} in an untwisted one. Moreover Dowlin \cite{DowlinHHH} gave a construction of link homology categorifying the HOMFLY--PT using Floer technics together with a spectral sequence to $\HFK$.\\

\subsection*{Details of the Construction}
\label{sec:details-construction}

In a previous paper \cite{RW2}, we constructed a link homology theory categorifying the Reshetikhin--Turaev invariant associated with symmetric powers of the standard representation of $U_q(\gll_N)$ for $N\geq 1$. It turns out that for a braid closure diagram $D$ the chain complex $C_N(D)$ can naturally be seen as a subcomplex of $C_{N+1}(D)$. It suggested that one could consider naturally  a chain complex "$C_0(D)$" which should be related to $U_q(\gll(1|1))$ and the Alexander polynomial, but unfortunately this chain complex is trivial. Nevertheless one can consider a slightly larger subchain complex $C'_0(D)$ of $C_1(D)$, but depending beforhand on a choice of a base point on $D$.\\

This is not a surprise, since at the decategorified level, the Reshetikhin--Turaev invariant associated with the standard representation of $U_q(\gll(1|1))$ is trivial. It needs to be renormalized to obtain a non trivial invariant (which turns out to be the Alexander--Conway polynomial). For knots, the renormalization can be achieved  by computing the invariant on a $1-1$ tangle whose closure is the given knot. If follows from a topological argument that the invariant for knots does not depend on which $1-1$ tangle is taken to compute the invariant. For links proving that the previous construction does not depend on where the links is cut open is the aim of the theory of non-semi-simple invariants developed by Geer--Patureau-Mirand and their coauthors. Note that even for knots the topological argument requires the invariant to be defined for general diagrams. The main technical issue in the present paper is that the chain complex $C'_0(D)$ is defined only for $D$ a braid closure diagram  with a base point.\\

The chain complex $C_1(D)$ in \cite{RW2} is obtained using the classical hypercube of resolution expanding link diagrams in terms of certain planar trivalent graphs. On each of these planar trivalent graph, one applies a functor $\syf_{1}$ which goes from a category of cobordisms (objects are  planar trivalent graphs and morphisms are $2$-dimensional CW complexes called foams, see Section ~\ref{sec:symgl1}) to the category of graded $\QQ$-vector spaces. The chain complex $C'_0(D)$ is obtained by applying
a functor $\syf'_{0}$ on each planar trivalent graph of the hypercube of resolution. The definition of $\syf'_{0}$ depends on the chosen base point. In fact for every planar trivalent graph $\Gamma$, $\syf'_{0}(\Gamma)$ is a subspace of $\syf_1(\Gamma)$  and $C'_{0}(D)$ is a subcomplex of $C_1(D)$. This is detailed in Section~\ref{sec:vinylgraphBP}.\\

Given a planar trivalent graph $\Gamma$, the graded dimension of the vector space $\syf'_{0}(\Gamma)$ is be given by the representation theory of $U_q(\gll(1|1))$.
It turns out to be a quite difficult problem to compute the actual graded dimension of $\syf'_{0}(\Gamma)$ in general. We solved it only for particular choices of the base point. This is done in Section~\ref{sec:vinylgraphBP} and Appendix~\ref{sec:depth1} for basepoints which are on one of the two rightmost strands of $\Gamma$.\\

Since we understand the dimension of $\syf'_{0}(\Gamma)$ only in certain specific cases, we use a tool called \emph{rectification}. The idea is to change the meaning of a braid closure diagram when the base point is not at a convenient place and to prove a Markov like theorem in this new context. We believe this is of independant interest. \\

\begin{itemize}
\item Section~\ref{sec:annul-comb} is devoted to combinatorics of vinyl graphs and
  knots in the annulus with an emphasis on objects with base point. We describe the process of rectification and give a Markov type theorem in this context. 
\item Section~\ref{sec:symgl1} reviews an evaluation of closed foams in the $\gll_1$ setting and presents a foam free version of it. It gives the definition of the functor $\syf_1$.
\item Section~\ref{sec:vinylgraphBP} provides the core of the paper where we define a modified foam evaluation and define the functor $\syf'_{0}$. 
\item Section~\ref{sec:mathrmgl_0-link-homo} elaborates on Sections~\ref{sec:annul-comb} and \ref{sec:vinylgraphBP}. It gives the definition of the knot homology and provides the proof that it is a knot invariant. The graded dimension of $\syf'_{0}(\Gamma_\star)$ when the base point is on the rightmost strand of $\Gamma$.
\item Section~\ref{sec:an-algebr-appr} relates the previous homology to the triply graded link homology through a spectral sequence and constructs the differential $d_0$ conjectured by Dunfield--Gukov--Rasmussen.
\item Section~\ref{sec:colored-version} is concerned with some possible generalizations of the construction of this paper.
\item Appendix~\ref{prop:Xk-1-to-UC} computes the dimension of $\syf'_{0}(\Gamma_\star)$ when the base point is on the second rightmost strand of $\Gamma$.
\end{itemize}

\subsection*{Acknowledgments}
\label{sec:acknowledgments}
The main idea of this paper grew up at a SwissMAP conference in Riederalp. We warmly thank Anna Beliakova for creating such an amazing working atmosphere. L.-H.~R. was supported by the NCCR SwissMAP, funded by the Swiss National Science Foundation.

\addtocontents{toc}{\protect\setcounter{tocdepth}{2}}

\section{Annular combinatorics}
\label{sec:annul-comb}
\newcommand{\hg}{\ensuremath{\theta}} 

\subsection{Hecke algebra}
\label{sec:hecke-algebra}

\begin{dfn}
  \label{dfn:Hecke-algebra}
  Let $n$ be a non-negative integer. The \emph{Hecke algebra $\Hecke_n$} is the unitary $\ZZ[q, q^{-1}]$-algebra generated by $\left(\hg_i\right)_{1\leq i \leq n-1}$ and  subjected to the relations:
  \begin{align*}
    \hg_i^2&= [2] \hg_i &&\text{for $1\leq i \leq n-1$,} \\
    \hg_i\hg_j&= \hg_j\hg_i &&\text{for $1\leq i, j \leq n-1$ and $|i-j|\geq 2$,} \\
    \hg_i\hg_{i+1}\hg_i + \hg_{i+1} &=     \hg_{i+1}\hg_i\hg_{i+1} + \hg_{i}
&&\text{for $1\leq i \leq n-2$,} \\
  \end{align*}
  where $[2]=q+q^{-1}$.
\end{dfn}

\begin{rmk}
\label{rmk:Hecke-more-general}
This is the Hecke algebra associated with the Coxeter group presented by the Dynkin diagram $A_{n-1}$, that is the symmetric group on $n$ elements $S_n$. 
\end{rmk}

It is convenient to have a diagrammatical definition of the Hecke algebra. The generator $\hg_i$ in $H_n$, is represented by:
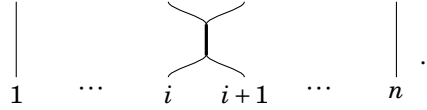
\begin{align*}
  \NB{
  \tikz{
  \draw (1,0) -- +(0,1) node [pos= 0, below] {$1$};
  \draw[white] (2,0) -- +(0,1) node [black, pos= 0, below] {$\cdots$};
  \draw[black, very thick] (3.5, 0.29) -- +(0,0.41);
  \draw (3,0) .. controls +(0,0.1) and +(0,-0.1) .. (3.5, 0.3)  node [pos= 0, below] {$i$} -- (3.5, 0.7) .. controls +(0,0.1) and +(0,-0.1) .. (3,1);
  \draw (4,0) .. controls +(0,0.1) and +(0,-0.1) .. (3.5, 0.3)  node [pos= 0, below] {$i+1$} -- (3.5, 0.7) .. controls +(0,0.1) and +(0,-0.1) .. (4,1);
  \draw[white] (5,0) -- +(0,1) node [black,pos= 0, below] {$\cdots$};
  \draw (6,0) -- +(0,1) node [pos= 0, below] {$n$};
  }
  }.
\end{align*}
The multiplication of generators is given by stacking diagrams one onto the other.
The first and the last relations can be rewritten as follows:

\begin{align}\label{eq:Hecke-diag}
  \NB{
  \tikz{
  \draw[very thick] (0,0) -- (0,0.3);
  \draw (0,0.3) .. controls + (0,0.1) and + (0, -0.1) .. (-0.5, 0.5) .. controls + (0,0.1) and + (0, -0.1) .. (0, 0.7);
  \draw (0,0.3) .. controls + (0,0.1) and + (0, -0.1) .. ( 0.5, 0.5) .. controls + (0,0.1) and + (0, -0.1) .. (0, 0.7);
  \draw[very thick] (0,0.7) -- (0,1);
  }
  }
 = [2]\NB{
  \tikz{
  \draw[very thick] (0,0) -- (0,1);}}\,,
  \qquad
  \NB{\tikz[scale=0.5]{
  \begin{scope}
  \draw[black, very thick] (3.5, 0.29) -- +(0,0.41);
  \draw (3,0) .. controls +(0,0.1) and +(0,-0.1) .. (3.5, 0.3) -- (3.5, 0.7) .. controls +(0,0.1) and +(0,-0.1) .. (3,1);
  \draw (4,0) .. controls +(0,0.1) and +(0,-0.1) .. (3.5, 0.3) -- (3.5, 0.7) .. controls +(0,0.1) and +(0,-0.1) .. (4,1);
  \end{scope}
  \begin{scope}[xshift = 1cm, yshift =1cm]
  \draw[black, very thick] (3.5, 0.29) -- +(0,0.41);
  \draw (3,0) .. controls +(0,0.1) and +(0,-0.1) .. (3.5, 0.3) -- (3.5, 0.7) .. controls +(0,0.1) and +(0,-0.1) .. (3,1);
  \draw (4,0) .. controls +(0,0.1) and +(0,-0.1) .. (3.5, 0.3) -- (3.5, 0.7) .. controls +(0,0.1) and +(0,-0.1) .. (4,1);
  \end{scope}
  \begin{scope}[xshift = 0, yshift =2cm]
  \draw[black, very thick] (3.5, 0.29) -- +(0,0.41);
  \draw (3,0) .. controls +(0,0.1) and +(0,-0.1) .. (3.5, 0.3) -- (3.5, 0.7) .. controls +(0,0.1) and +(0,-0.1) .. (3,1);
  \draw (4,0) .. controls +(0,0.1) and +(0,-0.1) .. (3.5, 0.3) -- (3.5, 0.7) .. controls +(0,0.1) and +(0,-0.1) .. (4,1);
  \end{scope}
  \draw (5,0) -- + (0,1);
  \draw (5,2) -- + (0,1);
  \draw (3,1) -- + (0,1); }}
  +
  \NB{\tikz[scale=0.5]{
  \begin{scope}
  \draw[black, very thick] (3.5, 0.29) -- +(0,2.41);
  \draw (3,0) .. controls +(0,0.1) and +(0,-0.1) .. (3.5, 0.3) -- (3.5, 2.7) .. controls +(0,0.1) and +(0,-0.1) .. (3,3);
  \draw (4,0) .. controls +(0,0.1) and +(0,-0.1) .. (3.5, 0.3) -- (3.5, 2.7) .. controls +(0,0.1) and +(0,-0.1) .. (4,3);
  \end{scope}
  \draw (2,0) -- + (0,3);}}\quad
  =
  \NB{\tikz[scale=0.5]{
  \begin{scope}[xshift = 1cm, yshift =0cm]
  \draw[black, very thick] (3.5, 0.29) -- +(0,0.41);
  \draw (3,0) .. controls +(0,0.1) and +(0,-0.1) .. (3.5, 0.3) -- (3.5, 0.7) .. controls +(0,0.1) and +(0,-0.1) .. (3,1);
  \draw (4,0) .. controls +(0,0.1) and +(0,-0.1) .. (3.5, 0.3) -- (3.5, 0.7) .. controls +(0,0.1) and +(0,-0.1) .. (4,1);
  \end{scope}
  \begin{scope}[xshift = 0cm, yshift =1cm]
  \draw[black, very thick] (3.5, 0.29) -- +(0,0.41);
  \draw (3,0) .. controls +(0,0.1) and +(0,-0.1) .. (3.5, 0.3) -- (3.5, 0.7) .. controls +(0,0.1) and +(0,-0.1) .. (3,1);
  \draw (4,0) .. controls +(0,0.1) and +(0,-0.1) .. (3.5, 0.3) -- (3.5, 0.7) .. controls +(0,0.1) and +(0,-0.1) .. (4,1);
  \end{scope}
  \begin{scope}[xshift = 1cm, yshift =2cm]
  \draw[black, very thick] (3.5, 0.29) -- +(0,0.41);
  \draw (3,0) .. controls +(0,0.1) and +(0,-0.1) .. (3.5, 0.3) -- (3.5, 0.7) .. controls +(0,0.1) and +(0,-0.1) .. (3,1);
  \draw (4,0) .. controls +(0,0.1) and +(0,-0.1) .. (3.5, 0.3) -- (3.5, 0.7) .. controls +(0,0.1) and +(0,-0.1) .. (4,1);
  \end{scope}
  \draw (3,0) -- + (0,1);
  \draw (3,2) -- + (0,1);
  \draw (5,1) -- + (0,1); }}
  +
    \NB{\tikz[scale=0.5]{
  \begin{scope}
  \draw[black, very thick] (3.5, 0.29) -- +(0,2.41);
  \draw (3,0) .. controls +(0,0.1) and +(0,-0.1) .. (3.5, 0.3) -- (3.5, 2.7) .. controls +(0,0.1) and +(0,-0.1) .. (3,3);
  \draw (4,0) .. controls +(0,0.1) and +(0,-0.1) .. (3.5, 0.3) -- (3.5, 2.7) .. controls +(0,0.1) and +(0,-0.1) .. (4,3);
  \end{scope}
  \draw (5,0) -- + (0,3);}}\,.\quad
\end{align}

For $n\geq 2$, the Hecke algebra $\Hecke_{n-1}$ embeds in $\Hecke_n$ by mapping $\hg_i$ to $\hg_i$. On a diagrammatic level this boils down to adding a free strands on the right of the diagrams. 

If $w$ is a word in the transpositions $\tau_i= (i\; i+1) \in S_n$, we define
\[\hg_w := \hg_{i_1} \hg_{i_2} \cdots \hg_{i_l}\quad \text{where $w= \tau_{i_1}
\tau_{i_2} \cdots \tau_{i_l}$}.\]

The following results can be found for instance, in Kassel--Turaev's book \cite{KT} (see also references therein):
\begin{prop}
  \label{prop:base-Hecke}
  The Hecke algebra $\Hecke_n$ is a free $\ZZ[q, q^{-1}]$-module of rank $n!$. For each $g$ in $S_n$, fix a reduced word $w_g$ presenting $g$. Then $\left(\hg_{w_g}\right)_{g\in S_n}$ is a $\ZZ[q, q^{-1}]$-base of $\Hecke_n$. Such a basis is called a word basis. 
\end{prop}

\begin{cor}
  \label{cor:base-induction}
  Let $n\geq 2$ be an integer and $\left(\hg_{w_g}\right)_{g\in S_{n-1}}$ be a word basis of $\Hecke_{n-1}$, then
  \[\left(\hg_{n-1}\hg_{n-2}\cdots\hg_{l}\hg_{w_g}\right)_{\substack{ 1\leq l \leq n \\ g \in S_{n-1}}}\]
is a word basis for $\Hecke_n$. In the previous list, when $l=n$, the expression $\hg_{n-1}\hg_{n-2}\cdots\hg_{l}$ equals $1$. 
\end{cor}

\begin{cor}
  \label{cor:base-induction}
  For $n\geq 2$, there exists a word basis $(\hg_{w_g})_{g\in S_n}$ of $\Hecke_n$ such that each word $w_g$ contains at most one $\tau_{n-1}$.
\end{cor}

\subsection{MOY graphs}
\label{sec:moy-graphs}

In what follows, we will be interested in closure (like braid closure) of the diagrams used to represent elements of the Hecke algebra. Algebraically, this means that we consider $\Hecke_n/ [\Hecke_n, \Hecke_n]$. Let us introduce some vocabulary.

\begin{dfn}[\cite{MR1659228}]
  \label{dfn:abstractMOYgraph}
  An \emph{abstract MOY graph} is a finite oriented graph $\Gamma = (V(\Gamma),E(\Gamma))$ with a labeling of its edges : $\ell : E(\Gamma) \to \NN_{>0}$ such that:
  \begin{itemize}
  \item The vertices are either univalent (we call this subset of vertices the \emph{boundary} of $\Gamma$ and denote it by $\partial\Gamma$) or trivalent (these are the \emph{internal vertices}).
  \item The flow given by labels and orientations is conservative 
around the trivalent vertices, meaning that every trivalent vertex follows one of the two following local models (\emph{merge} and \emph{split} vertices):
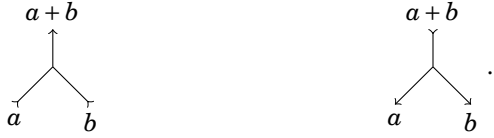
\[
\NB{\tikz{\input{\imagesfolder/sym_MOYvertex}}}.
\]
\end{itemize}
The univalent vertices are {\sl per se} either sinks or sources. We call the first \emph{positive boundary points} and the later \emph{negative boundary points}.

A \emph{MOY graph} is an abstract MOY graph together with an embedding in $\RR^2$.

An (abstract) MOY graph $\Gamma$ such that the $\ell(E(\Gamma)) \subseteq \{1,2\}$ is said to be \emph{elementary}.

An (abstract) MOY graph $\Gamma$ without univalent vertices is said to be \emph{closed}.
\end{dfn}
\begin{rmk}
  \label{rmk:edgebigsmall}
 Each internal vertex has three adjacent edges. One of this edges has a label which is strictly greater than the two others. We will say that it is the \emph{big} edge relatively to this vertex. The two other edges are called the \emph{small} edges relatively to this vertex.
\end{rmk}

\begin{dfn}
  \label{dfn:MOY}
  A \emph{braid-like MOY graph} is the image of an abstract MOY graph $\Gamma$  by a smooth embedding
in $[0,1] \times [0,1]$ such that:
  \begin{itemize}
  \item All the oriented tangent lines at vertices agree\footnote{In pictures which follow we may forget about this technical condition, since it is clear that we can always deform the embedding  locally so that this condition is fulfilled.}.
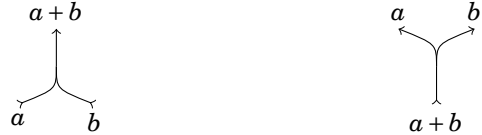
\[
\tikz{\input{\imagesfolder/alex_MOYvertexsmooth}}
\]
  \item For every point $p$ of $\Gamma$, the scalar product of the unit tangent vector at $p$ (whose orientation is given by the orientation of $\Gamma$) with $\left(\begin{smallmatrix}
      0 \\1
    \end{smallmatrix} \right)$
is strictly positive.
  \item The boundary $\partial \Gamma$ of $\Gamma$ is contained in $[0,1]\times \{0,1\}$ and $(0,1)\times \{0,1\} \cap \Gamma = \partial \Gamma$.
  \item The tangent lines at the boundary points of $\Gamma$ are vertical, (that is colinear with $ \left(\begin{smallmatrix}
      0 \\1
    \end{smallmatrix} \right)$). 
  \end{itemize}
\end{dfn}

\begin{notation}
  \label{not:listk}
In what follows, $\listk{k}$ always denotes a finite sequence  of integers (possibly empty). If $\listk{k}= (k_1, \dots, k_l)$, we say that $l$ is the \emph{length} of $\listk{k}$. For sequence of length 1, we abuse notation and write $k_1$ instead of $\listk{k}$. 
\end{notation}

\begin{rmk}
  \label{rmk:MOYgraphisopy}
  Braid-like MOY graphs are regarded up to ambient isotopy fixing the boundary. This fits into a category where objects are finite sequences of signed and labeled points in $]0,1[$, and morphisms are MOY graphs. The composition is given by concatenation and rescaling (see Figure~\ref{fig:exMOY}). 
\end{rmk}
\begin{figure}
  \centering
  \begin{tikzpicture}[xscale =0.7, yscale=0.8]
   {\tiny \input{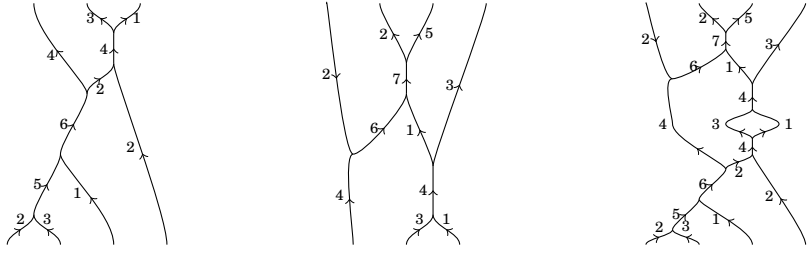} }
  \end{tikzpicture}
  \caption{Examples of a MOY graphs: a $((4,3,1),(2,3,1,2))$-MOY graph, a $((-2,2,5,3),(4,3,1))$-MOY graph and their concatenation.}
  \label{fig:exMOY}
\end{figure}
\subsection{Vinyl graphs}
\label{sec:vinyl-graphs}

We denote by $\ann$ the annulus $\{ x\in \RR^2 | 1 < |x| < 2\}$ and for all $x= \left(\begin{smallmatrix}   x_1 \\x _2 \end{smallmatrix}\right) $ in $\ann$, we denote by $t_x$ the vector $\left(\begin{smallmatrix} -x_2 \\x_1 \end{smallmatrix}\right)$. A \emph{ray} is a half-line which starts at $O$, the origin of $\RR^2$.

\begin{dfn}\label{dfn:vynil-graph}
  A \emph{vinyl graph} is the image of an abstract closed MOY graph $\Gamma$ in $\ann$ by a smooth\footnote{The smoothness condition is the same as the one of Definition~\ref{dfn:MOY}.} embedding such that for every point $x$ in the image $\Gamma$, the tangent vector at this point has a positive scalar product with $t_x$. 
  The set of vinyl graphs is denoted by $\Vin$.
The \emph{level} of a vinyl graph is the sum of the labels of the edges which intersect a ray $r$ (the flow condition ensures that it is well-defined).
If $k$ is a non-negative integer, we denote by $\Vin_k$ the set of vinyl graphs of level $k$. Vinyl graphs are regarded up to ambient isotopy in $\ann$.
\end{dfn}
\begin{figure}[ht]
  \centering
\NB{\tiny\tikz[scale=0.6]{\input{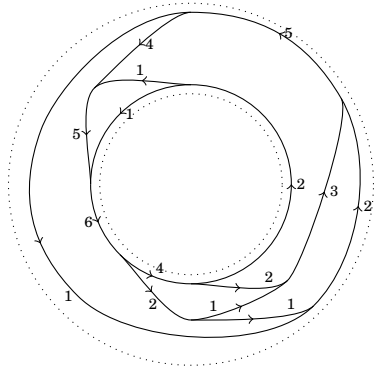}}}
  \caption{A vinyl graph of level $7$.}
  \label{fig:justforfun}
\end{figure}
\begin{rmk}
  \label{rmk:rotational}
  Let $\Gamma$ be a level $k$ vinyl graph and $D$ be a ray which does not contain any vertices of $\Gamma$. Then the condition on the tangent vectors of $\Gamma$ implies that 
    the intersection points of the ray $D$ with $\Gamma$ are all transverse and positive.
Informally, the level counts the numbers of tracks of a vinyl graph.
\end{rmk}
A natural way to obtain vinyl graphs is by closing  braid-like MOY graphs. This is illustrated in Figure~\ref{fig:closebraidlike}.

\begin{notation}
  \label{not:braidlikeclosure}
  Let $\listk{k}$ be a finite sequence of positively oriented labeled points and $\Gamma$ be a braid-like $(\listk{k},\listk{k})$-MOY graph. Then we denote by $\widehat{\Gamma}$ the vinyl graph obtained by closing up $\Gamma$. The level of $\Gamma$ is equal to the sum of the labels of the points of $\listk{k}$.
\end{notation}

\begin{figure}[ht]
  \centering 
  \begin{tikzpicture}[scale=1]
    \input{\imagesfolder/sym_closebraidlike}
  \end{tikzpicture}
  \caption{The vinyl graph $\widehat{\Gamma}$  is obtained by closing up the braid-like MOY graph $\Gamma$.}
  \label{fig:closebraidlike}
\end{figure}

\subsection{Skein modules}
\label{sec:skein-modules}

In this subsection, we introduce three skein modules. We first set up conventions for quantum integers and binomials. For a non-negative integer $n$, we set \[[n]:=\frac{q^n - q^{-n}}{q - q^{-1}} \quad \text{and} \quad [n]! := \prod_{i=1}^n[i].\] If $a$ is a non-negative integer, we define:
\[
  \begin{bmatrix}
    n \\ a
  \end{bmatrix} =
  \begin{cases}
    0 &\text{if $a>n$,}\\
    \frac{[n]!}{[a]![n-a]!} & \text{otherwise}.
  \end{cases}
\]

\begin{dfn}
  \label{dfn:skein-full}
  We denote by $\skeinf$ the $\ZZ[q,q^{-1}]$-module generated by isotopy classes of vinyl graphs subjected to the following relations:
\begin{align} \label{eq:extrelass}
   \kup{\stgamma} = \kup{\stgammaprime},
 \end{align}
\begin{align} \label{eq:extrelass2}
   \kup{\stgammar} = \kup{\stgammaprimer},
 \end{align}
 \begin{align} \label{eq:extrelbin1} 
\kup{\digona} = \arraycolsep=2.5pt
  \begin{bmatrix}
    m+n \\ m
  \end{bmatrix}
\kup{\verta},
\end{align}
\begin{align}
  \kup{\!\squarec\!}=\!\! \sum_{j=\max{(0, m-n)}}^m\!\begin{bmatrix}l \\ k-j \end{bmatrix}
 \kup{\!\squared\!},\label{eq:extrelsquare3}
\end{align}
\begin{align}
  \kup{\!\squarecc\!}=\!\! \sum_{j=\max{(0, m-n)}}^m\!\begin{bmatrix}l \\ k-j \end{bmatrix}
 \kup{\!\squaredd\!}.\label{eq:extrelsquare4}
\end{align}
These relations holds when embedded in the interior of $\ann$. 
\end{dfn}

The local relations given in Definition~\ref{dfn:skein-full} are often called \emph{braid-like MOY relations} and the combinatorics they set up is called \emph{braid-like MOY calculus}.

\begin{lem}\label{lem:Hecke-satisfied}
  In $\skeinf$, the following relation holds:
  \[
\NB{\tikz[scale = 0.5]{\input{\imagesfolder/alex_Heckerelation}}},
\]
where the unadorned edges are labeled $1$. 
\end{lem}

\begin{proof}
  We use relations~(\ref{eq:extrelsquare3}) and (\ref{eq:extrelsquare4})  when $n=m=l=k=1$. This gives:
  \begin{align*}
    &    \NB{
    \tikz[scale = 0.5]{
    \begin{scope}
  \coordinate (B1) at (-1,0);
  \coordinate (B2) at (1,0);
  \coordinate (C1) at (-1,1.1);
  \coordinate (D1) at (-1,1.9);
  \coordinate (C2) at (1,0.9);
  \coordinate (D2) at (1,2.1);
  \coordinate (T1) at (-1,3);
  \coordinate (T2) at (1,3);
  \draw[->] (B1) -- (C1) node[at start, below] {{$\scriptstyle{1}$}};
  \draw[->] (C1) -- (D1) node[midway, left   ] {{$\scriptstyle{2}$}};
  \draw[->] (D1) -- (T1) node[at end , above ] {{$\scriptstyle{1}$}};
  \draw[->] (B2) -- (C2) node[at start, below] {{$\scriptstyle{2}$}};
  \draw[->] (C2) -- (D2) node[midway, right] {{$\scriptstyle{1}$}};
  \draw[->] (D2) -- (T2) node[at end, above] {{$\scriptstyle{2}$}};
  \draw[->] (D1) -- (D2) node[midway, above] {{$\scriptstyle{1}$}};
  \draw[->] (C2) -- (C1) node[midway, below] {{$\scriptstyle{1}$}};
\end{scope}
\node at (2.5,1.5) {$=$};
\node at (7.5,1.5) {$+$};
\begin{scope}[xshift = 5cm]
  \coordinate (B1) at (-1,0);
  \coordinate (B2) at (1,0);
  \coordinate (C1) at (-1,1.1);
  \coordinate (D1) at (-1,1.9);
  \coordinate (C2) at (1,0.9);
  \coordinate (D2) at (1,2.1);
  \coordinate (T1) at (-1,3);
  \coordinate (T2) at (1,3);
  \draw[->] (B1) -- (T1) node[midway, left   ] {{$\scriptstyle{1}$}};
  \draw[->] (B2) -- (T2) node[midway, right] {{$\scriptstyle{2}$}};
\end{scope}
\begin{scope}[xshift = 10cm]
\coordinate (B1) at (-1,0);
\coordinate (B2) at (1,0);
\coordinate (C1) at (-1,0.9);
\coordinate (D1) at (-1,2.1);
\coordinate (C2) at (1,1.1);
\coordinate (D2) at (1,1.9);
\coordinate (T1) at (-1,3);
\coordinate (T2) at (1,3);
\draw[->-] (B1) .. controls (C1) .. (C2) node[at start, below] {{$\scriptstyle{1}$}};
\draw[->] (D2) .. controls (D1) .. (T1) node[at end , above ] {{$\scriptstyle{1}$}};
\draw[->] (B2) -- (C2) node[at start, below] {{$\scriptstyle{2}$}};
\draw[->] (C2) -- (D2) node[midway, right] {{$\scriptstyle{3}$}};
\draw[->] (D2) -- (T2) node[at end, above] {{$\scriptstyle{2}$}};
\end{scope}
    }
    }
    \qquad \textrm{and} \\
    &    \NB{
    \tikz[scale = 0.5]{
    \begin{scope}
  \coordinate (B1) at (-1,0);
  \coordinate (B2) at (1,0);
  \coordinate (C1) at (-1,0.9);
  \coordinate (D1) at (-1,2.1);
  \coordinate (C2) at (1,1.1);
  \coordinate (D2) at (1,1.9);
  \coordinate (T1) at (-1,3);
  \coordinate (T2) at (1,3);
  \draw[->] (B1) -- (C1) node[at start, below] {{$\scriptstyle{2}$}};
  \draw[->] (C1) -- (D1) node[midway, left   ] {{$\scriptstyle{1}$}};
  \draw[->] (D1) -- (T1) node[at end , above ] {{$\scriptstyle{2}$}};
  \draw[->] (B2) -- (C2) node[at start, below] {{$\scriptstyle{1}$}};
  \draw[->] (C2) -- (D2) node[midway, right] {{$\scriptstyle{2}$}};
  \draw[->] (D2) -- (T2) node[at end, above] {{$\scriptstyle{1}$}};
  \draw[<-] (D1) -- (D2) node[midway, above] {{$\scriptstyle{2}$}};
  \draw[<-] (C2) -- (C1) node[midway, below] {{$\scriptstyle{2}$}};
\end{scope}
\node at (2.5,1.5) {$=$};
\node at (7.5,1.5) {$+$};
\begin{scope}[xshift = 5cm]
  \coordinate (B1) at (-1,0);
  \coordinate (B2) at (1,0);
  \coordinate (C1) at (-1,1.1);
  \coordinate (D1) at (-1,1.9);
  \coordinate (C2) at (1,0.9);
  \coordinate (D2) at (1,2.1);
  \coordinate (T1) at (-1,3);
  \coordinate (T2) at (1,3);
  \draw[->] (B1) -- (T1) node[midway, left   ] {{$\scriptstyle{2}$}};
  \draw[->] (B2) -- (T2) node[midway, right] {{$\scriptstyle{1}$}};
\end{scope}
\begin{scope}[xshift = 10cm]
  \coordinate (B1) at (-1,0);
  \coordinate (B2) at (1,0);
  \coordinate (C1) at (-1,1.1);
  \coordinate (D1) at (-1,1.9);
  \coordinate (C2) at (1,0.9);
  \coordinate (D2) at (1,2.1);
  \coordinate (T1) at (-1,3);
  \coordinate (T2) at (1,3);
  \draw[->] (B1) -- (C1) node[at start, below] {{$\scriptstyle{1}$}};
  \draw[->] (C1) -- (D1) node[midway, left   ] {{$\scriptstyle{3}$}};
  \draw[->] (D1) -- (T1) node[at end , above ] {{$\scriptstyle{1}$}};
  \draw[->] (B2) .. controls (C2) .. (C1) node[at start, below] {{$\scriptstyle{2}$}};
  \draw[->] (D1) .. controls (D2) .. (T2) node[at end, above] {{$\scriptstyle{2}$}};
\end{scope}.
    }
    }
  \end{align*}
We conclude by using relations~(\ref{eq:extrelass}) and (\ref{eq:extrelass2}) in the case $i=j=k=1$.
\end{proof}

\begin{dfn}\label{dfn:skein}
  We denote by $\skein$ the $\ZZ[q, q^{-1}]$ module generated by elementary vinyl graphs (see Definition~\ref{dfn:abstractMOYgraph}) and modded out by relations~(\ref{eq:Hecke-diag}) where thin edges are replaced by upward oriented edges labeled by $1$ and thick edges are replaced by upward oriented edges labeled by $2$.
\end{dfn}

From Lemma~\ref{lem:Hecke-satisfied}, we deduce the following Proposition.

\begin{prop}\label{prop:skein2skeinf}
  There is a unique well-defined map $\iota: \skein \to \skeinf$ which sends an elementary vinyl graph to itself.
\end{prop}

From Turaev \cite{TSkein} and Kauffman--Vogel \cite{KV}, we obtain:

\begin{prop}\label{prop:iotainjective}
  The map $\iota: \skein \to \skeinf$ is injective.
\end{prop}

The next result follows Turaev  \cite{TSkein} combined with the MOY calculus, and also from Queffelec--Rose \cite{QRAnnular}:

\begin{prop}[{\cite[Lemma 5.2]{QRAnnular}}]\label{prop:QR}
  Families of nested circles up to permutation of their labels form a $\ZZ[q, q^{-1}]$-basis of $\skeinf$.
\end{prop}

We need to introduce an Alexander analogue to the skein module $\skein$. 

\begin{dfn}
  \label{dfn:markedVG}
  A \emph{marked vinyl graph} is a vinyl graph with a base point located in the interior of an edge of label $1$. Just like vinyl graphs, they are regarded up to isotopy. Given a base point $X$, consider the ray $[OX)$ passing through it. If $d$ is the sum of the labels of the edges intersecting $[OX)$ outside $[OX]$, we say  that $\Gamma$ has \emph{depth} $d$. 
\end{dfn}
In diagrams the base point is represented by a red star (${\color{red}\star}$). We will mostly consider marked vinyl graph of depth $0$ or $1$.  

\begin{dfn}\label{dfn:skeinp}
  We denote by $\skeinp$ the $\ZZ[q, q^{-1}]$ module generated by marked elementary vinyl graphs of depth $0$ and modded out by relations~(\ref{eq:Hecke-diag}) which do not involve the base point.
\end{dfn}

There is a unique well defined $\ZZ[q, q^{-1}]$-linear map $\skeinp \to \skein$ which consists in forgetting the base point.

\begin{prop}
  \label{prop:skeinrelation-to-unicity}
  There exists at most one $\ZZ[q,q^{-1}]$-linear map $\phi$ (\resp$\phi^*$) from $\skein$ (\resp$\skeinp$) to the ring $\QQ(q)$, satisfying the following:
  \begin{itemize}
  \item If a vinyl graph $\Gamma$ (\resp{}marked vinyl graph $\Gamma^*$) is not connected, then $\phi(\Gamma) =0$ (\resp$\phi^*(\Gamma^*)=0$).
  \item If $\Gamma$ is a circle ($\Gamma^*$ is a marked circle) of label $1$, then $\phi(\Gamma) =1$ (\resp$\phi^*(\Gamma^*)=1$).
  \item If $\Gamma$ is the empty vinyl graph, then $\phi(\Gamma)=1$. 
  \item In the $\skein$ case, the following relation holds:    
\begin{align}\label{eq:curl}
  \phi\left(  \NB{\tikz{\input{\imagesfolder/alex_baddigon}}} \right) =
  \phi\left(
\NB{\tikz{\begin{scope}[scale=0.35]
  \draw (-0.1, 1) -- (3.1, 1) -- (3.1, -1) -- (-0.1, -1) -- cycle;
  \node[scale = 0.5] at (  2, 1.2) {$\dots$};
  \node[scale = 0.5] at (  1, 1.2) {$\dots$};
  \node[scale = 0.5] at (  2,-1.2) {$\dots$};
  \node[scale = 0.5] at (  1,-1.2) {$\dots$};
  \draw (3, -2) -- +(0,1);
  \draw (  3,1) arc (0:180:3.5) -- +(0, -3) arc (-180:0:3.5);
  \draw (2.5,1) arc (0:180:3  ) -- +(0, -3) arc (-180:0:  3) -- +(0,1);
  \draw (1.5,1) arc (0:180:2) -- +(0, -3) arc (-180:0:2) -- +(0,1);
  \draw (0.5,1) arc (0:180:  1) -- +(0, -3) arc (-180:0:1  ) -- +(0,1);
  \draw (  0,1) arc (0:180:0.5) -- +(0, -3) arc (-180:0:0.5) -- +(0,1);

\end{scope}}}
 \right).
\end{align}
\item In the $\skeinp$ case, the following relation holds:
\begin{align}\label{eq:curl-m}
  \phi^*\left(  \NB{\tikz{\input{\imagesfolder/alex_baddigon-m}}} \right) =
  \phi^*\left(
\NB{\tikz{\input{\imagesfolder/alex_baddigon0-m}}}
 \right).
\end{align}
  \end{itemize}
\end{prop}

The vinyl graphs appearing in (\ref{eq:curl}) and in (\ref{eq:curl-m}) are said to be \emph{related by an outer digon}. If a vinyl graph has the form on the left hand-side of (\ref{eq:curl}) or (\ref{eq:curl-m}), we say that it \emph{contains a curl}.

\begin{cor}\label{cor:phi-phistar}
  For any marked vinyl graph $\Gamma^*$ of depth $0$, denote by $\Gamma$ the vinyl graph obtained from $\Gamma^*$ by forgetting the base point.
  Suppose $\phi\colon \skein \to \QQ(q)$ and $\phi^\star\colon\skeinp \to \QQ(q)$ are two functions satisfying the conditions of Proposition~\ref{prop:skeinrelation-to-unicity},  then $\phi(\Gamma)=\phi^*(\Gamma^*)$. 
\end{cor}

\begin{rmk}\label{rmk:existence-phi}
  The proof can be modified in order to prove existence of such a function $\phi$. We will see in Section~\ref{sec:vinylgraphBP} such a function arising naturally.
\end{rmk}

\begin{proof}[Proof of Proposition~\ref{prop:skeinrelation-to-unicity}]
  Using the MOY relations \ref{eq:extrelass}, \ref{eq:extrelass2}, \ref{eq:extrelbin1}, \ref{eq:extrelsquare3} and \ref{eq:extrelsquare4}, we can express any vinyl graph as a $\QQ(q)$-linear combination of elementary vinyl graphs (see \cite{pre06302580}). Since $\phi$ intertwines the MOY relations, it is enough to prove the uniqueness of $\phi$ for elementary vinyl graphs. In the $\skeinp$ case, since the base point is on an edge with label $1$, the argument still works.
  Since such a function $\phi$ intertwines the Hecke relations, its value on the set of vinyl graphs of level $k$ is entirely determined on a collection of vinyl graphs $\Gamma_i$ which consist of closures of braid-like MOY graphs representing a basis of the Hecke algebra $\Hecke_k$. By Corollary~\ref{cor:base-induction}, we can choose a basis such that the $\Gamma_i$ are either not connected or contains a curl. Hence the values of $\phi$ on vinyl graphs of level $k$ are entirely determined by the values of $\phi$ on vinyl graphs of level $k-1$. We conclude by induction.  
\end{proof}

\begin{rmk} \label{rmk:skeinrelation-to-unicity-q=1}
  The content of Subsection~\ref{sec:skein-modules} remains true after evaluating $q$ to $1$. In the statements, the quantum integers and binomials should be replaced by classical integers and binomials and the spaces $\ZZ[q, q^{-1}]$ and $\QQ(q)$ should be replaced by $\ZZ$ and $\QQ$ respectively. 
\end{rmk}

\subsection{Braid closures with base point}
\label{sec:base-points}

\begin{dfn}\label{def:braid-closure-bspoint}
  A \emph{braid closure diagram with base point} is a diagram $\beta$ of a braid closure in $\ann$ with a choice of a point (depicted by a red star {\color{red}$\star$} in figures) away from crossings. 
\end{dfn}

Since the base point has no meaning so far, there is an Alexander theorem for these diagrams.

\begin{prop}\label{prop:alexander-bspoint}
  Every link can be represented by a braid diagram with base point.
\end{prop}

We now give a Markov-like theorem for braids with base point.

\begin{prop}\label{prop:markov-base}
  Let $\beta_1$ and $\beta_2$ be two braid diagrams with base point. Suppose that they represent the same knot, then they are related by a finite sequence of moves:
  \begin{enumerate}
  \item \label{it:mv-iso}Isotopy of the diagrams respecting the braid closure and the base point  conditions.
  \item \label{it:mv-braid}Braid relations disjoint from the base point:
    \begin{align*}
      \NB
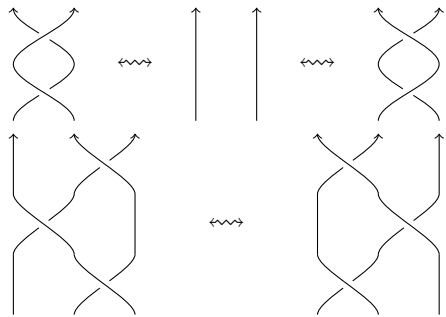
{\tikz[xscale = 0.4,
      yscale = 0.5 ]{\begin{scope}
  \begin{scope}
    \draw[->] (-1,-1.5) .. controls +(0 ,0.5) and +(0,-0.5 ) .. (1,0) .. controls +(0,0.5) and +(0 ,-0.5) ..  (-1,+1.5); 
    \fill[white] (0,0.7) circle (2mm);
    \fill[white] (0,-0.7) circle (2mm);
    \draw[->] (+1, -1.5) .. controls +(0 ,0.5) and +(0,-0.5 ) .. (-1,0) .. controls +(0,0.5) and +(0 ,-0.5) ..  (+1,+1.5);
  \end{scope}
  \begin{scope}[xshift = 12cm]
    \draw[->] (+1, -1.5) .. controls +(0 ,0.5) and +(0,-0.5 ) .. (-1,0) .. controls +(0,0.5) and +(0 ,-0.5) ..  (+1,+1.5);
    \fill[white] (0,0.7) circle (2mm);
    \fill[white] (0,-0.7) circle (2mm);
    \draw[->] (-1,-1.5) .. controls +(0 ,0.5) and +(0,-0.5 ) .. (1,0) .. controls +(0,0.5) and +(0 ,-0.5) ..  (-1,+1.5); 
  \end{scope}
  \begin{scope}[xshift = 6cm]
    \draw[->] (-1, -1.5) -- (-1, 1.5);
    \draw[->] (+1, -1.5) -- ( +1, 1.5);
  \end{scope}
  \node at (3, 0) {$\leftrightsquigarrow$};
  \node at (9, 0) {$\leftrightsquigarrow$};
\end{scope}}} \\
      \NB{\tikz[xscale = 0.4,
      yscale = 0.4 ]{\input{\imagesfolder/alex_R3move}}}
    \end{align*}
  \item \label{it:mv-bs-point}Moving the base point through a crossing:
       \begin{align*}
      \NB
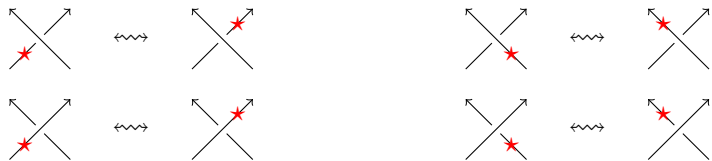
{\tikz[xscale = 0.4,
      yscale = 0.4 ]{\begin{scope}
\begin{scope}
  \begin{scope}
    \draw[->] (-1,-1) -- (+1,+1) node[pos = 0.25, red] {${\star}$}; 
    \fill[white] (0,0) circle (2mm);
    \draw[->] (+1,-1) -- (-1,+1);
  \end{scope}
  \node at (3, 0) {$\leftrightsquigarrow$};
  \begin{scope}[xshift = 6cm]
    \draw[->] (-1,-1) -- (+1,+1) node[pos = 0.75, red] {${\star}$}; 
    \fill[white] (0,0) circle (2mm);
    \draw[->] (+1,-1) -- (-1,+1);
  \end{scope}
\end{scope}

\begin{scope}[yshift = -3cm]
  \begin{scope}
    \draw[->] (+1,-1) -- (-1,+1);
    \fill[white] (0,0) circle (2mm);
    \draw[->] (-1,-1) -- (+1,+1) node[pos = 0.25, red] {${\star}$}; 
  \end{scope}
  \node at (3, 0) {$\leftrightsquigarrow$};
  \begin{scope}[xshift = 6cm]
    \draw[->] (+1,-1) -- (-1,+1);
    \fill[white] (0,0) circle (2mm);
    \draw[->] (-1,-1) -- (+1,+1) node[pos = 0.75, red] {${\star}$}; 
  \end{scope}
\end{scope}
\begin{scope}[xshift = 15cm]
  \begin{scope}
    \draw[->] (-1,-1) -- (+1,+1);
    \fill[white] (0,0) circle (2mm);
    \draw[->] (+1,-1) -- (-1,+1) node[pos = 0.25, red] {${\star}$}; 
  \end{scope}
  \node at (3, 0) {$\leftrightsquigarrow$};
  \begin{scope}[xshift = 6cm]
    \draw[->] (-1,-1) -- (+1,+1);
    \fill[white] (0,0) circle (2mm);
    \draw[->] (+1,-1) -- (-1,+1) node[pos = 0.75, red] {${\star}$}; 
  \end{scope}
\end{scope}

\begin{scope}[yshift = -3cm, xshift = 15cm]
  \begin{scope}
    \draw[->] (+1,-1) -- (-1,+1) node[pos = 0.25, red] {${\star}$}; 
    \fill[white] (0,0) circle (2mm);
    \draw[->] (-1,-1) -- (+1,+1) ;
  \end{scope}
  \node at (3, 0) {$\leftrightsquigarrow$};
  \begin{scope}[xshift = 6cm]
    \draw[->] (+1,-1) -- (-1,+1) node[pos = 0.75, red] {${\star}$}; 
    \fill[white] (0,0) circle (2mm);
    \draw[->] (-1,-1) -- (+1,+1) ;
  \end{scope}
\end{scope}

\end{scope}}} 
    \end{align*}
  \item \label{it:mv-stab}Reidemeister $1$ (aka stabilization or 2nd Markov move) with the base point on the curl:
    \begin{align*}
      \NB
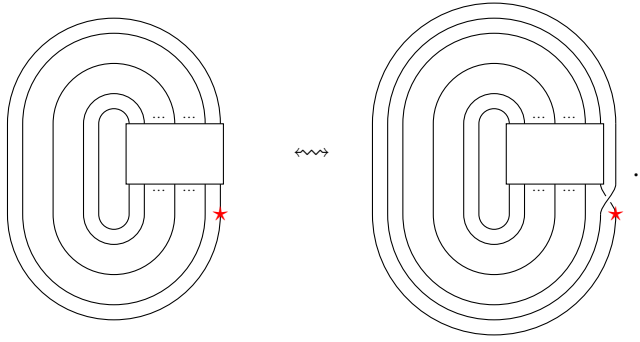
{\tikz{\input{\imagesfolder/alex_R1-move}}}.
    \end{align*}
  \end{enumerate}
\end{prop}

\begin{proof}[Sketch of the proof]
  Thanks to move (\ref{it:mv-bs-point}), one can actually use braid relations and stabilization regardless of where the base point is. 
  We conclude by using the classical Markov theorem. Since the stabilization requires the base point to be on the curl, it is important for braids to represent knots. 
\end{proof}

\begin{dfn}
  Let $\beta$ be a braid closure diagram with a base point. We denote by $\overrightarrow{\beta}$ the braid closure diagram with base point obtained from $\beta$ by pulling the strand of the base point over every strands on its right in order to move it to the right-most position (see Figure~\ref{fig:rectification}). We call  $\overrightarrow{\beta}$ the \emph{rectified version of $\beta$}.
  \begin{figure}[ht]
    \centering
      {\tikz[xscale = 1,
      yscale = 1 ]{\input{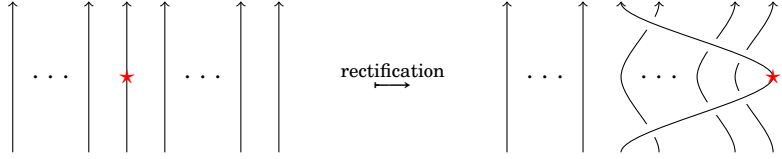}}} 
    \caption{Rectification of braid closure diagrams with base point.}
    \label{fig:rectification}
  \end{figure}
If $\beta = \overrightarrow{\beta}$, we say that $\beta$ is a \emph{braid closure diagram with right-most base point.}
\end{dfn}

\begin{lem}\label{lem:rectified-markov}
  Suppose, we have a map $H$ from the set of braid closure diagrams with right base points to the set of bigraded $\QQ$-vector spaces such that $H(\beta_1) \simeq H(\beta_2)$ whenever: 
  
  \begin{enumerate}
  \item\label{it:mv-2-iso} The diagrams $\beta_1$ and $\beta_2$ are related by a planar isotopy.
  \item\label{it:mv-2-braid} The diagrams $\beta_1$ and $\beta_2$ are related by braid relation far from the base point.
  \item\label{it:mv-2-stab} The diagrams $\beta_1$ and $\beta_2$ are related by stabilization as stated in the move (\ref{it:mv-stab}) of Proposition~\ref{prop:markov-base}.
  \item\label{it:mv-2-R2} The diagrams $\beta_1$ and $\beta_2$ are related by the following move:
        \begin{align*}
      \NB
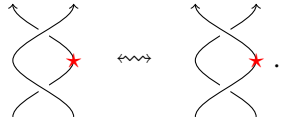
{\tikz[xscale = 0.4,
      yscale = 0.5 ]{\begin{scope}
  \begin{scope}
    \draw[->] (-1,-1.5) .. controls +(0 ,0.5) and +(0,-0.5 ) .. (1,0) .. controls +(0,0.5) and +(0 ,-0.5) ..  (-1,+1.5) 
node[pos = 0, red] {${\star}$}; 
    \fill[white] (0,0.7) circle (2mm);
    \fill[white] (0,-0.7) circle (2mm);
    \draw[->] (+1, -1.5) .. controls +(0 ,0.5) and +(0,-0.5 ) .. (-1,0) .. controls +(0,0.5) and +(0 ,-0.5) ..  (+1,+1.5);
  \end{scope}
  \begin{scope}[xshift = 6cm]
    \draw[->] (+1, -1.5) .. controls +(0 ,0.5) and +(0,-0.5 ) .. (-1,0) .. controls +(0,0.5) and +(0 ,-0.5) ..  (+1,+1.5);
    \fill[white] (0,0.7) circle (2mm);
    \fill[white] (0,-0.7) circle (2mm);
    \draw[->] (-1,-1.5) .. controls +(0 ,0.5) and +(0,-0.5 ) .. (1,0) .. controls +(0,0.5) and +(0 ,-0.5) ..  (-1,+1.5) 
node[pos = 0, red] {${\star}$}; 
  \end{scope}
  \node at (3, 0) {$\leftrightsquigarrow$};
\end{scope}}}. 
     \end{align*}
  \end{enumerate}
  Then, we can extend $H$ to the set of braid closure diagrams with base points by defining $H(\beta):=H\left(\overrightarrow{\beta}\right)$. With this extended definition $H$ is a knot invariant.
\end{lem}

\begin{proof}
  It is enough to check that $H\left(\overrightarrow{\bullet}\right)$ respects the moves of Proposition~\ref{prop:markov-base}. The moves (\ref{it:mv-iso}), (\ref{it:mv-braid}) and (\ref{it:mv-stab}) follow directly from the hypothesis. It remains to show that $H\left(\overrightarrow{\bullet}\right)$ respects move~(\ref{it:mv-bs-point}) of Proposition~\ref{prop:markov-base}. 
  Let us consider the following case:
        \begin{align*}
      \NB{\tikz[scale = 0.4]{\begin{scope}
  \begin{scope}
    \node at (-2, 0) {$\beta_1 :=$};
    \draw[->] (+1,-1) -- (-1,+1) node[pos = 0.25, red] {${\star}$};  
    \fill[white] (0,0) circle (2mm);
    \draw[->] (-1,-1) -- (+1,+1);
  \end{scope}
  \begin{scope}[xshift = 6cm]
   \node at (-2, 0) {$\beta_2 :=$};
    \draw[->] (+1,-1) -- (-1,+1) node[pos = 0.75, red] {${\star}$};    
    \fill[white] (0,0) circle (2mm);
    \draw[->] (-1,-1) -- (+1,+1);
  \end{scope}
\end{scope}
}}. 
     \end{align*}
  The proof follows from the following sequence of moves on diagrams:
     \begin{align*}
      \NB
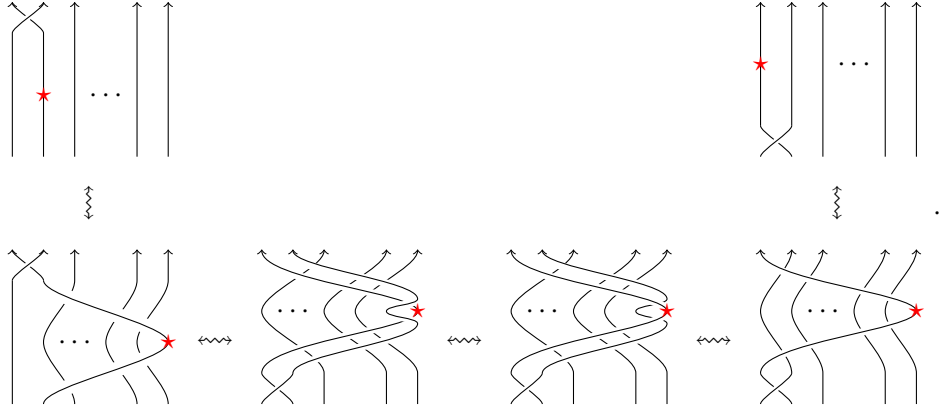
{\tikz[scale = 0.82]{\input{\imagesfolder/alex_rectification-R2}}}. 
     \end{align*}
  The other cases are similar. 
\end{proof}

\section{A reminder of the  symmetric $\gll_1$-homology}
\label{sec:symgl1}
In this section we recall the definition of the symmetric $\gll_1$-homology as described in \cite{RW2} in a non-equivariant setting. We first introduce the notion of foams and vinyl foams. At the end of this section we give an alternative definition which is foam-free.

\subsection{Foams}
\label{sec:foams}

\begin{dfn}\label{dfn:foam}
  A \emph{$\gll_N$-foam} (or simply \emph{foam}) $F$ is a collection 
  of \emph{facets} $\mathcal{F}(F)=(\Sigma_i)_{i\in I}$, that is a finite collection of
  oriented connected surfaces with boundary, together with the
  following data:
  \begin{itemize}
  \item A \emph{labeling} $\ell\co (\Sigma_i)_{i\in I} \to \{0, \dots, N\}$.
  \item A set of \emph{decorations}, that is, for each facet $f$, a symmetric polynomial $P_f$ in $\ell(f)$ variables with rational coefficients.  For a given facet $f$, we say that $f$ is \emph{trivially decorated} if $P_f=1$. 
  \item A ``gluing recipe'' of the facets along their boundaries such
    that when glued together using the recipe we have the three
    possible local models: the neighborhood of a point is homeomorphic to a surface, a tripod times an interval or the cone of the 1-skeleton of the tetrahedron as depicted below.
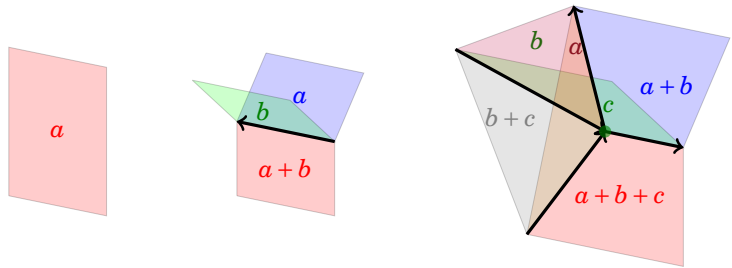
\[
      \begin{tikzpicture}
        \begin{scope}
\tdplotsetmaincoords{80}{140}
  \begin{scope}[tdplot_main_coords]
    \filldraw [very thin, fill=red, opacity = 0.2] (0,-1, -1) -- (0,1,-1) -- (0,1,1) -- (0,-1,1) -- (0,-1,-1);
    \node[sloped, red] at (0,0,0) {$a$};
  \end{scope}
\begin{scope}[xshift = 3cm, tdplot_main_coords]
  \begin{scope}
    \filldraw [very thin, fill =red, opacity = 0.2] (0,-1, -1) -- (0,1,-1) -- (0,1,0) -- (0,-1,0) -- (0,-1,-1);
        \node[sloped, red] at (0,0,-0.5) {$a+b$};
    \end{scope}
  \begin{scope}[rotate around y = 150]
    \filldraw [very thin, fill =blue, opacity = 0.2] (0,-1, -1) -- (0,1,-1) -- (0,1,0) -- (0,-1,0) -- (0,-1,-1);
        \node[sloped, blue] at (0,0,-0.5) {$a$};
  \end{scope}
  \begin{scope}[rotate around y =-130]
    \filldraw [very thin, fill =green, opacity = 0.2] (0,-1, -1) -- (0,1,-1) -- (0,1,0) -- (0,-1,0) -- (0,-1,-1);
        \node[sloped, green!50!black] at (0,0,-0.5) {$b$};
        \draw[very thick, ->] (0,1,0) -- (0,-1,0);
  \end{scope}
  \end{scope}
\begin{scope}[scale = 1.6, xshift = 4.5cm, tdplot_main_coords]
  \begin{scope}
    \filldraw [very thin, fill =red, opacity = 0.2] (0,-1, -1) -- (0,1,-1) -- (0,1,0) -- (0,0,0) -- (0,-1,-1);
    \coordinate (a) at (0, -1, -1);
        \node[sloped, red] at (0,0.2,-0.5) {$a+b+c$};
    \end{scope}
  \begin{scope}[rotate around y = 150]
    \filldraw [very thin, fill= blue, opacity = 0.2] (0,-1, -1) -- (0,1,-1) -- (0,1,0) -- (0,0,0) -- (0,-1,-1);
    \coordinate (b) at (0, -1, -1);
        \node[sloped, blue] at (0,0.5,-0.5) {$a+b$};
  \end{scope}
  \begin{scope}[rotate around y =-130]
    \filldraw [very thin, fill= green, opacity = 0.2] (0,-1, -1) -- (0,1,-1) -- (0,1,0) -- (0,0,0) -- (0,-1,-1);
    \coordinate (c) at (0, -1, -1);
        \node[sloped, green!50!black] at (0,0.5,-0.5) {$c$};
        \node[sloped, purple!50!black] at (0,1,-1.5) {$a$};
        \node[sloped, green!50!black] at (0,+0.5,-1.5) {$b$};
        \node[sloped, gray] at (0,-1.3,0.1) {$b+c$};
  \end{scope}
  \filldraw[very thin, fill= orange, opacity = 0.2] (a) -- (b) -- (0,0,0) -- (a);
  \filldraw[very thin, fill= gray, opacity = 0.2] (a) -- (c) -- (0,0,0) -- (a);
  \filldraw[very thin, fill= purple, opacity = 0.2] (b) -- (c) -- (0,0,0) -- (b);
  \draw[very thick, ->] (a) -- (0,0,0);
  \draw[very thick, <-] (b) -- (0,0,0);
  \draw[very thick, ->] (c) -- (0,0,0);
  \draw[very thick, <-] (0,1, 0) -- (0,0,0);
  \fill[green!50!black, opacity = 0.6] (0,0,0) circle (0.5mm);
\end{scope}
\end{scope}

      \end{tikzpicture}
      \label{fig:FBSP} \]
    The letter appearing on a facet indicates the label of this facet.
    That is we have: \emph{facets}, \emph{bindings} (which are compact oriented
    $1$-manifolds) and \emph{singular vertices}. Each binding carries:
    \begin{itemize}
    \item An orientation which agrees with the orientations of the facets with
      labels $a$ and $b$ and disagrees with the orientation of the facet with label
      $a+b$. Such a binding has \emph{type $(a, b, a+b)$}.
    \item A cyclic ordering of the three facets around it. When a foam
      is embedded in $\RR^3$, we require this cyclic ordering to agree with 
      the left-hand rule\footnote{This agrees with Khovanov's convention used in~\cite{MR2100691}.} with respect to its
      orientation (the dotted circle in the middle indicates that the
      orientation of the binding points to the reader).
      
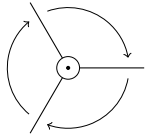
\[
        \begin{tikzpicture}[xscale=1]
          \begin{scope}
 \draw (0,0) -- +(0:1);
 \draw (0,0) -- +(120:1);
 \draw (0,0) -- +(240:1);
 \filldraw[fill= white, draw=black, very thin] (0,0) circle (0.15cm);
 \filldraw[fill = black] (0,0) circle (0.02cm);
 \draw[very thin,->] (-10:0.8cm) arc (-10:-110 :0.8); 
 \draw[very thin,->] (230:0.8cm) arc (230: 130:0.8); 
 \draw[very thin,->] (110:0.8cm) arc (110:10 :0.8); 
\end{scope}

        \end{tikzpicture}
      \]
    \end{itemize}
    The cyclic orderings of the different bindings adjacent to a
    singular vertex should be compatible. This means that a
    neighborhood of the singular vertex is embeddable in $\RR^3$ in a
    way that respects the left-hand rule for the four binding adjacent to this singular vertex.
  \end{itemize}
  In particular, when forgetting about the labels and the orientations, it has a structure of a compact, finite 2-dimensional CW-complex. 

Denote by $\mathcal{S}$ the collection of circles which are boundaries of the facets of $F$. The gluing recipe consists of:
  \begin{itemize}
  \item For a subset $\mathcal{S}'$ of $\mathcal{S}$, a subdivision of each circle of $\mathcal{S'}$ into a finite number of closed intervals. This gives us a collection $\mathcal{I}$ of closed intervals.  
  \item Partitions of $\mathcal{I} \cup (\mathcal{S} \setminus \mathcal{S'})$ into subsets of three elements. For every subset $(Y_1, Y_2, Y_3)$ of this partition, three diffeomorphisms $\phi_1 : Y_2 \to Y_3$, $\phi_2 : Y_3 \to Y_1$, $\phi_3 : Y_1 \to Y_2$  such that $\phi_3 \circ \phi_2 \circ \phi_1 = \mathrm{id}_{Y_2}$.
  \end{itemize}

A foam whose decorations are all equal to $1$ (as a symmetric polynomial) is said to be \emph{dry}. On facets of label $1$, decorations are polynomials in one variable. It is convenient to represent the monomial $X$ by a \emph{dot} ($\bullet$) and more generally the monomial $X^i$ by $a$ \emph{dot} with $i$ next to it ($\bullet^i$).
\end{dfn}

When embedded in $\RR^3$, the concept of foam extends naturally to the concept of \emph{foam with boundary}. The boundary of a foam has the structure of a MOY graph. We require that the facets and bindings are locally orthogonal to the boundary in order to be able to glue them together canonically. Probably the most local framework is given by the concept of \emph{canopolis} of foams; we refer to \cite{MR2174270, FunctorialitySLN} for more details about this approach.

\begin{dfn}
  \label{dfn:catfoam}
  The category $\Foam$ consists of the following data:
  \begin{itemize}
\item Objects are closed MOY graphs in $\RR^2$,
\item Morphisms from $\Gamma_0$ to $\Gamma_1$ are (ambient isotopy class relatively to the boundary of) decorated foams in $\RR^2\times [0,1]$ whose boundary is contained in the $\RR^2\times\{0,1\}$. The part of the boundary in $\RR^2\times\{0\}$ (\resp$\RR^2\times\{1\}$) is required to be equal to $-\Gamma_0$ (\resp$\Gamma_1$).     
  \end{itemize}
Composition of morphisms is given by stacking foams and rescaling.
\end{dfn}

\subsection{vinyl foams}
\label{sec:vinyl-foams}

In this part, we work in the thickened annulus $\ann\times [0,1]$. If $
x:= \left(\begin{smallmatrix}
  x_1 \\ x_2 \\ x_3
\end{smallmatrix}\right)$ is an element of $\ann\times [0,1]$, we denote by $t_x$ the vector 
$ \left(\begin{smallmatrix}
  -x_2 \\ x_1 \\ 0
\end{smallmatrix}\right)$, by $v$ the vector
$ \left(\begin{smallmatrix}
  0 \\ 0 \\ 1
\end{smallmatrix}\right)$, and by $P_x$ the plane spanned by $t_x$ and $v$.  If $\theta$ is an element of $[0,2\pi[$, $P_\theta$ is the half-plane
$\left\{
\left.\left(
    \begin{smallmatrix}
  \rho \cos \theta \\ \rho \sin \theta \\ t 
\end{smallmatrix}\right) \right | (\rho,t) \in \RR_+\times \RR\right\}$.
Planes parallel to $\RR\times \RR\times \{0\}$ are called \emph{horizontal}. 

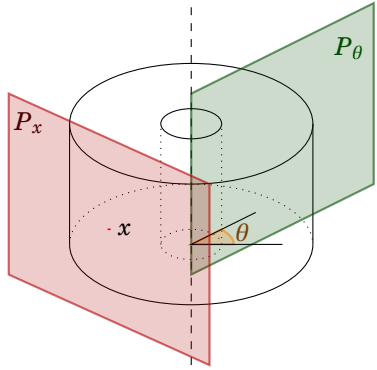
\[
\tikz[scale=0.8]{\begin{scope}[yscale =0.5]
\draw[dashed] (0, -4) -- (0,8);
\draw (0,4) circle (2cm and 2cm);
\draw (0,4) circle (0.5cm and 0.5cm);
\draw (-2,0) arc (-180:0:2cm and 2cm);
\draw[dotted] (-2,0) arc (180:0:2cm and 2cm);
\draw[dotted] (0,0) circle (0.5cm and 0.5cm);
\draw (-2, 0) -- +(0,4);
\draw ( 2, 0) -- +(0,4);
\draw[dotted] ( -0.5, 0) -- +(0,4);
\draw[dotted] ( +0.5, 0) -- +(0,4);
\filldraw[thick, draw= green!30!black, draw opacity =0.6, fill = green!30!black, fill opacity =0.2] (0, -1) -- (0, +5) -- (3, 8) -- (3,2) node[near start, left, opacity=100, green!30!black] {$P_\theta$} -- cycle;
\filldraw[thick, draw= orange, draw opacity =0.6, fill = orange, fill opacity =0.2] (0,0) -- (0.7,0) arc (0:45 :0.7cm) node[near end, right, orange!50!black, opacity= 100] {$\theta$} --cycle;
\draw[thin] (45:1.5) -- (0,0) -- (1.5,0);
\filldraw[thick, draw= red!70!black, draw opacity =0.6, fill = red!70!black, fill opacity =0.2] (-3, 5)  -- (0.3, 2) node[pos=0.1, below, opacity=100, red!30!black] {$P_x$} -- (0.3, -4) -- (-3,-1) --cycle;
\fill[red] (-1.35, 0.5) circle (0.03) node[right, black] {$x$};
\end{scope}}
\]

\begin{dfn}
  \label{dfn:tubelikefoam}
  Let $k$ be a non-negative integer and $\Gamma_0$ and $\Gamma_1$ two vinyl graphs of level $k$. Let $F$ be a foam with boundary embedded in $\ann \times [0,1]$. Suppose that $F\cap (\ann\times \{0\}) = -\Gamma_0$ and $F\cap (\ann\times \{1\}) = \Gamma_1$. We say $F$ is a \emph{vinyl $(\Gamma_1,\Gamma_0)$-foam of level $k$} if for every point $x$ of $F$, the normal line to $F$ at $x$ is \emph{not} contained in $P_x$. See Figure~\ref{fig:tuble-like} for an example.
\end{dfn}

\begin{figure}[ht]
  \centering
\[
\tikz[scale=0.9]{\input{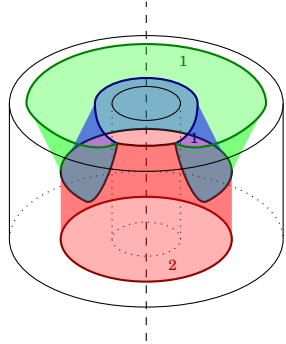}}
\] 
  \caption{An example 
of a vinyl foam}
  \label{fig:tuble-like}
\end{figure}

\begin{dfn}
  \label{dfn:cattubelike}
  The category $\TL_k$ of vinyl foams of level $k$ consists of the following data:
    \begin{itemize}
  \item The objects are elements of $\Vin_k$ \ie{}vinyl graphs of level $k$,
  \item Morphisms from $\Gamma_0$ to $\Gamma_1$ are (ambient isotopy classes of) vinyl $(\Gamma_1,\Gamma_0)$-foams.
  \end{itemize}
  Composition is given by stacking vinyl foams together and rescaling. By this process, some facets may receive different decorations (\ie{}symmetric polynomials), the new decorations of these facets are the products of all the decorations they receive. 
  In the category $\TL_k$ we have one distinguished object which consists of a single essential circle with label $k$. We denote it by $\SS_k$. The \emph{degree} $\degT(F)$ of a vinyl foam $F$ is equal to $\degext_0(F)$ (see below Definition~\ref{dfn:degreefoam}). Note that the additive under the composition in $\TL_k$.
\end{dfn}

\subsection{Foam evaluation}
\label{sec:foam-evaluation}

\begin{dfn}\label{dfn:degreefoam}
    We define the \emph{degree} $\degext_N$ of a decorated foam $F$ as the sum of the following contributions:
    \begin{itemize}
    \item For each facet $f$ with label $a$, set $d(f)=a(N-a)\chi(f)$, where $\chi$ stands for the Euler characteristic;
    \item For each interval binding $b$ (\ie{}not circle-like binding) surrounded by facets with labels $a$, $b$ and $a+b$, set $d(b)= ab + (a+b)(N-a -b)$; 
    \item For each singular point $p$ surrounded with facets with labels $a$, $b$, $c$, $a+b$, $b+c$, $a+b+c$, set $d(p) = ab + bc+ cd + da + ac+ bd $ with $d =N -a -b -c $;
    \item Finally set \[
\degext_N(F) = -\sum_fd(f) + \sum_bd(b) -\sum_p d(p) + \sum_{f} \deg(P_f),
\]
where the variables of polynomials $P_\bullet$ have degree 2.
    \end{itemize}
\end{dfn}
 The degree can be thought of as an analogue of the Euler characteristic.

\begin{dfn}\label{dfn:coloring}
  Let $N$ be a positive integer and $F$ a closed foam. A \emph{$\gll_N$-coloring} (or simply coloring) of a foam $F$ is a map $c$ from $\mathcal{F}(F)$, the set of facets of $F$ to $\mathcal{P}(\{1, \dots, N\})$, the powerset of $\{1, \dots, N\}$, such that:
  \begin{itemize}
  \item For each facet $f$, the number of elements $\# c(f)$ of $c(f)$ is equal to $l(f)$.
  \item For each binding joining a facet $f_1$ with label $a$, a facet $f_2$ with label $b$, and a facet $f_3$ with label $a+b$, we have $c(f_1) \cup c(f_2) = c(f_3)$. This condition is called the \emph{flow condition}.
  \end{itemize}
A \emph{$\gll_N$-colored foam} is a foam together with a $\gll_N$-coloring. A \emph{pigment} is an element of $\{1,\dots,N\}$.
\end{dfn}

\begin{lem}[{\cite[Lemma 2.5]{RW1}}]\label{lem:monobicercle}
  \begin{enumerate}
  \item If $(F,c)$ is a colored foam and $i$ is an element of $\{1,\dots,N\}$, the union (with the identification coming from the gluing procedure) of all the facets which contain the pigment $i$ in their color set is a surface. 
The restriction we imposed on the orientations of facets ensures that $F_i(c)$ inherits an orientation.
  \item  If $(F,c)$ is a colored foam and $i$ and $j$ are two distinct elements of $\{1,\dots,N\}$, the union (with the identification coming from the gluing procedure) of all the facets which contain $i$  or $j$ but not both in their colors is a surface. It is the symmetric difference of $F_i(c)$ and $F_j(c)$. The restriction we imposed on the orientations of facets ensures that $F_{ij}(c)$ can be oriented by taking the orientation of the facets containing $i$ and the reverse orientation on the facets containing $j$. 
  \item In the same situation, we may suppose $i<j$. Consider a binding and denote its surrounding facets $f_1$, $f_2$ and $f_3$. Suppose that $i$ is in $c(f_1)$, $j$ is in $c(f_2)$ and $\{i,j\}$ is  in $c(f_3)$. We say that the binding is \emph{positive with respect to $(i,j)$} if the cyclic order on the binding is $(f_1, f_2, f_3)$ and \emph{negative with respect to $(i,j)$} otherwise. The set $F_{i}(c) \cap F_{j}(c) \cap F_{ij}(c)$ is a collection of disjoint circles. Each of these circles is a union of bindings; for every circle the bindings are either all positive or all negative with respect to $(i,j)$. 
\end{enumerate}
\end{lem}

\begin{dfn}
  \label{dfn:monobicercle}Let $(F,c)$ be a colored foam. 
\begin{enumerate}  \item If $i$ an element of $\{1,\dots,N\}$, the union (with the identification coming from the gluing procedure) of all the facets which contain the pigment $i$ in their colors is a surface. This surface is called the \emph{monochrome surface of $(F,c)$ associated with $i$} and denoted by $F_i(c)$.
\item If $i<j$ is are two elements of $\{1,\dots,N\}$,  the symmetric difference of $F_i(c)$ and $F_j(c)$ is called the \emph{bichrome surface of $(F,c)$ associated with $i,j$} and denoted by $F_{ij}(c)$. 
\item If $i<j$ are two elements of $\{1,\dots,N\}$, a circle in $F_{i}(c) \cap F_{j}(c) \cap F_{ij}(c)$ is 
\emph{positive} (\resp\emph{negative}) \emph{of type $(i,j)$} if it consists of positive (\resp{}negative)
bindings with respect to $(i,j)$. 
We denote by $\theta^+_{ij}(c)_F$ (\resp$\theta^-_{ij}(c)_F$) or simply $\theta^+_{ij}(c)$ (\resp$\theta^-_{ij}(c)$) the number of positive (resp. negative) circles with respect to $(i,j)$. We set  $\theta_{ij}(c)= \theta^+_{ij}(c) +\theta^{-}_{ij}(c)$. 
\end{enumerate}
\end{dfn}

The definition of positive and negative circles is illustrated in Figure~\ref{fig:signsofcircles}.

\begin{figure}[h]
  \centering
  \begin{tikzpicture}
    \input{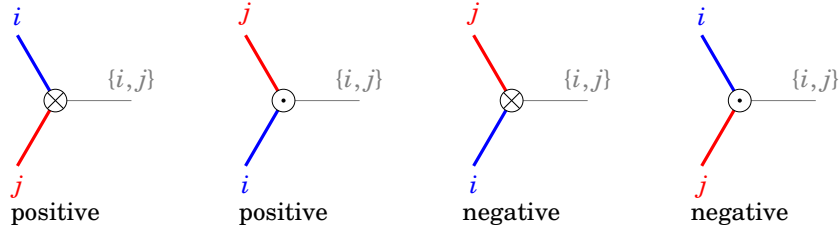}
  \end{tikzpicture}
  \caption{A pictorial definition of the signs of the circle. Here we assume $i<j$. A dotted circle in the middle indicates that the orientation of the binding points to the reader and a crossed circle indicates the other orientation.}
  \label{fig:signsofcircles}
\end{figure}

  \begin{dfn} \label{dfn:exteval} If $(F,c)$ is a $\gll_N$-colored decorated foam, we define:
    \begin{align*}
      s(F,c) &= \sum_{i=1}^N {i\chi(F_i(c))/2}  + \sum_{1\leq i < j \leq N} \theta^+_{ij}(F,c), \\
      P(F,c) &= \prod_{f \textrm{ facet of $F$}} P_f(c(f)),\\
      Q(F,c) &= \prod_{1\leq i < j \leq N} (X_i-X_j)^{\frac{\chi(F_{ij}(c))}{2}}, \\
      \kup{F,c}_N &= (-1)^{s(F,c)} \frac{P(F,c)}{Q(F,c)}.
    \end{align*}
In the definition of $P(F,c)$, $P_f(c(f))$ means the polynomial $P_f$ evaluated on the variables $\{X_i\}_{i\in c(f)}$. Since the polynomial $P_f$ is symmetric, the order of the variables does not matter.
   If $F$ is a decorated foam, we define the \emph{evaluation of the foam $F$ by}:
    \[ \kup{F}_N := \sum_{c \textrm{ $\gll_N$-coloring of $F$}} \kup{F,c}_N. \]
  \end{dfn}

  \begin{rmk}\label{rmk:remove0faces}
    Let $F$ be a foam. We consider $F'$ to be the foam obtained  from $F$ by removing the facets with label $0$. There is a one-one correspondence between the colorings of $F$ and $F'$. For every coloring $c$ of $F$ and the corresponding coloring $c'$ of $F'$, we have $\kup{F,c}_N = \kup{F',c'}_N$ so that $\kup{F}_N = \kup{F'}_N$.  
\end{rmk}

\begin{prop}[\cite{RW1}]\label{prop:sympol}
  Let $F$ be a foam, then $\kup{F}_N$ is an element of $\ZZ[X_1, \dots, X_N]^{S_N}$ which is homogeneous of degree $\degext_N(F)$.
\end{prop}

\begin{dfn}
  \label{dfn:infty-equivalence}
  Let $\sum_i a_i F_i$ and $\sum_j b_jG_j$ be two $\QQ$-linear combinations of foams with the same boundary $\Gamma$. We say that they are \emph{$\infty$-equivalent} if for any foam $H$ with boundary $-\Gamma$ and any $N$, we have $\sum_i a_i\kup{F_i \cup_\Gamma H}_N = \sum_j b_j\kup{G_j\cup_\Gamma H}$.
\end{dfn}

\begin{exa}[{\cite[Equation (23)]{RW2}}]
  \label{exa:dot-migration}
  The foam 
\[
\scriptstyle{\NB{\tikz[scale=1]{\tdplotsetmaincoords{50}{10}
\begin{scope}[tdplot_main_coords]
  \coordinate (aT) at (-1, 1, 1);
  \coordinate (bT) at (-1,-1, 1);
  \coordinate (cT) at ( 1, 0, 1);
  \coordinate (oT) at ( 0, 0, 1);
  \coordinate (aB) at (-1, 1,-1);
  \coordinate (bB) at (-1,-1,-1);
  \coordinate (cB) at ( 1, 0,-1);
  \coordinate (oB) at ( 0, 0,-1);
  \filldraw[thick, draw= black, fill = green, fill opacity =0.3] (aT)-- (oT) node[pos=0.2, below, opacity =1] {$\scriptstyle{a}$} -- (oB) -- (aB) -- (aT);
  \filldraw[thick, draw= black, fill = blue,  fill opacity =0.3] (bT)-- (oT) -- (oB) -- (bB) node[pos = 0.8, above, opacity =1] {$\scriptstyle{b}$} -- (bT);
  \filldraw[thick, draw= black, fill = red,   fill opacity =0.3] (cT)-- (oT) node[midway, below, opacity =1, rotate =-5] {$\scriptstyle{a+b}$} -- (oB) -- (cB) -- (cT);
\end{scope}
  \draw[red, thick, <-] ($(oB)!0.5!(cB)$) -- +(0, -0.55) node[below, red] {$s_{\gamma}$};}}},
\]
where $s_\gamma$ denotes the Schur polynomial associated with the Young diagram $\gamma$, and the linear combination 
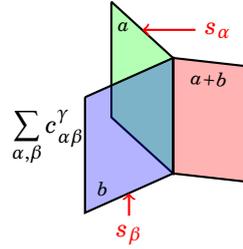
\[
\sum_{\alpha, \beta} c^{\gamma}_{\alpha\beta}
    \scriptstyle{\NB{\tikz[scale=1]{\tdplotsetmaincoords{50}{10}
\begin{scope}[tdplot_main_coords]
  \coordinate (aT) at (-1, 1, 1);
  \coordinate (bT) at (-1,-1, 1);
  \coordinate (cT) at ( 1, 0, 1);
  \coordinate (oT) at ( 0, 0, 1);
  \coordinate (aB) at (-1, 1,-1);
  \coordinate (bB) at (-1,-1,-1);
  \coordinate (cB) at ( 1, 0,-1);
  \coordinate (oB) at ( 0, 0,-1);
  \filldraw[thick, draw= black, fill = green, fill opacity =0.3] (aT)-- (oT) node[pos=0.2, below, opacity =1] {$\scriptstyle{a}$} -- (oB) -- (aB) -- (aT);
  \filldraw[thick, draw= black, fill = blue,  fill opacity =0.3] (bT)-- (oT) -- (oB) -- (bB) node[pos = 0.8, above, opacity =1] {$\scriptstyle{b}$} -- (bT);
  \filldraw[thick, draw= black, fill = red,   fill opacity =0.3] (cT)-- (oT) node[midway, below, opacity =1, rotate =-5] {$\scriptstyle{a+b}$} -- (oB) -- (cB) -- (cT);
\end{scope}
\draw[red, thick, <-] ($(oT)!0.5!(aT)$) -- +(0.7, 0) node[right, red] {$s_{\alpha}$};
\draw[red, thick, <-] ($(oB)!0.5!(bB)$) -- +(0, -0.3) node[below, red] {$s_{\beta}$};
}}}
    \] 
are $\infty$-equivalent. In this linear combination, the integers $c_{\alpha \beta}^\gamma$ are the Littlewood--Richardson coefficients.
\end{exa}

\begin{dfn}
  \label{dfn:tree-like}
  Let $\Gamma$ be a vinyl graph of level $k$ and $F$ a vinyl $(\Gamma,\SS_k)$-foam. The foam $F$ is \emph{tree-like} if for any $\theta$ in $[0, 2\pi]$, $P_\theta\cap F$ is a tree.
\end{dfn}

\begin{prop}[{\cite[Lemmas 3.36 and 3.38]{RW2}}]
  \label{prop:tree-like}
  \begin{itemize}
  \item Let $\Gamma$ be a vinyl graph of level $k$ and $F$ a vinyl
    $(\Gamma,\SS_k)$-foam. Then it is $\infty$-equivalent to a
    $\QQ$-linear combination of tree-like foams with non-trivial
    decorations only on facets bounding $\Gamma$.
  \item Let $\Gamma$ be a vinyl graph of level $k$ and $F_1$ and $F_2$ be two tree-like foams with non-trivial decorations only on facets bounding $\Gamma$. If on those facets the decorations of $F_1$ and $F_2$ coincide then $F_1$ and $F_2$ are $\infty$-equivalent.
  \end{itemize}
\end{prop}
Proposition~\ref{prop:tree-like} tells that the combinatorics of vinyl foams up to $\infty$-equivalence is especially simple. In Subsection~\ref{sec:1-dimens-appr}, we derive from this proposition a new construction of the state spaces associated with vinyl graphs which only deal with 1-dimensional objects.

\begin{dfn}
  \label{dfn:evaluation-vinyl}
  Let $F$ be a vinyl $(\SS_k,\SS_k)$-foam, we denote by $\kups{F}_1$ the value of the polynomial $\kup{\mathrm{cl}(F)}_k$ on $(\underbrace{0, \dots, 0}_{\textrm{$k$ times $0$ }})$, where $\mathrm{cl}(F)$ is the closed (non-vinyl) foam obtained by gluing two disks of label $k$ to $F$ along the two copies of $\SS_k$ which form the boundary of $F$. It is called the \emph{symmetric $\gll_1$-evaluation} of the vinyl foam $F$.  In other words, we take the polynomial $\kup{\mathrm{cl}(F)}_k$ and set $x_1 = \dots = x_k$ to get an rational number.
\end{dfn}

This foam evaluation yields a functor via the so-called \emph{universal construction} \cite{MR1362791}. We denote by $\qvg$ the category of finite dimensional $\ZZ$-graded $\QQ$-vector space. If $V = \bigoplus_{i\in \ZZ} V_i$ is an object of this category, its \emph{graded dimension} is the Laurent polynomial $\dim_qV = \sum_{i \in \ZZ} q^i \dim V_i$. The shift functor is denoted by $q$, so that $(q^jV)_i= V_{i-j}$.  If $P(q)= \sum_{i\in \ZZ}{a_i}q^i$ is a Laurent polynomial with positive integer coefficients, we set
\[
P(q)V:= \bigoplus_{i\in \ZZ} q^i\left(\underbrace{V \oplus V \oplus \dots \oplus V}_{\textrm{$a_i$ times}}\right).
\]
With this conventions, $\dim_q P(q)V = P(q) \dim_q V$.

\begin{dfn}
  \label{dfn:universal-construction}Let $\Gamma$ be a vinyl graph. 
  Define $\TLv(\Gamma)$ to be the $\QQ$-vector space generated by all vinyl $(\Gamma,\SS_k)$-foams. Consider the bilinear map $(\bullet, \bullet)_\Gamma\colon\TLv(\Gamma)\otimes \TLv(\Gamma) \to \QQ$ given by:
\[
(F,G)_\Gamma = \kups{\overline{F}G}_1,
\]
where $\overline{F}$ denotes the mirror image of $F$ about the horizontal plane. We denote by $\syf_{1}(\Gamma)$ the $\QQ$-vector space $\TLv(\Gamma)/\ker(\bullet, \bullet)$. 
\end{dfn}

  In \cite{RW2}, we define a functor $\syf_{k,1}$ from the category $\TL_k$ to $\QQ[T]-\mathsf{proj}_{\mathrm{gr}}$. For any vinyl graph $\Gamma$ of level $k$, the graded $\QQ$-vector space $\syf_1(\Gamma)$ is equal to $\syf_{k,1}(\Gamma) \otimes_{\QQ[T]} \QQ$ where $T$ acts on $\QQ$ by $0$. Hence, most properties of $\syf_{k,1}$ translate to similar ones for $\syf_{1}$. We list some of them below.

\begin{prop}[{\cite[Section 5.1.2]{RW2}}]
  Let $\Gamma_1$ and $\Gamma_2$ be two vinyl graphs and $F$ a vinyl $(\Gamma_1,\Gamma_2)$-foam. 
  The map which sends any vinyl $(\Gamma_1,\SS_k)$-foam $G$ to the class of $F\circ G$ in $\syf_1(\Gamma_2)$ defines a linear map from $\syf_1(\Gamma_1)$ to $\syf_1(\Gamma_2)$, furthermore this map is homogeneous of degree $\deg(F)$.

  This promotes $\syf_1$ into a functor: $\TL_k \to \qvg$. Summing functors for all $k \in \NN$, one obtains a functor
  $\bigoplus_{k\in \NN} \TL_k \to \qvg$ still denoted $\syf_1$.    
\end{prop}

\begin{prop}[{\cite[Propostion 5.9]{RW2}}]
  \label{prop:sym1-rel}
  The functor $\syf_1$ is monoidal (with respect to concentric disjoint union and tensor product) and lifts relations~(\ref{eq:extrelass}), (\ref{eq:extrelass2}), (\ref{eq:extrelbin1}), (\ref{eq:extrelsquare3}) and (\ref{eq:extrelsquare4}). More precisely we have
\begin{align} \label{eq:extrelass-gl1}
   \syf_1\left(\stgamma\right) \simeq \syf_1\left(\stgammaprime\right),
 \end{align}
\begin{align} \label{eq:extrelass2-gl1}
   \syf_1\left(\stgammar\right) \simeq \syf_1\left(\stgammaprimer\right),
 \end{align}
 \begin{align} \label{eq:extrelbin1-gl1} 
\syf_1\left(\digona\right) \simeq \arraycolsep=2.5pt
  \begin{bmatrix}
    m+n \\ m
  \end{bmatrix}
\syf_1\left(\verta\right),
\end{align}
\begin{align}
  \syf_1\!\left(\!\!\!\!\squarec\!\!\!\!\!\right)\simeq\!\!\!\! \bigoplus_{j=\max{(0, m-n)}}^m\!\begin{bmatrix}l \\ k-j \end{bmatrix}
 \syf_1\!\left(\!\!\!\!\!\squared\!\!\!\!\!\right)\!,\label{eq:extrelsquare3-gl1}
\end{align}
\begin{align}
  \syf_1\!\left(\!\!\!\!\squarecc\!\!\!\!\right)\simeq\!\!\!\!\bigoplus_{j=\max{(0, m-n)}}^m\!\begin{bmatrix}l \\ k-j \end{bmatrix}
 \syf_1\!\left(\!\!\!\!\squaredd \!\!\!\!\right)\!.\label{eq:extrelsquare4-gl1}
\end{align}
Moreover all the morphisms are given by images by $\syf_1$ of foams which are trivial outside the region concerned with the local relation.
\end{prop}

\begin{lem}[{\cite[Section 6.3]{RW2}}]
  \label{lem:injective-surjective}
  Suppose $\Gamma_1$ and $\Gamma_2$ are two vinyl graphs which are identical except in a small ball $B$ where
\[
\Gamma_1 =
\NB{\tikz[scale=0.7]{  \begin{scope}[xshift=0cm, yshift =0cm]
    \draw[->] (-0.5,-1) -- (-0.5,1); 
    \draw[->] ( 0.5,-1) -- ( 0.5,1); 
  \end{scope}
}}
\quad \textrm{and}
\quad
\Gamma_2 = \NB{\tikz[scale=0.7]{  \begin{scope}[xshift=0cm, yshift =0cm]
    \draw[->] (-0.5,-1) .. controls +(0,0.3) and +(0,-0.3) .. (0, -0.5)--(0,0.5) node[midway, scale=0.5, left] {$2$} .. controls +(0,0.3) and +(0,-0.3) .. (-0.5, 1);
    \draw[->] ( 0.5,-1) .. controls +(0,0.3) and +(0,-0.3) .. (0, -0.5)--(0,0.5)                                     .. controls +(0,0.3) and +(0,-0.3) .. ( 0.5, 1);
  \end{scope}
}}.
\]
Let $Z$ and $U$ be vinyl $(\Gamma_2,\Gamma_1)$-foam and $(\Gamma_1,\Gamma_2)$-foam  which are trivial expect in $B\times [0,1]$ where
  \[
Z = \NB
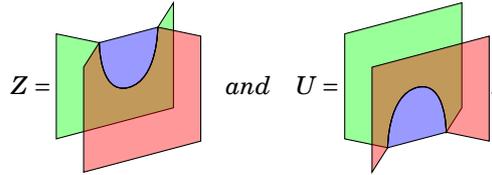
{\tikz[scale=0.85]{\tdplotsetmaincoords{-125}{115}
\begin{scope}[tdplot_main_coords]
  \coordinate (A1B) at (0, 0, 0);
  \coordinate (A2B) at (1, 0, 0);
  \coordinate (C1B) at (0.5, 0.5, 0);
  \coordinate (C2B) at (0.5, 1.5, 0);
  \coordinate (B1B) at (0, 2, 0);
  \coordinate (B2B) at (1, 2, 0);
  \coordinate (A1T) at (0, 0, 2);
  \coordinate (A2T) at (1, 0, 2);
  \coordinate (B1T) at (0, 2, 2);
  \coordinate (B2T) at (1, 2, 2);
  \coordinate (CM) at (0.5, 1,1);
  \filldraw [draw= black, fill =green, fill opacity =0.4] (A2B) --  (A2T) -- (B2T) -- (B2B) -- (C2B)  .. controls +(0,0,0) and +(0,0.5,0 ) .. (CM)  .. controls +(0, -0.5,0) and +(0,0, 0) .. (C1B) -- cycle;
 \filldraw [draw= black, fill =red, fill opacity =0.4] (A1B) --  (A1T) -- (B1T) -- (B1B) -- (C2B)  .. controls +(0,0,0) and +(0,0.5,0 ) .. (CM)  .. controls +(0, -0.5,0) and +(0,0, 0) .. (C1B) -- cycle;
 \filldraw [draw= black, fill =blue, fill opacity =0.4] (C1B) -- (C2B)  .. controls +(0,0,0) and +(0,0.5,0 ) .. (CM)  .. controls +(0, -0.5,0) and +(0,0, 0) .. (C1B) -- cycle;
\end{scope}}}
\quad \textrm{and}
\quad
U = \NB{\tikz[scale=0.85]{\tdplotsetmaincoords{125}{115}
\begin{scope}[tdplot_main_coords]
  \coordinate (A1B) at (0, 0, 0);
  \coordinate (A2B) at (1, 0, 0);
  \coordinate (C1B) at (0.5, 0.5, 0);
  \coordinate (C2B) at (0.5, 1.5, 0);
  \coordinate (B1B) at (0, 2, 0);
  \coordinate (B2B) at (1, 2, 0);
  \coordinate (A1T) at (0, 0, 2);
  \coordinate (A2T) at (1, 0, 2);
  \coordinate (B1T) at (0, 2, 2);
  \coordinate (B2T) at (1, 2, 2);
  \coordinate (CM) at (0.5, 1,1);
  \filldraw [draw= black, fill =green, fill opacity =0.4] (A2B) --  (A2T) -- (B2T) -- (B2B) -- (C2B)  .. controls +(0,0,0) and +(0,0.5,0 ) .. (CM)  .. controls +(0, -0.5,0) and +(0,0, 0) .. (C1B) -- cycle;
 \filldraw [draw= black, fill =red, fill opacity =0.4] (A1B) --  (A1T) -- (B1T) -- (B1B) -- (C2B)  .. controls +(0,0,0) and +(0,0.5,0 ) .. (CM)  .. controls +(0, -0.5,0) and +(0,0, 0) .. (C1B) -- cycle;
 \filldraw [draw= black, fill =blue, fill opacity =0.4] (C1B) -- (C2B)  .. controls +(0,0,0) and +(0,0.5,0 ) .. (CM)  .. controls +(0, -0.5,0) and +(0,0, 0) .. (C1B) -- cycle;
\end{scope}}}.
\]
The maps $\syf_1(Z): \syf_1(\Gamma_1) \to  \syf_1(\Gamma_2)$ and  $\syf_1(U): \syf_1(\Gamma_2) \to  \syf_1(\Gamma_1)$ are injective and surjective respectively.  
\end{lem}

\subsection{Rickard complexes}
\label{sec:rickard-complexes}

In this part, we recall the construction of the uncolored (all strands have label $1$) symmetric $\gll_1$ link homology given in \cite{RW2} (see \cite{queffelec2018annular} and \cite{MR3709661} for alternative approaches). The techniques used here go back to Khovanov
\cite{MR1740682}. In this subsection $D$ denotes a braid closure diagram and $\Xing$ its set of crossings. We denote by $V$ the $\QQ$-vector space generated by $\Xing$.

For any $x$ in $\Xing$, we define a \emph{$0$-resolution} and a \emph{$1$-resolution} by the following rules:

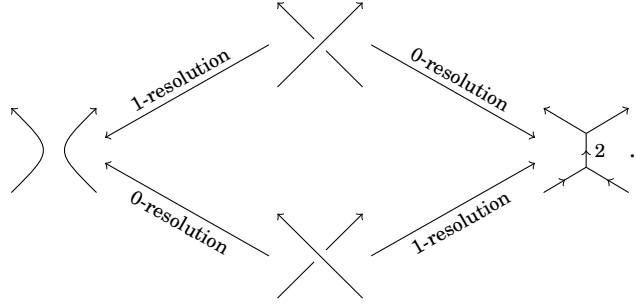
\[
\NB{\tikz[scale=0.7]{\begin{scope}
  \begin{scope}[yshift= 2cm, scale =0.8]
    \coordinate (CTL) at (-1.2, 0);
    \coordinate (CTR) at ( 1.2, 0);
    \coordinate (BL) at (-1, -1);
    \coordinate (BR) at (+1, -1);
    \coordinate (TL) at (-1, +1);
    \coordinate (TR) at (+1, +1);
    \coordinate (O)  at (0, 0);
    \draw[->] (BR) -- (TL);
    \fill[white] (O) circle (2mm);
    \draw[->] (BL) -- (TR);
  \end{scope}
  \begin{scope}[yshift = -2cm, scale =0.8]
    \coordinate (CBL) at (-1.2, 0);
    \coordinate (CBR) at ( 1.2, 0);
    \coordinate (BL) at (-1, -1);
    \coordinate (BR) at (+1, -1);
    \coordinate (TL) at (-1, +1);
    \coordinate (TR) at (+1, +1);
    \coordinate (O)  at (0, 0);
    \draw[->] (BL) -- (TR);
    \fill[white] (O) circle (2mm);
    \draw[->] (BR) -- (TL);
  \end{scope}
  \begin{scope}[xshift = -5cm, scale =0.8]
    \coordinate (CLT) at ( 1.2, 0.3);
    \coordinate (CLB) at ( 1.2,-0.3);
    \coordinate (BL) at (-1, -1);
    \coordinate (BR) at (+1, -1);
    \coordinate (TL) at (-1, +1);
    \coordinate (TR) at (+1, +1);
    \coordinate (O)  at (0, 0);
    \draw[->] (BL) .. controls (O).. (TL);
    \draw[->] (BR) .. controls (O).. (TR);
  \end{scope}
  \begin{scope}[xshift = 5cm, scale =0.8]
    \coordinate (CRT) at (-1.2, 0.3);
    \coordinate (CRB) at (-1.2,-0.3);
    \coordinate (BL) at (-1, -1);
    \coordinate (BR) at (+1, -1);
    \coordinate (TL) at (-1, +1);
    \coordinate (TR) at (+1, +1);
    \coordinate (OB)  at (0, -0.4);
    \coordinate (OT)  at (0, +0.4);
    \coordinate (O)  at (0, 0);
    \draw[->-] (BL) -- (OB);
    \draw[->-] (BR) -- (OB);
    \draw[<-]  (TL) -- (OT);
    \draw[<-]  (TR) -- (OT);
    \draw[->-] (OB) -- (OT) node[midway, right] {$\scriptstyle{2}$};
  \end{scope}
\draw[->] (CBL) -- (CLB) node[sloped, midway, below, scale =0.8] {$0$-resolution};
\draw[->] (CBR) -- (CRB) node[sloped, midway, below, scale =0.8] {$1$-resolution};
\draw[->] (CTL) -- (CLT) node[sloped, midway, above, scale =0.8] {$1$-resolution};
\draw[->] (CTR) -- (CRT) node[sloped, midway, above, scale =0.8] {$0$-resolution};
\end{scope}}}.
\]
The leftmost (\resp{}rightmost) diagram is called the \emph{smooth resolution} (\resp{}\emph{dumble resolution}) of a crossing. The upper crossing is \emph{positive}, while the lowest is \emph{negative}. We denote by $n_+$ and $n_-$ the number of positive and negative crossings respectively.

A \emph{state} $s$ is a function from $\Xing$ to $\{0,1\}$. With each state $s$ is associated a \emph{state diagram} $D_s$: it is the elementary vinyl graph obtained by replacing every $x$ in $\Xing$ by its $s(x)$-resolution.

With each state $s$, we associate a \emph{sign space} $\QQ_s$. It is the one dimensional sub-space of $\Lambda V$ generated by $x_{i_1}\wedge \dots\wedge x_{i_l}$ where $x_{i_1}, \dots, x_{i_l}$ is the full list of crossings for which the value of $s$ is equal to $1$.

With each state $s$, we associate a graded $\QQ$-vector space $C_{\gll_1}(D_s):= \syf_1(D_s) \otimes q^{2n_--n_+ -|s|}\QQ_s$, where $|s|= \#s^{-1}(\{1\})$.

If $y$ is a crossing, a \emph{state transition} at $y$ is a pair of states $(s,s')$, such that $s(x)= s'(x)$ for all $x\in \Xing \setminus\{y\}$ and $s(y) = 0 $ and $s'(y) =1$. 

With each state transition $(s,s')$ is associated a \emph{transition cobordism} $F_{s\to s'}$. It is a vinyl $(D_{s'},D_s)$-foam. Suppose that this state transition is at $y$, then $F_{s\to s'}$ is the identity outside a regular neighborhood of $y$, where it is given by:

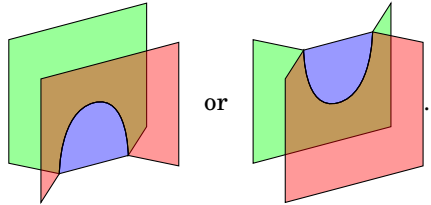
\[
\NB{\tikz[scale=1]{}} \quad\textrm{or}\quad \NB{\tikz[scale=1]{}}. 
\]

With each state transition $(s, s')$ at $y$, one associates a map $d_{s\to s'}: C_{\gll_1}(D_s) \to C_{\gll_1}(D_{s'})$ by setting: $d_{s\to s'} = \syf_{1}(F_{s\to s'})\otimes \bullet \wedge y.$

We define
\[C_{\gll_1}(D):= \bigoplus_{s \textrm{ state}} C_{\gll_1}(D_s)\quad \textrm{and} \quad d_{D} = \sum_{(s,s') \textrm{ state transition}} d_{s \to s'}.\]
The space $C_{\gll_1}(D)$ is bigraded: each space $C_{\gll_1}(D_s)$ is endowed with a grading (the $q$ grading) coming from the grading of $\syf_1(D_s)$. Furthermore, we declare that each space $C_{\gll_1}(D)$ sits in $t$-grading $|s| -n_-$.
The map $d$ has  bidegree $(t,q)=(1,0)$. A schematic description of this construction with the different degree shifts is given in Figure~\ref{fig:rickgl1}.
\begin{figure}[ht]
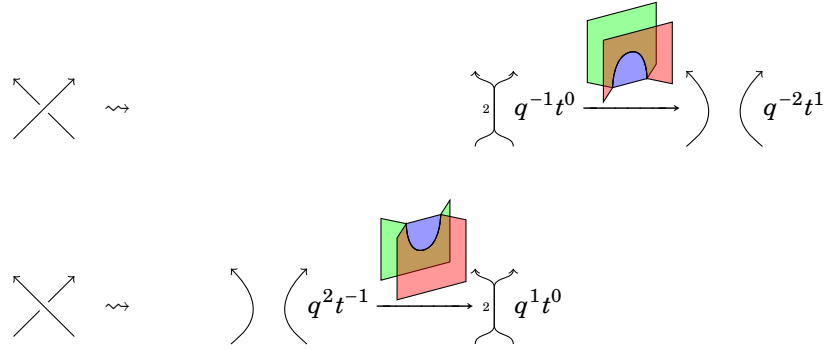

  \centering
  \begin{align*}
    \NB{\tikz[scale=0.4]{\input{\imagesfolder/alex_posXing}}} 
    \quad
&    \rightsquigarrow&
    \quad&
    \NB{\tikz[scale=0.5]{\input{\imagesfolder/alex_dumble}}}q^{-1}t^0 
\xrightarrow{
    \NB{\tikz[scale=0.5]{\input{\imagesfolder/sym_splitfoam}}}}    
 \NB{\tikz[scale=0.5]{\input{\imagesfolder/alex_smoothing}}}q^{-2}t^1  \\
&\\
    \NB{\tikz[scale=0.4]{\input{\imagesfolder/alex_negXing}}} 
    \quad
&    \rightsquigarrow&
    \quad
    \NB{\tikz[scale=0.5]{\input{\imagesfolder/alex_smoothing}}}q^{2}t^{-1} 
\xrightarrow{
    \NB{\tikz[scale=0.5]{\input{\imagesfolder/sym_jointfoam}}}}    
& \NB{\tikz[scale=0.5]{\input{\imagesfolder/alex_dumble}}}q^{1}t^0  \\
  \end{align*}
  \caption{A schematic description of the complex $C_{\gll_1}(D)$.}
  \label{fig:rickgl1}
\end{figure}
  
By standard arguments we get the following result. 
 
  \begin{prop}
    \label{prop:hypercube2complex}
    The pair $(C_{\gll_1}(D),d_D)$ is a chain complex of graded $\QQ$-vector spaces.
  \end{prop}

  \begin{prop}[{\cite[Theorem 6.2]{RW2}}]
    The homology of $(C_{\gll_1}(D), d_D)$ is a bigraded $\QQ$-vector space which only depends on the link represented by the diagram $D$.
  \end{prop}

  The proof of invariance under braid relations (that is Reidemeister $2$ and $3$) is actually not given in \cite{RW2}, but follows from the work on foams by Queffelec--Rose \cite[Theorem 4.8]{queffelec2014mathfrak} in the colored case. Let us point out that this was already present in Vaz PhD thesis \cite[Theorem 5.5.1, Figures 5.14 and 5.16]{VAZPHD} and his proof relies only on $\infty$-equivalence of foams (even if Vaz does not use this terminology). Indeed, the foam relations used by by Vaz for proving invariance under braid-like Reidemeister move follow from $\infty$-equivalence. For further references we gather these results in a proposition.

  \begin{prop}[\cite{VAZPHD}]\label{prop:R2R3-inftyequivalence}
    Let $D$ and $D'$ be two braid closure  diagrams which are the same except in a small ball $B$ where:
      \begin{align*}
\NB
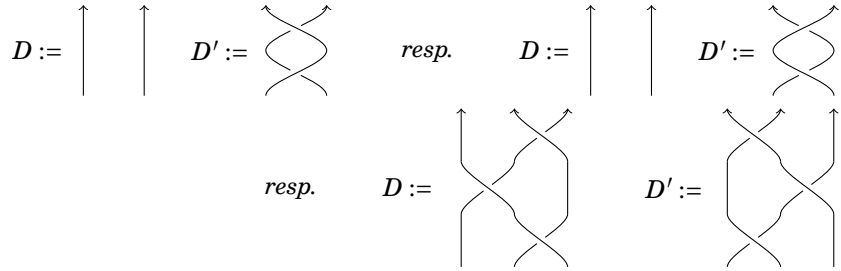
{\tikz[scale =0.4]{\begin{scope}
  \begin{scope}[xshift = 12cm]
    \draw[->] (+1, -1.5) .. controls +(0 ,0.5) and +(0,-0.5 ) .. (-1,0) .. controls +(0,0.5) and +(0 ,-0.5) ..  (+1,+1.5);
    \fill[white] (0,0.7) circle (2mm);
    \fill[white] (0,-0.7) circle (2mm);
    \draw[->] (-1,-1.5) .. controls +(0 ,0.5) and +(0,-0.5 ) .. (1,0) .. controls +(0,0.5) and +(0 ,-0.5) ..  (-1,+1.5); 
  \end{scope}
  \begin{scope}[xshift = 6cm]
    \draw[->] (-1, -1.5) -- (-1, 1.5);
    \draw[->] (+1, -1.5) -- ( +1, 1.5);
  \end{scope}
\node at (3.5,0) {$D:=$};
\node at (9.5,0) {$D':=$};
\end{scope}}} \qquad \quad \textrm{resp.~} \qquad
\NB{\tikz[scale =0.4]{\begin{scope}
  \begin{scope}[xshift=12cm]
    \draw[->] (-1,-1.5) .. controls +(0 ,0.5) and +(0,-0.5 ) .. (1,0) .. controls +(0,0.5) and +(0 ,-0.5) ..  (-1,+1.5); 
    \fill[white] (0,0.7) circle (2mm);
    \fill[white] (0,-0.7) circle (2mm);
    \draw[->] (+1, -1.5) .. controls +(0 ,0.5) and +(0,-0.5 ) .. (-1,0) .. controls +(0,0.5) and +(0 ,-0.5) ..  (+1,+1.5);
  \end{scope}
  \begin{scope}[xshift = 6cm]
    \draw[->] (-1, -1.5) -- (-1, 1.5);
    \draw[->] (+1, -1.5) -- ( +1, 1.5);
  \end{scope}
\node at (3.5,0) {$D:=$};
\node at (9.5,0) {$D':=$};
\end{scope}}} \\ \textrm{resp.~} \qquad
\NB{\tikz[scale =0.35]{\input{\imagesfolder/alex_R3}}}
      \end{align*}
      Then the complexes $(C_{\gll_1}(D), d_D)$ and $(C_{\gll_1}(D'), d_{D'})$ are homotopy equivalent. Futhermore, the maps and the homotopies are given by foams which are identities except in $B \times I$ and the homotopy equivalence can be deduced using only $\infty$-equivalences.
  \end{prop}

\subsection{A 1-dimensional approach}
\label{sec:1-dimens-appr}

The aim of this section is to rephrase some constructions of Section~\ref{sec:foam-evaluation} in a foam-free language. We think the foamy approach makes the definition of the symmetric $\gll_N$ homologies rather natural. 

\begin{dfn}
  Let $\Gamma$ be a vinyl graph. For each (possibly circular) edge $e$ of $\Gamma$, define the \emph{space of decorations of $e$} to be $\QQ[x_1, \dots, x_{\ell(e)}]^{S_{\ell(e)}}$, where $\ell(e)$ is the label of $e$ and we denote the space of decorations of $e$ by $\D(\Gamma, e)$.
  Define the \emph{space of decorations of $\Gamma$} to be the tensor product over $\QQ$ of all the spaces of decorations of edges of $\Gamma$, namely
  \[
\D(\Gamma) := \bigotimes_{e \in E(\Gamma)} \D(\Gamma,e).
\]
A pure tensor in $\D(\Gamma)$ is called a \emph{decoration} of $\Gamma$.
\end{dfn}

\begin{rmk}
\label{rmk:deco-mult}
For any vinyl graph $\Gamma$, $\D(\Gamma)$ is naturally an associative commutative algebra as tensor product of algebras $\D(\Gamma,e)$.
\end{rmk}

\begin{notation}
  If $e$ has label $1$, its space of decorations is $\QQ[x_1]$. The monomial $x_1$ is  represented by a dot $\bullet$. Similarly the monomial $x_1^a$ is represented by $a$ dots or by $\bullet^a$. This notation is compatible with dot notation appearing at the end of Definition~\ref{dfn:foam}.
\end{notation}

\begin{dfn}\label{dfn:rainbow-col}
  Let $\Gamma$ be a vinyl graph of level $k$ and denote by $E(\Gamma)$ its set of edges. An \emph{omnichrome coloring} of $\Gamma$ is a map $c:E(\Gamma) \to \mathcal{P}(\{X_1,\dots, X_k\})$ such that:
  \begin{itemize}
  \item For all $e$ in $E(\Gamma)$, $\#c(e)= \ell(e)$.
  \item For each ray $r$ of $\ann$, the union of the $c(e)$ for the edges $e$ intersecting $r$ is equal to $\{X_1,\dots, X_k\}$.
  \item If $e_1$, $e_2$ and $e_3$ are three adjacent edges with  $\ell(e_1)= \ell(e_2) + \ell(e_3)$, then $c(e_1)= c(e_2)\sqcup c(e_3)$.  
  \end{itemize}
  The set $c(e)$ is called the \emph{color of $e$.} 
\end{dfn}

\begin{rmk}
  The definition of omnichrome colorings implies that: 
  \begin{itemize}
  \item For each vertex of $\Gamma$, the color of the big edge is the
    disjoint union of the colors of the two small edges.
  \item For each ray $r$ of $\ann$, the union of sets $c(e)$ for the edges $e$ intersecting $r$ is actually a disjoint union.
  \item Each coloring $c$ induces an algebra morphism $\varphi_c\thinspace\colon D(\Gamma) \to \QQ(X_1, \dots X_k)$ which identifies $\D(\Gamma,e)$ with $\QQ[(X)_{X\in c(e)}]^{S_{\ell(c)}}$.
\end{itemize}
\end{rmk}

\begin{notation}
  Let $\Gamma$ be a vinyl graph and $c$ be an omnichrome coloring of $\Gamma$. For each split vertex $v$, denote by $e_l(v)$ and $e_r(v)$ the left and right small edges going out of $v$. Set
  \[
    \widehat{Q}(\Gamma,c) \eqdef \prod_{\substack{v \textrm{ split} \\ \textrm{vertex}}} \prod_{\substack{X_i \in c(e_l(v)) \\ X_j \in c(e_r(v))}} (X_i -X_j).
  \]
  Suppose now that $D$ is a decoration of $\Gamma$ and let $(D_e)_{e\in E(\Gamma)}$ such that $D$ is the tensor product of the $D_e$s. We set:
  \[\widehat{P}(\Gamma,D,c) \eqdef \varphi_c(D). \]
  We extend linearly this definition to all elements of $\D(\Gamma)$, and set:
  \[
    \kup{\Gamma,D,c}_\infty = \frac{\widehat{P}(\Gamma,D, c)}{\widehat{Q}(\Gamma, c)}
  \]
\end{notation} 
\begin{dfn} \label{dfn:evaluations-decoration}
  
  Let $\Gamma$ be a vinyl graph and $D$ be a decoration of $\Gamma$. The \emph{$\infty$-evaluation} of $D$ is defined by:
  \[
    \kup{\Gamma,D}_\infty = \sum_{\substack{c \textrm{ omnichrome} \\ \textrm{coloring}}} \kup{\Gamma,D,c}_\infty.
  \]
\end{dfn}

\begin{prop}\label{prop:colev-colev}
  Let $\Gamma$ be a vinyl graph and $D$ be a decoration of $\Gamma$. Let $F_D$ be the closed foam  obtained by composing:
  \begin{itemize}
  \item A dry (see end of Definition~\ref{dfn:foam}) cup of label $k$.
  \item A dry tree-like $(\SS_k,\Gamma)$-foam.
  \item The identity on $\Gamma$ with decoration prescribed by $D$.
  \item A dry tree-like $(\Gamma,\SS_k)$-foam.
  \item A dry cap of label $k$.
  \end{itemize}
  An omnichrome coloring $c$ of $\Gamma$ induces a $\gll_k$-coloring of $F_D$, still denoted by $c$, and we have:
  \[\kup{\Gamma, D,c}_\infty = (-1)^{\frac{k(k+1)}{2}} \kup{F_D,c}_k.
  \]
\end{prop}

Since the omnichrome colorings of $\Gamma$ and the $\gll_k$-colorings of $F_D$ are in one-to-one correspondence, we have then following lemma. 

\begin{cor}
  \label{cor:ev-ev}
  For any vinyl graph $\Gamma$ of level $k$ and any $D$ in $\D(\Gamma)$, with the previous notations, we have:  \[\kup{\Gamma,D}_\infty = (-1)^{\frac{k(k+1)}{2}} \kup{F_D}_k,\]
  and this quantity is a symmetric polynomial in $k$ variables.
\end{cor}

\begin{proof}[Proof of Proposistion~\ref{prop:colev-colev}]
  We fix $\Gamma$, $D$ and $F_D$ as in the proposition. We write $F$ instead of $F_D$. Let $c$ be an omnichrome coloring of $\Gamma$. We still denote by $c$ the corresponding $\gll_k$-coloring of $F$. It follows directly from the definitions of $P$, $\widehat{P}$ and $F$, that $\widehat{P}(\Gamma,D,c)=P(F,c)$. Since
  \[s(F, c)  = \sum_{i=1}^ki\chi(F_i(c))/2 + \sum_{1\leq i<j\leq k} \theta_{ij}^+(F,c) = \sum_{i=1}^ki + \sum_{1\leq i<j\leq k} \theta_{ij}^+(F,c), \]
it is enough to show
\begin{align}
  \label{eq:QwQ}
  \widehat{Q}(\Gamma,c) = (-1)^{\sum_{1\leq i<j\leq k} \theta_{ij}^+(F,c)}Q(F,c).
\end{align}
For $1\leq i < j \leq k$, let $s^+_{ij}$ be the number of split vertices of $\Gamma$ such that $i$ belongs to the color of the right small edge, and $j$ belongs to the color on the left small edge. Similarly let $s^-_{ij}$ be the number of split vertices of $\Gamma$ such that $i$ belongs to the color of the left small edge, and $j$ belongs to the color on the right small edge.

Let us rewrite $\widehat{Q}(\Gamma,c)$ in a different way:
\begin{align*}
  \widehat{Q}(\Gamma,c) = \prod_{1\leq i<j\leq k}  (-1)^{s^+_{ij}}(X_i - X_j)^{s^+_{ij} + s^-_{ij}}.
\end{align*}

It is enough to show that
\begin{align*}
  \chi(F_{ij}(c))/2 = s^+_{ij} + s^-_{ij} \quad \textrm{and} \quad
  \theta_{ij}^+(F,c) = 0 = s^+_{ij} \mod 2.
\end{align*}

Let $1\leq i < j \leq k$ be two integers and let $S$ be the surface $F_{ij}(c)$. The surface $S$ is stable under the reflection about an horizontal plane. Let us denote $S_b$ the part of $S$ which is below this horizontal plane. Since $S = S_b \cup \overline{S_b}$ (where $S_b$ is the mirror image of $S_b$) and that $S_b\cap \overline{S_b}$ is a collection of circle (in $\Gamma$), we have $\chi(S) = 2 \chi(S_b)$.
The surface $S_b$ is totally described by the way $i$ and $j$ interact on the graph $\Gamma$. The edges of $\Gamma$ containing $i$ (\resp{}$j$) in their colors is a cycle $C_i$ (\resp{}$C_j$) winding around $\ann$. The symmetric difference of $C_i$ and $C_j$ is a collection of cycles. We distinguish two cases:
\begin{enumerate}
\item The cycles $C_i$ and $C_j$ are disjoint, hence both $s^+_{ij}$ and $s^-_{ij}$ are zero.
\item The cycles $C_i$ and $C_j$ are not disjoint.
\end{enumerate}
  In the first case $S_b$ is an annulus. Moreover, in this case there are two $(i,j)$-bicolored circles in $S$: one in $S_b$ and one in $\overline{S_b}$ and either both are positive or both are negative. Hence,
  
\begin{align*}
  \chi(S_b)=0= s^+_{ij} + s^-_{ij} =0  \quad \textrm{and} \quad
  \theta_{ij}^+(F,c) = 0=s^+_{ij} \mod 2.
\end{align*}
  
If $C_i$ and $C_j$ are not disjoint, their symmetric difference is a collection of cycles $(C_a)_{a \in A}$ which are homologicaly trivial in the annulus. Moreover, each of these cycles is divided into two intervals, one colored by $i$ the other by $j$. Therefore, there are precisely $s^+_{ij}+ s^-_{ij}$ of these circles. The $(i,j)$-bicolored circles  are in one-to-one correspondence with the cycles $(C_a)_{a \in A}$ and precisely $s^{+}_{ij}$ of them are positive. Hence we have:
\begin{align*}
  \chi(S_b)= s^+_{ij} + s^-_{ij}  \quad \textrm{and} \quad
  \theta_{ij}^+(F,c) =  s^+_{ij}.
\end{align*}
\end{proof}

\begin{notation}
  \label{rmk:R_k}
  If $k$ is a non-negative integer, $R_k$ denotes the algebra $\QQ[X_1, \dots X_k]$.  
\end{notation}

\begin{dfn}
  \label{dfn:pairing-decoration}
  Let $\Gamma$ be a vinyl graph of level $k\geq 0$. We define
an $\R_k$-valued bilinear map on $\R_k \otimes_\QQ \D(\Gamma)$ (denoted by $\D(\Gamma)_{\R_k})$) by:
\[
\begin{array}{crcl}
(\bullet, \bullet)_\infty  \colon\thinspace & \D(\Gamma)_{\R_k} \otimes_{\R_k} \D(\Gamma)_{\R_k} & \to & \R_k \\
  & D\otimes D' &\mapsto & (D, D')_\infty= \kup{\Gamma,DD'}_\infty,
\end{array}
\]
where $DD'$ is the product of decorations $D$ and $D'$ (see Remark~\ref{rmk:deco-mult}). 

We define a $\QQ$-valued bilinear form $(\bullet, \bullet)_1$ on $\D(\Gamma)$ by setting that $(D,D')_1$ is the specialization of $(D,D')_\infty$ at $(0, \dots, 0)$. 
\end{dfn}

From Proposition~\ref{prop:colev-colev} and Proposition~\ref{prop:tree-like}, 
one obtains the following proposition. 

\begin{prop}
  \label{prop:gl1-1dim-is-gl1}
  For any vinyl graph $\Gamma$, the map $D \mapsto F_D$ induces an isomorphism
\[
\D(\Gamma)/ \ker(\bullet, \bullet)_1 \simeq \syf_1(\Gamma).
\]
\end{prop}

Let us denote by $\D(\Gamma)/ \ker(\bullet, \bullet)_1$ by $\syf'_1(\Gamma)$. To use the language of vinyl graph decorations, we  translate the morphisms which define the differential of the complex $C_{\gll_1}(\beta)$ into this new setup.

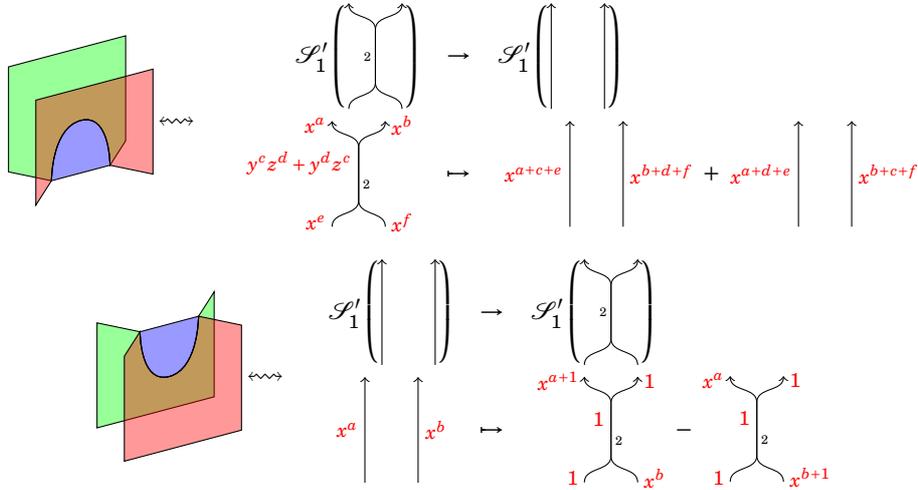
\[
\NB{\tikz[scale=0.85]{}} \leftrightsquigarrow 
\begin{array}{crcl}
  \thinspace & \syf'_1 \left(
\,\NB{\tikz[scale=0.7]{}}\, 
\right) & \to & \syf'_1 \left(
\,\NB{\tikz[scale=0.7]{}}\, 
\right) \\
  & 
\NB{\tikz[scale=0.7]{\input{\imagesfolder/alex_dumble-dec-1}}} 
 &\mapsto & 
\NB{\tikz[scale=0.7]{  \begin{scope}[xshift=0cm, yshift =0cm]
    \draw[->] (-0.5,-1) -- (-0.5,1) node[midway, scale=0.8, red, left] {$x^{a+c+e}$}; 
    \draw[->] ( 0.5,-1) -- ( 0.5,1) node[midway, scale=0.8, red, right] {$x^{b+d+f}$};  
  \end{scope}
}} 
+
\NB{\tikz[scale=0.7]{  \begin{scope}[xshift=0cm, yshift =0cm]
    \draw[->] (-0.5,-1) -- (-0.5,1) node[midway, red, scale=0.8, left] {$x^{a+d+e}$}; 
    \draw[->] ( 0.5,-1) -- ( 0.5,1) node[midway, red, scale=0.8, right] {$x^{b+c+f}$};  
  \end{scope}
}} 
\end{array}
\]

\[
  \NB{\tikz[scale=0.85]{}} 
  \leftrightsquigarrow 
  \begin{array}{crcl}
    \thinspace & \syf'_1 \left(
      \,\NB{\tikz[scale=0.7]{}}\, 
    \right) & \to & \syf'_1 \left(
      \,\NB{\tikz[scale=0.7]{}}\,
    \right)   \\
    & 
    \NB{\tikz[scale=0.7]{  \begin{scope}[xshift=0cm, yshift =0cm]
    \draw[->] (-0.5,-1) -- (-0.5,1) node[midway, scale=0.8, red, left] {$x^{a}$}; 
    \draw[->] ( 0.5,-1) -- ( 0.5,1) node[midway, scale=0.8, red, right] {$x^{b}$};  
  \end{scope}
}} 
    &\mapsto &
    \NB{\tikz[scale=0.7]{\input{\imagesfolder/alex_dumble-dec-2}}} 
    -
    \NB{\tikz[scale=0.7]{\input{\imagesfolder/alex_dumble-dec-3}}} 
  \end{array}
\]
  
\begin{prop}
  \label{prop:HHSoergel-with-deco}
  Let $\Gamma$ be a braid-like MOY graph of level $k$, $\widehat{\Gamma}$ its closure and $\BS(\Gamma)$ be the corresponding Soergel bimodule (see Definition~\ref{dfn:soergel-bimodules}). We have a canonical isomorphism
\[
\left.\D(\widehat{\Gamma})_{\R_k}\right/\ker (\bullet, \bullet)_\infty \simeq \HH_0(\R_k, \BS(\Gamma)).
\]
\end{prop}

\begin{proof}[Sketch of proof.]
  This isomorphism follows from \cite[Proposition 4.18]{RW2}. In this paper we exhibit an isomorphism between $\HH_0(\R_k, \BS(\Gamma))$ and the space $\F_\infty(\widehat{\Gamma})$ of vinyl $(\widehat{\Gamma},\SS_k)$-foams modded out by $\infty$-equivalence.
  Define a morphism by sending a decoration $D$ to the $\infty$-class of the tree-like $(\Gamma,\SS_k)$-foam with decoration prescribed by $D$. It is well-defined and injective because of Corollary~\ref{cor:ev-ev}. It is surjective because of Proposition~\ref{prop:tree-like}.
\end{proof}

\section{Marked vinyl graphs and the $\gll_0$-evaluation}
\label{sec:vinylgraphBP}

The aim of this part is to define and study a TQFT-like functor from the category of marked vinyl graphs to the category of finite dimensional graded $\QQ$-vector spaces.

\subsection{The category of marked foams}

\begin{dfn}
  \label{dfn:markedfoam}
  Let $\Gamma_b$ and $\Gamma_t$ be two marked vinyl graphs. A \emph{marked $(\Gamma_t,\Gamma_b)$-foam} is a vinyl $(\Gamma_t,\Gamma_b)$-foam $F$ with a smooth path $\gamma$ from the base point of $\Gamma_b$ to the base point of $\Gamma_t$ such that $\gamma$ does not intersect any singular locus of $F$.
\end{dfn}

\begin{rmk}
  \label{rmk:not-all-foams-are-marked}
  \begin{enumerate}
  \item All vinyl graphs with an edge labeled $1$ can be considered as marked vinyl graphs by simply
    choosing a base point. Such a statement is not true for foams
    since such path $\gamma$ may not exists.
  \item Note that for a marked $(\Gamma_t,\Gamma_b)$-foam to exist, it is necessary (and actually sufficient) that $\Gamma_b$ and $\Gamma_t$ have the same  level and depth. 
  \end{enumerate}
\end{rmk}

\begin{dfn}
  \label{dfn:catoffoams}
  For each level $k$ and depth $d$, define the category $\Foam_{k,d}$: objects are marked vinyl graphs of level $k$ and depth $d$ and morphisms from $\Gamma_b$ to $\Gamma_t$ are (formal linear combination of) isotopy classes (fixing the boundaries on top and bottom) of marked $(\Gamma_t,\Gamma_b)$-foams.
\end{dfn}

\subsection{Three times the same space} 
\label{sec:gl0-evals}

\begin{notation}
  \label{not:sympol}
  Let $i$ be a non-negative integer, $E_i$ denotes the $i$th elementary symmetric polynomial and $H_i$ the $i$th complete homogeneous symmetric polynomial (the set of variables will be clear from the context). If $i$ is strictly larger that the number of variables, it is understood that $E_i=0$.
\end{notation}

\begin{notation}
  \label{not:addadotatp}
  Let $\Gamma$ be a vinyl graph and $e$ be an edge of $\Gamma$ of label $i$. We denote by  $X_e(\Gamma)$ the vinyl foam $\Gamma\times [0,1]$ decorated by $1$ on every facet but on $e\times [0,1]$ where the decoration is $H_{k-i}$, the complete symmetric polynomial of degree $k-i$. 
We abuse the notations and still denote by $X_e$ the endomorphism of $\syf_1(\Gamma)$ induced by $X_e(\Gamma)$ (\ie$\syf_1(X_e(\Gamma))$).
\end{notation}

\begin{dfn}
  \label{dfn:kup0-dots-at-marked-point}
  Let $\Gamma$ be a marked vinyl graph of level $k$ and denote the marked edge by $e$ on which $\Gamma$ is marked. We define $\syf'_0{(\Gamma,e)}$ to be the space $q^{-(k-1)}\Im (X_e(\Gamma))$.  
\end{dfn}

There is a more elegant way to present this space, which follows from the very definition of the universal construction.
\begin{prop}
  \label{prop:Xk-1-to-UC}
  Let $\Gamma$ be a marked vinyl graph and denote the marked edge by $e$. Let $\TLv(\Gamma)$ be the $\QQ$-vector space generated by all vinyl $(\Gamma,\SS_k)$-Foam. Consider the bilinear map $(\bullet, \bullet)_\Gamma\colon\TLv(\Gamma)\otimes \TLv(\Gamma) \to \QQ$ given by:
\[
(F,G)_\Gamma = \kups{\overline{F}X_eG}_1,
\]
where $\overline{F}$ denotes the mirror image of $F$ about the horizontal plane. The map
\[
\begin{array}{crcl}
  & q^{k-1}\TLv(\Gamma)/(\ker (\bullet, \bullet)_\Gamma) & \to & \syf'_0{(\Gamma,e)} \\
  & [F] &\mapsto & X_e([F])
\end{array}
\]
defines an isomorphism of graded $\QQ$-vector spaces.
\end{prop}

\begin{rmk}
  \label{rmk:syf00}
  \begin{enumerate}
  \item In what follows, the graded $\QQ$-vector space
    $\TLv(\Gamma)/(\ker (\bullet, \bullet)_\Gamma)$ is denoted
    $\syf_0$. Since $\syf_0$ is defined by a universal construction,
    it follows that it extends to a functor (still denoted
    $\syf_0$) from the category $\Foam_{k,d}$ to $\qvg$ for all $k$
    and $d$.
\item The content of Lemma~\ref{lem:injective-surjective} remains true when replacing $\syf_1$ by $\syf_0$. 
  \end{enumerate}
\end{rmk}

\begin{exa}\label{exa:trivial-exa}
  If $\Gamma$ is a marked circle of label $1$, with a base point, then \[\syf_0(\Gamma) \simeq \syf_1(\Gamma)\] and these spaces are isomorphic to $\QQ$ as a graded vector space.
\end{exa}

  Let $\Gamma$ be a MOY graph of level $k$ and let $r$ be a ray of $\ann$ which does not contain any vertex of $\Gamma$. Denote by $e_1, \dots, e_l$ the edges of $\Gamma$ which non-trivially intersect $r$. They have labels $k_1, \dots, k_l$ with $k = k_1 + \dots + k_l$. Suppose that $k_i$ has label $1$ and consider the foam $F_{r,e_i}$ which is the identity with decoration $E_{k_j}$ on $e_j$ for $j=1,\dots,\widehat{i}, \dots l$ and trivial decoration everywhere else. An example is given in Figure~\ref{fig:exampleFre}. 

  \begin{figure}[ht]
\[
\NB{\tikz[scale = 0.9]{\input{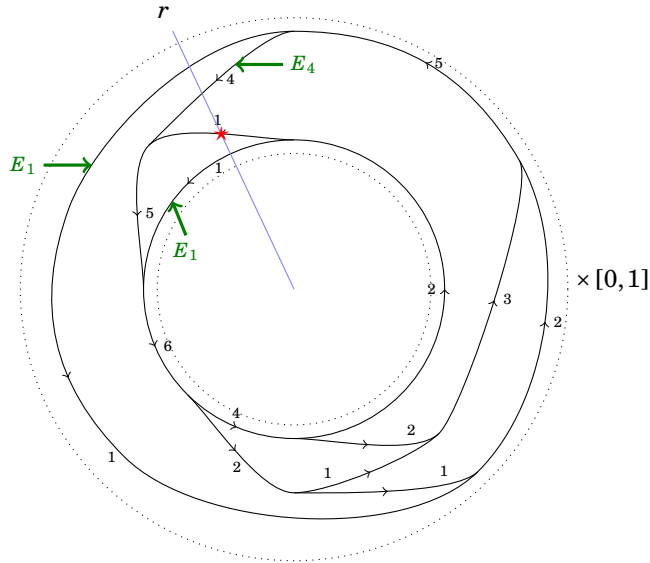}}} \times [0,1]
\]
    \caption{Example of a foam $F_{r, e_i}$. In this case $i=2$, the {\color{red} $\star$} indicate the edge $e_2$. }    \label{fig:exampleFre}
  \end{figure}

\begin{lem}
  \label{lem:eq-dots-on-marked-distributed}
  With the above notations, $\syf_1(F_{r,e_i})= (-1)^{k-1}X_{e_i}(\Gamma)$ as endomorphism of $\syf_1(\Gamma)$. In particular the endomorphism $\kups{F_{r,e_i}}_1$ does not depends on the ray $r$ (provided it intersects $e_i$ non-trivially). 
\end{lem}

\begin{proof}
  This comes from the following identity on symmetric polynomials:
  \begin{align} \label{eq:dotmigration-Xk-1}
X_j^{k-1} = \sum_{a=0}^{k-1}(-1)^a E_{a}(X_1, \dots,\widehat{X_j},\dots, ,X_{k}) H_{k-1-a}(X_1, \dots, X_{k}).
  \end{align}
  Let $G$ be a vinyl $(\SS_k,\Gamma)$-foam. Up to $\infty$-equivalence, we can suppose that $G$ is tree-like and that it is the composition of a tree-like $(\Gamma_b,\SS_k)$-foam $G_b$ with trivial decoration and of a marked $(\Gamma,\Gamma_b)$-foam $G_t$, where
  \[
\Gamma_b:=\NB
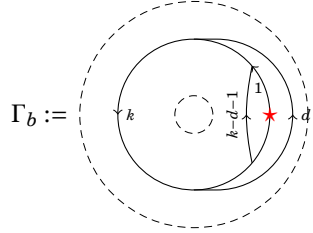
{ \tikz[scale = 1]{\input{\imagesfolder/alex_gammab}}}
    \]

We have $\kups{X_{e_i}(\Gamma)G_tG_b}_1= \kups{G_tX_{e_i}(\Gamma_b)G_b}_1$. Now using the formula~(\ref{eq:dotmigration-Xk-1}) and dots migration (Example~\ref{exa:dot-migration}), we obtain that $X_{e_i}(\Gamma_b)$ is $\infty$-equivalent to
\[
\sum_{a+b+c = k-1} \NB{\tikz[scale=0.9]{\input{\imagesfolder/alex_XGammab-eq}}}
\]
In this formula, the foams involved are of the form $\Gamma_b \times [0,1]$, we only drew $\Gamma_b$ and indicated the non-trivial decoration in dark red. 
Since for the non-equivariant version of the $\gll_1$-evaluation, a foam of level $k$ evaluates to zero if one of its facets of level $k$ carries an homogeneous polynomial of positive degree, this gives:
\[
 \syf_1(X_{e_i}(\Gamma)) = \syf_1 \left( \NB{\tikz[scale=0.9]{\input{\imagesfolder/alex_XGammab-eq-2}}} \right ). 
\]
Let us denote by $F'_{r,e_i}$ the foam on the right-hand side of this identity. It remains to prove that $\syf_1(G_tF'_{r,e_i}G_b) = \syf_1(F_{r,e_i}G_tG_b)$. This holds because  the foams $G_tF_{r,e_i}'G_b$ and $F_{r,e_i}G$ are $\infty$-equivalent. Indeed, using the dot migration, one can migrate the polynomials $E_{d-1}$ and $E_{k-d-t+1}$ on the top of $G_tG_b$ without changing the $\infty$-equivalence class. Once this is done, the foam we have is precisely $F_{r,e_i}G$.
\end{proof}

\begin{rmk}
  \label{rmk:Xk-is-zero-other-UC}
  \begin{enumerate}
  \item\label{it:too-high-deg} Using essentially the same proof, one can prove that for any
    vinyl graph $\Gamma$ of level $k$ and any edge $e$ of $\Gamma$
    the endomorphism $\syf_1(\Gamma)$ which puts a Schur polynomial $s_\alpha$ on the facet adjacent to the edge $e$ is $0$ in $\End(\syf_1(\Gamma))$ when $\alpha$ has strictly more than $k-1$ non-empty columns.

  \item Because of Lemma~\ref{lem:eq-dots-on-marked-distributed}, one can use the foam $F_{r, e_i}$ instead of $X_{e_i}$ in Proposition~\ref{prop:Xk-1-to-UC} and in the definition of the functors $\syf_0$ (Notation~\ref{rmk:syf00}). 
  \end{enumerate}
\end{rmk}

We finally give a last definition of $\syf_0(\Gamma)$ using the universal construction with a deformation of the symmetric $\gll_1$-evaluation. We think that this definition, even if it is not probably the handiest one, is conceptually deeper than the other two and may eventually yield an equivariant version of the theory.

\begin{dfn}
  \label{dfn:forthe-gl0-evaluation}
  Let $F$ be a vinyl $(\SS_k,\SS_k)$-foam of level $k$ with a base point $\star$ in the interior of a facet $f$. Denote by $F'$ the closed (non-vinyl) foam obtained by gluing a cup and a cap of level $k$ on the boundary of $F$ (\ie{}$F'= \mathrm{cl}(F)$).
  A \emph{pointed coloring} of $F$ is a $\gll_k$-coloring $c$ of $F'$ (see Definition~\ref{dfn:coloring}) such that $1$ belongs to $c(f)$. For any pointed coloring $c$ of $F$, we define:
  \[
  \widetilde{Q}(F',c) = \prod_{1\leq i < j \leq k} (X_i - X_j)^{\widetilde{\chi}({F}_{ij}'(c))/2},
  \]
  where
  \[
\widetilde{\chi}({F'_{ij}(c)}) = 
\begin{cases}
  {\chi}({F'_{ij}(c)}) -2 & \textrm{if $i=1$,} \\ 
  {\chi}({F'_{ij}(c)} & \textrm{otherwise.} 

\end{cases}
\]
 The \emph{colored $\star$-evaluation} of the foam $F$ is given by:
\[
\kups{F,c}_\star = \frac{(-1)^{s(F',c)}P(F',c)}{\widetilde{Q}(F',c)}.
\]
Finally the \emph{$\star$-pre-evaluation} of the foam $F$ is given by:
\[
\kups{F}_{\star,\infty} = \sum_{c \textrm{ pointed coloring}} \kups{F,c}_\star.
\]
\end{dfn}

\begin{lem}
  \label{lem:pre-ev-is-polynomial}
For any vinyl $(\SS_k,\SS_k)$-foam $F$, $\kups{F}_{\star, \infty}$ is an element of $R_k^{S_{k-1}}$, where $S_{k-1}$ acts by permuting the $k-1$ last variables.
\end{lem}

\begin{proof}[Sketch of proof]
  The symmetry in the $k-1$ last variables is a direct consequence of the definition of marked colorings. We cannot expect symmetry between the first variable and any other since $X_1$ plays a peculiar role for marked colorings.
  The proof that $\kups{F}_\star$ is a polynomial follows the same line as \cite[Proposition 2.19]{RW1}: the restriction on marked colorings prevent us from performing Kempe moves on connected components of bichrome surfaces containing $\star$. This would have been a problem if these components were spheres and if $\kups{F,s}$ were defined using $Q$ instead of $\widetilde{Q}$. The $-2$ in the definition of $\widetilde{\chi}$ takes care of this issue.
\end{proof}

\begin{dfn}\label{df:star-eval}
  The \emph{$\star$-evaluation} of a vinyl $(\SS_k,\SS_k)$-foam $F$ of level $k$ is the specialization of $\kups{F}_{\star, \infty}$ at $(\underbrace{0,\dots,0}_{k \textrm{ times } 0})$. Equivalently, it is the constant term of the polynomial $\kups{F}_{\star, \infty}$.
\end{dfn}

\begin{prop}
  \label{prop:eqXk-gl0ev}
  Let $\Gamma$ be a marked vinyl graph. Recall that $\TLv(\Gamma)$ is the $\QQ$-vector space generated by all vinyl $(\Gamma,\SS_k)$-foams. Consider the bilinear map $(\!(\bullet, \bullet)\!)_\Gamma:\,\TLv(\Gamma)\otimes \TLv(\Gamma) \to \QQ$ given by:
  \[
(\!(F,G)\!)_\Gamma = \kups{\overline{F}G}_\star,
\]
where the base point of the foam $\overline{F}G$ is the base point of $\Gamma$.
The space $\syf_0(\Gamma)$ is canonically isomorphic to $\TLv(\Gamma)/(\ker (\!(\bullet, \bullet)\!)_{\Gamma})$. 
\end{prop}

\begin{proof}
  The claim follows directly from the fact that for any two foams $F$ and $G$ like in the lemma $\kups{\overline{F}G}_{\star, \infty}= (-1)^{k-1} \kup{\mathrm{cl}(\overline{F}F_{r,e_i}G)}_k$, where $r$ is a ray passing by $\star$, $e_i$ is the edge of $\Gamma$ containing $\star$.
  
  Note that the $\gll_1$-evaluation and the colored marked evaluation makes sense when setting $X_1$ to $0$ and take values in $\QQ\left[X_i^{\pm 1}, \frac1{X_i - X_j}\right]$ where $2 \leq i \leq k$ and $2\leq i<j\leq k$. 
  When setting $X_1$ to $0$ and maintaining the other variables as variables, ${\kup{\overline{F}F_{r,e_i}G,c}_k}$ is equal to $0$ unless the coloring $c$ is a marked coloring. Hence it is enough to show that for any marked coloring $c$, we have:
\[
{\kup{\mathrm{cl}(\overline{F}F_{r,e_i}G,c)}_k}_{|X_1\mapsto0}= 
(-1)^{k-1}{\kups{\overline{F}G,c}_{\star, \infty}}_{|X_1\mapsto0}. 
\]
By definition of $F_{r,e_i}$, we have
\[
  {\kup{\mathrm{cl}(\overline{F}F_{r,e_i}G),c}_k}_{|X_1\mapsto 0}= X_2\cdots X_k {\kup{\mathrm{cl}(\overline{F}G),c}_k}_{|X_1\mapsto0}.
\]
Similarly, by definition of $\widetilde{\chi}(F_{ij}(c))$, we have:
\[
{\kups{\overline{F}G,c}}_{\star, \infty} = (X_1 - X_{2})\cdots(X_1 - X_{k})\kup{\mathrm{cl}(\overline{F}G),c}_k. 
\]
We conclude by setting $X_1$ to $0$ in the previous identity.
\end{proof}

\subsection{A foamy functor}
\label{sec:foamy-functor}

\begin{dfn}\label{dfn:bd-graph-foam}
\begin{enumerate}
\item  Let $\Gamma$ be a marked vinyl graph of level $k$ and depth $0$. We define $\bd(\Gamma)$ to be the marked vinyl graph of level $k+1$ and depth $0$ obtained from $\Gamma$ by replacing 
$\NB{\tikz{\draw (-30:0.5) arc (-30:30:0.5) node[midway,red]{$\star$};}}$ in $\Gamma$ by.
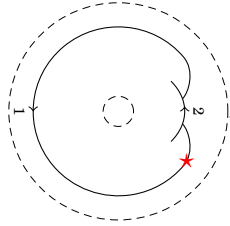
\[
\NB{\tikz[scale = 0.8]{\input{\imagesfolder/alex_theta-kminus2-1-1-open}}}.
\]
\item  Let $\Gamma_1$ and $\Gamma_2$ be two marked vinyl graphs of level $k$ and depth $0$, and $F$ a vinyl marked $(\Gamma_2,\Gamma_1)$-foam. We define $\bd(F)$ to be the marked vinyl $(\bd(\Gamma_2),\bd(\Gamma_1))$-foam obtained from $F$ by replacing $\NB{\tikz{\draw (-30:0.5) arc (-30:30:0.5) node[midway,red]{$\star$};}}\times [0,1]$ in $F$ by.
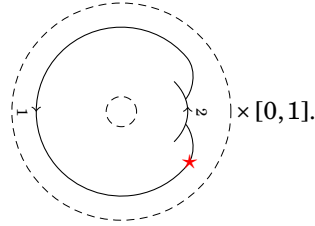
\[
\NB{\tikz[scale = 0.8]{\input{\imagesfolder/alex_theta-kminus2-1-1-open}}} \times [0,1].
\]
\end{enumerate}
\end{dfn}

We immediately get the following lemma:
\begin{lem}
  \label{lem:bd-functor}
  For all $k>0$, the maps $\bd$ given in Definition~\ref{dfn:bd-graph-foam} define a functor from $\Foam_{k,0}$ to $\Foam_{k+1,0}$.
\end{lem}

  Let $\Gamma$ be a vinyl graph. Consider graded vector spaces $\syf_0(\Gamma)$ and $\syf_0(\bd(\Gamma))$. Let $x$ be an element of $\syf_0(\Gamma)$ represented by a tree-like $(\Gamma,\SS_k)$-foam $F$, up to $\infty$-equivalence. We can suppose that $F$ is the composition of a dry tree-like $(\Theta_{k-1,1},\SS_k)$-foam $F_b$ and a marked vinyl $(\Gamma,\Theta_{k-1,1})$-foam $F_t$ with
  \[
    \Theta_{k-1,1}:=\NB
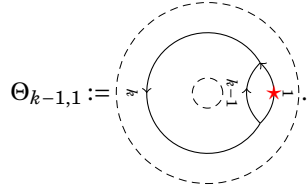
{\tikz[scale = 0.8]{\input{\imagesfolder/alex_theta-kminus1-1}}}.
  \]
  Let us define $F'$ to be the composition of a tree-like
  $(\bd(\Theta_{k-1,1,1}),\SS_{k+1})$-foam $F'_b$ with trivial decoration and $\bd(F_t)$ with
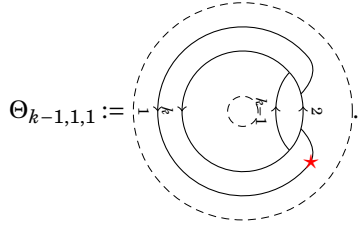
  \[
    \Theta_{k-1,1,1}:=\NB{\tikz[scale = 0.8]{\input{\imagesfolder/alex_theta-kminus1-1-1}}}.
  \]
  The vinyl (actually tree-like) $(\bd(\Gamma),\SS_{k+1})$-foam $F'$ defines an element in $\syf_0(\bd(\Gamma))$. We denote this element by $\psi_\Gamma(x)$. We extend $\psi$ linearly on combinations of tree-like foams. 

  \begin{lem}
    \label{lem:psi-wd}
    The map $\psi: \syf_0(\Gamma) \to \syf_0(\bd(\Gamma))$ is well-defined and is an isomorphism.
  \end{lem}
  \begin{proof}
    The map is definitely well-defined at the level of tree-like foams. Such foams span $\syf_{\Gamma}$. In order to prove that it is well-defined, it is enough to show that a linear combination of foams which is equal to zero in $\syf_{\Gamma}$ is sent onto zero.
    We keep the same notations as above and consider $F= F_t\circ F_b$ and $G=G_t\circ G_b$ be two tree-like $(\Gamma,\SS_k)$-foams and $F'= \bd(F_t)\circ F'_b$ and $G'=\bd(G_t)\circ G'_b$. Note that $F_b= G_b$.
    Let us prove that:
    \begin{align} \label{eq:eval-bd}
    (F,G)_\Gamma = -(F', G')_{\bd(\Gamma)}. 
    \end{align}
    We need to prove that $\kups{\overline{F}X_eG}_1 = - \kups{\overline{F'}X_{e'}G'}_1$.
    We have:
    \begin{align*}
      \kups{\overline{F}X_eG}_1 &= \kups{\overline{F_b}\,\overline{F_t}X_eG_tG_b}_1 =
      \kups{\overline{F_b}X_e\overline{F_t}G_tG_b}_1  \quad \textrm{and} \\
      \kups{\overline{F'}X_eG'}_1 &= \kups{\overline{F'_b}\,\overline{\bd(F_t)}X_{e'}\bd(G_t)G_b}_1 =
      \kups{\overline{F'_b}X_{e'}\overline{\bd(F_t)}\bd(G_t)G_b}_1 \\ &=
      \kups{\overline{F'_b}X_{e'}\bd(\overline{F_t}{G_t})G_b}.
    \end{align*}
    Up to $\infty$-equivalence on both side, we can suppose that $\overline{F_t}{G_t}$ is the identity foam (with some decorations) on $\Theta_{k-1,1}$ and $\bd(\overline{F_t}{G_t})$ is the identity foam (decorated accordingly) on $\bd(\Theta_{k-1,1})= \Theta_{k-1,1,1}$. If the decorations are non-trivial, then for degree reasons, both $\kups{\overline{F}X_eG}_1$ and $\kups{\overline{F'}X_{e'}G'}_1$ are equal to zero. Else, it is an easy computation to check that:
    \[
      \kups{\overline{F_b}X_eG_b}_1 = (-1)^k
\qquad      \textrm{and} \qquad
\kups{\overline{F'_b}X_{e'}G'_b}_1 = (-1)^{k+1}.
    \]
    Indeed, we have:
    \begin{align*}
      \kup{\mathrm{cl}(\overline{F_b}X_eG_b)}_k =
      \sum_{i=1}^{k} \frac{(-1)^{k(k+1)/2} X_i^{k-1}}{\prod_{\substack{j=1 \\ j\neq i}}^k (X_i-X_j)} = (-1)^k
    \end{align*}
    and
    \begin{align*}
      \kups{\overline{F'_b}X_{e'}G'_b}_1 = (-1)^{k+1}=
      \sum_{i=1}^{k+1} \sum_{\substack{j=1 \\ j \neq i}}^k \frac{(-1)^{(k+2)(k+1)/2 +1} X_j^{k}}{\prod_{\substack{\ell=1 \\ \ell \neq i}}^{k+1} (X_i-X_\ell)}   = (-1)^{k+1}
    \end{align*}

This implies that $\psi$ is well-defined. It gives as well that it is injective. The surjectivity comes from the fact, that the foams of the form $\bd(F)F'_b$ span $\syf_0(\bd(\Gamma))$.  
  \end{proof}

  \begin{cor} \label{cor:equivalence}
    The functors $\syf_0\colon \Foam_{k,0}\to \qvg $ and $\syf_{0}\circ \bd\colon \Foam_{k,0} \to \qvg$ are isomorphic. The collection of maps $(\psi_{\Gamma})_{\Gamma \in \Foam_{k,0}}$ gives a natural isomorphism.
  \end{cor}

\subsection{Graded dimension of $\syf_0(\Gamma)$}
\label{sec:grad-dimens-sym_0g}
For $N \geq 1$, the graded dimension of $\syf_N(\Gamma)$ is the Laurent polynomial in $q$ given by the value of $\Gamma$ seen as an endomorphism of a 1-dimensional representation of the $\CC(q)$-algebra $U_q(\gll_N)$. It is far from clear what would be an analog result for $\syf_0$. 

The aim of this part is to give tools for computing the graded dimension of $\syf_0(\Gamma)$ for $\Gamma$ a marked vinyl graph. Unfortunately, these tools are not always sufficient  to compute this dimension. 

Then, we define a quantity attached to any marked vinyl graph $\Gamma$. Finally, we show that in some favorable cases, this quantity equals the graded dimension $\syf_0(\Gamma)$. 

\begin{prop}
  \label{prop:syf0-skeinrelation}
  The functor $\syf_0$ lift relations~(\ref{eq:extrelass}), (\ref{eq:extrelass2}), (\ref{eq:extrelbin1}), (\ref{eq:extrelsquare3}) and (\ref{eq:extrelsquare4}) far from the base point to isomorphism. Namely:
\begin{align} \label{eq:extrelass-gl0}
   \syf_0\left(\stgamma\right) \simeq \syf_0\left(\stgammaprime\right),
 \end{align}
\begin{align} \label{eq:extrelass2-gl0}
   \syf_0\left(\stgammar\right) \simeq \syf_0\left(\stgammaprimer\right),
 \end{align}
 \begin{align} \label{eq:extrelbin1-gl0} 
\syf_0\left(\digona\right) \simeq \arraycolsep=2.5pt
  \begin{bmatrix}
    m+n \\ m
  \end{bmatrix}
\syf_0\left(\verta\right),
\end{align}
\begin{align}
  \syf_0\!\left(\!\!\!\!\squarec\!\!\!\!\!\right)\simeq\!\!\!\! \bigoplus_{j=\max{(0, m-n)}}^m\!\begin{bmatrix}l \\ k-j \end{bmatrix}
 \syf_0\!\left(\!\!\!\!\!\squared\!\!\!\!\!\right)\!,\label{eq:extrelsquare3-gl0}
\end{align}
\begin{align}
  \syf_0\!\left(\!\!\!\!\squarecc\!\!\!\!\right)\simeq\!\!\!\!\bigoplus_{j=\max{(0, m-n)}}^m\!\begin{bmatrix}l \\ k-j \end{bmatrix}
 \syf_0\!\left(\!\!\!\!\squaredd \!\!\!\!\right)\!.\label{eq:extrelsquare4-gl0}
\end{align}
\end{prop}

\begin{proof}
  This follows from Proposition~\ref{prop:sym1-rel}. Since the base point is not in the ball where the local relation happens, the morphism $\syf_1(X_e)$ ($e$ being the edge containg the base point), commutes with all the morphisms given in Proposition~\ref{prop:sym1-rel}.
\end{proof}

\begin{prop}\label{prop:syf0-0-split}
  If $\Gamma$ is a marked vinyl graph which is not connected, then $\syf_0(\Gamma)=0$.
\end{prop}

\begin{proof}
  It follows from Remark~\ref{rmk:Xk-is-zero-other-UC} (\ref{it:too-high-deg}) and from the monoidality of the functor $\syf_1$.
\end{proof}

\begin{prop}
  \label{prop:gd-satifies-rel}
  The function $\gd$ which associates with any marked vinyl graph $\Gamma$ of depth $0$ the graded dimension of $\syf_0(\Gamma)$ induces a function $\gd: \skeinp \to \CC(q)$ which satisfies the hypothesis of Proposition~\ref{prop:skeinrelation-to-unicity}. 
\end{prop}

\begin{proof}
  By Proposition~\ref{prop:syf0-skeinrelation}, $\syf_0$ categorifies relations (\ref{eq:extrelass}), (\ref{eq:extrelass2}), (\ref{eq:extrelbin1}), (\ref{eq:extrelsquare3}) and (\ref{eq:extrelsquare4}), hence $\gd$ induces a map from $\skeinp$ to $\CC(q)$. Moreover, Propostion~\ref{prop:syf0-0-split}, Example~\ref{exa:trivial-exa}, and Corollary~\ref{cor:equivalence} imply that the hypothesis of Proposition~\ref{prop:skeinrelation-to-unicity} is fulfilled. 
\end{proof}

\subsubsection{Expected graded dimension}
\label{sec:expect-grad-dimens}

The aim of this subsection, is to give an elementary approach to the graded dimension of $\syf_0(\Gamma)$, for $\Gamma$ a marked vinyl graph of depth $0$ or $1$.
\begin{dfn}
Let $\Gamma$ be a vinyl graph of level $k$. As an element of $\skein$, it decomposes in the basis of nested circles (see Proposition~\ref{prop:QR}). Let us denote by $a_k$ the coefficient of the circle $\SS_k$ in this decomposition. 
The \emph{expected graded dimension} is defined by:
\[
\egd(\Gamma):=  \frac{a_k}{[k]}.
\]
We extend this map linearly to $\ZZ[q, q^{-1}]$-linear combinations of vinyl graphs.
\end{dfn}

We want to point out that at least when specializing to $q=1$, the expected graded dimension has a more natural definition. For a vinyl graph $\Gamma$, let us write:
\[
U_\Gamma(N) = \left(\kups{\Gamma}_N\right)_{q=1}.
\] 

\begin{prop}
  \label{prop:egd-q=1}
  Let $\Gamma$ be a vinyl graph of level $k$ then $U_\Gamma$ is a polynomial in $N$ and 
\[\egd(\Gamma)_{q=1} =  U'_\Gamma(0).\] 
\end{prop}

\begin{proof}
  Let us first consider the case where $\Gamma$ is a collection of $i$
  circles (the circles are meant to have non-zero labels). Then $0$ is
  a root of multiplicity $i$ of $U_\Gamma$. Hence if $\Gamma$ consists
  of more than one circle, then $U'_\Gamma(0)$ is $0$ as it should. If
  $\Gamma$ consists of one circle, then this circle has level $k$.  We
  have:
  \[
    U_{\SS_k}(N) = \frac{\prod_{i=0}^{k-1}(N+i)}{\prod_{i=1}^ki}.
  \]
  Hence $U_{\SS_k}'(0) = \frac{1}{k}$, we have
  \[
    \egd(\SS_k)_{q=1}=\frac{1}{k}= U_{\SS_k}'(0). 
  \]
  For general graphs, it follows from Proposition~\ref{prop:QR}.
\end{proof}

\begin{cjc}
  \label{cjc:expected-graded-dimension}
  Let $\Gamma$ be a marked vinyl graph, then
  \[
  \dim_q\syf_0(\Gamma) = \egd(\Gamma).
\]
\end{cjc}

\begin{rmk}
  In particular, the previous conjecture implies that $\egd(\Gamma)$ is a Laurent polynomial in $q$. 
\end{rmk}

The rest of this section is devoted to proving Conjecture~\ref{cjc:expected-graded-dimension} when $\Gamma$ has depth $0$ or $1$. For this we will make use of Proposition~\ref{prop:skeinrelation-to-unicity}.

\begin{rmk}  \label{rmk:egd-skein-relation}
From its very definition, one obtains that the expected graded dimension $\egd$ is a map from $\skein$ to $\CC(q)$ such that:
  \begin{itemize}
  \item If $\Gamma$ is a single circle with label $1$, then $\egd(\Gamma) =1$,
  \item If $\Gamma$ is not connected, then $\egd(\Gamma) =0$.
\end{itemize}
\end{rmk}

\begin{lem}
  \label{lem:baddigon-egd}
  Suppose that $\Gamma$ and $\Gamma'$ are vinyl graphs which are related by an outer digon (see Proposition~\ref{prop:skeinrelation-to-unicity}). Then ${\egd(\Gamma)} = {\egd(\Gamma')}$. 
\end{lem}

\begin{proof}
  By linearity of $\egd$, it is enough to prove the statement for some $\Gamma'$ being the closure of elements of a basis of the Hecke algebra $\Hecke_{k-1}$. We can actually reduce further the cases by using conjugation on the $k-2$ first strands: it is enough to prove the statement for  closure of graphs of the form:
  \[
\underbrace{\NB
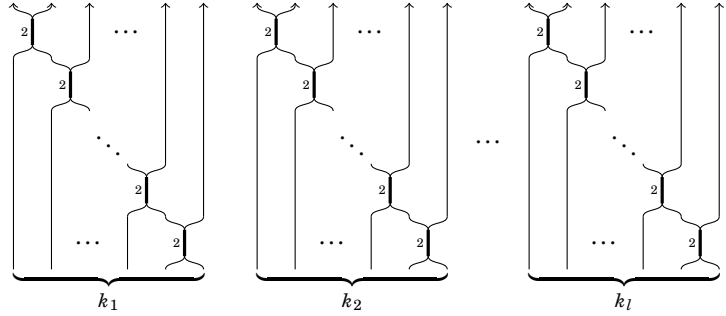
{\tikz[xscale=0.5, yscale =0.5]{\begin{scope}[yscale=0.7]
\newcommand{\dumble}[2]{
  \begin{scope}[xshift=#1cm, yshift =#2cm]
    \draw (-0.5,-1) .. controls +(0,0.3) and +(0,-0.3) .. (0, -0.5)--(0,0.5) node[midway, scale=0.5, left] {$2$} .. controls +(0,0.3) and +(0,-0.3) .. (-0.5, 1);
    \draw ( 0.5,-1) .. controls +(0,0.3) and +(0,-0.3) .. (0, -0.5)--(0,0.5)                                     .. controls +(0,0.3) and +(0,-0.3) .. ( 0.5, 1);
    \draw[very thick] (0, -0.5)--(0,0.5);
  \end{scope}
  }
  \dumble{4}{0}
  \dumble{3}{2}
  \dumble{1}{6}
  \dumble{0}{8}
  \draw[<-] ( 4.5, 9.1) -- +(0,-8.1);
  \draw[<-] ( 3.5, 9.1) -- +(0,-6.1);
  \draw[<-] ( 1.5, 9.1) -- +(0,-2.1);
  \draw[<-] ( 0.5, 9.1) -- +(0,-0.1);
  \draw[<-] ( -0.5,9.1) -- +(0,-0.1);
  \draw ( 2.5, -1) -- +(0,2);
  \draw ( 0.5, -1) -- +(0,6);
  \draw (-0.5, -1) -- +(0,8);
  \node at (1.5, 0) {$\dots$};
  \node at (2.5, 8) {$\dots$};
  \node at (2  , 4) {$\ddots$};
\end{scope}}}}_{k_1} \qquad
\underbrace{\NB{\tikz[xscale=0.5, yscale =0.5]{}}}_{k_2}
\quad \dots
\quad
\underbrace{\NB{\tikz[xscale=0.5, yscale =0.5]{}}}_{k_l} 
\]
If $\Gamma'$ is not connected, neither is $\Gamma$ and we obtain $\egd(\Gamma) = \egd(\Gamma') =0$. Hence the only remaining case to check is when $\Gamma$ is the closure of
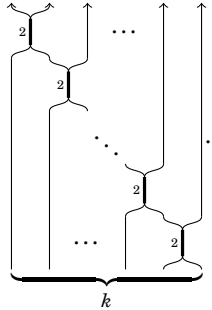
\[
\underbrace{\NB{\tikz[xscale=0.5, yscale =0.5]{}}}_{k}.
\]
Using the skein relation (\ref{eq:extrelsquare3}), we get:
\begin{align*}
&\egd\left(\mathrm{cl}\left(\, \NB{\tikz[xscale=0.5, yscale =0.5]{\begin{scope}[yscale=0.7]
  \newcommand{\dumble}[3]{
  \begin{scope}[xshift=#1cm, yshift =#2cm]
    \draw (-0.5,-1) .. controls +(0,0.3) and +(0,-0.3) .. (0, -0.5)--(0,0.5) node[midway, scale=0.5, left] {$#3$} .. controls +(0,0.3) and +(0,-0.3) .. (-0.5, 1);
    \draw ( 0.5,-1) .. controls +(0,0.3) and +(0,-0.3) .. (0, -0.5)--(0,0.5)                                     .. controls +(0,0.3) and +(0,-0.3) .. ( 0.5, 1);
  \end{scope}
  }
  \dumble{4}{0}{i+1} 
  \dumble{3}{2}{2}
  \dumble{1}{6}{2}
  \dumble{0}{8}{2}
  \node[right, scale = 0.5] at (4.5 , 3) {$i$};
  \node[right, scale = 0.5] at (4.5 , -1) {$i$};
  \draw[<-] ( 4.5, 9.1) -- +(0,-8.1);
  \draw[<-] ( 3.5, 9.1) -- +(0,-6.1);
  \draw[<-] ( 1.5, 9.1) -- +(0,-2.1);
  \draw[<-] ( 0.5, 9.1) -- +(0,-0.1);
  \draw[<-] ( -0.5,9.1) -- +(0,-0.1);
  \draw ( 2.5, -1) -- +(0,2);
  \draw ( 0.5, -1) -- +(0,6);
  \draw (-0.5, -1) -- +(0,8);
  \node at (1.5, 0) {$\dots$};
  \node at (2.5, 8) {$\dots$};
  \node at (2  , 4) {$\ddots$};
\end{scope}}}\right)\right)
\\ \qquad&=
 \egd\left(\mathrm{cl}\left(\, \NB{\tikz[xscale=0.5, yscale =0.5]{\begin{scope}[yscale=0.7]
  \newcommand{\dumble}[3]{
  \begin{scope}[xshift=#1cm, yshift =#2cm]
    \draw (-0.5,-1) .. controls +(0,0.3) and +(0,-0.3) .. (0, -0.5)--(0,0.5) node[midway, scale=0.5, left] {$#3$} .. controls +(0,0.3) and +(0,-0.3) .. (-0.5, 1);
    \draw ( 0.5,-1) .. controls +(0,0.3) and +(0,-0.3) .. (0, -0.5)--(0,0.5)                                     .. controls +(0,0.3) and +(0,-0.3) .. ( 0.5, 1);
  \end{scope}
  }
  \dumble{3}{2}{2}
  \dumble{1}{6}{2}
  \dumble{0}{8}{2}
  \node[right, scale = 0.5] at (4.5 , 5) {$i$};
  \draw[<-] ( 4.5, 9.1) -- +(0,-10.1);
  \draw[<-] ( 3.5, 9.1) -- +(0,-6.1);
  \draw[<-] ( 1.5, 9.1) -- +(0,-2.1);
  \draw[<-] ( 0.5, 9.1) -- +(0,-0.1);
  \draw[<-] ( -0.5,9.1) -- +(0,-0.1);
  \draw ( 2.5, -1) -- +(0,2);
  \draw ( 3.5, -1) -- +(0,2);
  \draw ( 0.5, -1) -- +(0,6);
  \draw (-0.5, -1) -- +(0,8);
  \node at (1.5, 0) {$\dots$};
  \node at (2.5, 8) {$\dots$};
  \node at (2  , 4) {$\ddots$};
\end{scope}}} \right)\right)+
 \egd\left(\mathrm{cl}\left( \,\NB{\tikz[xscale=0.5, yscale =0.5]{\begin{scope}[yscale=0.7]
  \newcommand{\dumble}[3]{
  \begin{scope}[xshift=#1cm, yshift =#2cm]
    \draw (-0.5,-1) .. controls +(0,0.3) and +(0,-0.3) .. (0, -0.5)--(0,0.5) node[midway, scale=0.5, left] {$#3$} .. controls +(0,0.3) and +(0,-0.3) .. (-0.5, 1);
    \draw ( 0.5,-1) .. controls +(0,0.3) and +(0,-0.3) .. (0, -0.5)--(0,0.5)                                     .. controls +(0,0.3) and +(0,-0.3) .. ( 0.5, 1);
  \end{scope}
  }
  \dumble{3}{2}{i+2}
  \dumble{1}{6}{2}
  \dumble{0}{8}{2}
  \node[right, scale = 0.5] at (3.5 , 0) {$i+1$};
  \node[right, scale = 0.5] at (3.5 , 5) {$i+1$};
  \draw[<-] ( 3.5, 9.1) -- +(0,-6.1);
  \draw[<-] ( 1.5, 9.1) -- +(0,-2.1);
  \draw[<-] ( 0.5, 9.1) -- +(0,-0.1);
  \draw[<-] ( -0.5,9.1) -- +(0,-0.1);
  \draw ( 2.5, -1) -- +(0,2);
  \draw ( 3.5, -1) -- +(0,2);
  \draw ( 0.5, -1) -- +(0,6);
  \draw (-0.5, -1) -- +(0,8);
  \node at (1.5, 0) {$\dots$};
  \node at (2.5, 8) {$\dots$};
  \node at (2  , 4) {$\ddots$};
\end{scope}}}\right)\right),
\end{align*}
and deduce by induction that in this case $\egd(\Gamma') = 1=\egd(\Gamma)$.
\end{proof}

Hence from Proposition~\ref{prop:skeinrelation-to-unicity}, we deduce:
\begin{prop}\label{prop:egd-depth0}
  If $\Gamma$ is a marked vinyl graph of depth $0$, then
  \[
\egd(\Gamma) = \gd(\Gamma).
  \]
\end{prop}

\begin{rmk}\label{rmk:egd-depthk}
  \begin{enumerate}
  \item We could have played the same game and put the base point on
    the leftmost strand. The same arguments would give that if
    $\Gamma$ is a marked vinyl graph of level $k$ and depth $k$ then
    $\egd(\Gamma) = \gd(\Gamma)$.
  \item The same result for marked vinyl graphs of depth $1$ holds and is proved in Appendix~\ref{sec:depth1}.
\end{enumerate}
\end{rmk}

\begin{cor}
  If $\Gamma_1$ and $\Gamma_2$ are two marked vinyl graphs of depth $0$ or $1$, which are equal when forgetting the base point, then $\egd(\Gamma_1)=\egd(\Gamma_2)$.
\end{cor}

\begin{cor}\label{cor:egd-baddigon}
  If $\Gamma$ and $\Gamma'$ are related by an outer digon, then $\egd(\Gamma')=\egd(\Gamma)$.
\end{cor}

\begin{rmk}
  Thanks to Proposition~\ref{prop:skeinrelation-to-unicity}, the expected dimension is characterized by the relation given in Remark~\ref{rmk:egd-skein-relation} and Lemma~\ref{lem:baddigon-egd}.
\end{rmk}

\section{$\gll_0$ link homology}
\label{sec:mathrmgl_0-link-homo}
We are now in position to defined a knot homology theory that we call $\gll_0$-homology. It categorifies the Alexander polynomial. The proof of invariance is contained in subsections~\ref{sec:invariance} and \ref{sec:moving-base-point}. 

\subsection{Definition}
\label{sec:definition}

In this subsection, we associate with every marked braid closure diagram a chain complex of graded $\QQ$-vector space. This follows the usual road. We take over notations of section~\ref{sec:rickard-complexes}. First, we consider a braid closure diagram $\beta$ with a base point. For every state of $\beta$ we define $\widetilde{C}_{\gll_0}(\beta_s):= \syf_0(\beta_s)\otimes q^{n_- - |s|}\QQ_s$.

For every state transition $(s,s')$ at $y$, we define $d_{s\to s'}:= \syf_0(F_{s\to s'})\otimes \bullet\wedge y$.
Finally we define $\widetilde{C}_{\gll_0}(\beta):= \bigoplus_{s \textrm{ state}} \widetilde{C}(\beta_s)$ and $d_\beta: \widetilde{C}_{\gll_0}(\beta) \to \widetilde{C}_{\gll_0}(\beta)$ by \[
d_\beta:= \sum_{(s,s') \textrm{ state transition}} d_{s\to s'}.\]
A schematic description of this construction with the different degree shifts is given on Figure~\ref{fig:rickgl0}.

\begin{figure}[ht]
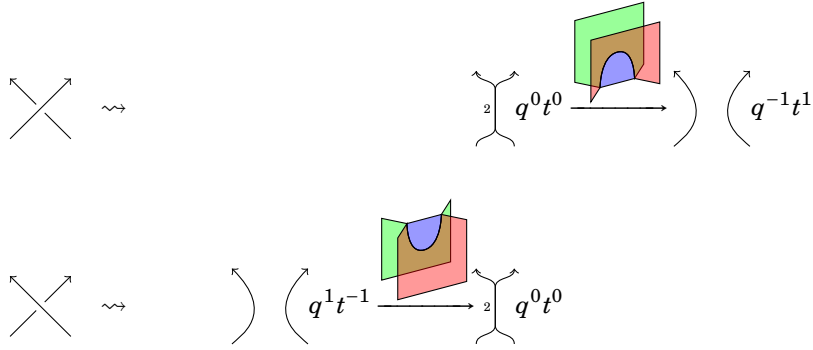

  \centering
  \begin{align*}
    \NB{\tikz[scale=0.4]{\input{\imagesfolder/alex_posXing}}} 
    \quad
&    \rightsquigarrow&
    \quad&
    \NB{\tikz[scale=0.5]{\input{\imagesfolder/alex_dumble}}}q^{0}t^0 
\xrightarrow{
    \NB{\tikz[scale=0.5]{\input{\imagesfolder/sym_splitfoam}}}}    
 \NB{\tikz[scale=0.5]{\input{\imagesfolder/alex_smoothing}}}q^{-1}t^1  \\
&\\
    \NB{\tikz[scale=0.4]{\input{\imagesfolder/alex_negXing}}} 
    \quad
&    \rightsquigarrow&
    \quad
    \NB{\tikz[scale=0.5]{\input{\imagesfolder/alex_smoothing}}}q^{1}t^{-1} 
\xrightarrow{
    \NB{\tikz[scale=0.5]{\input{\imagesfolder/sym_jointfoam}}}}    
& \NB{\tikz[scale=0.5]{\input{\imagesfolder/alex_dumble}}}q^{0}t^0  \\
  \end{align*}
  \caption{A schematic description of the complex $\widetilde{C}_{\gll_0}(\beta)$.}
  \label{fig:rickgl0}
\end{figure}

By standard arguments, we have:

  \begin{prop}
    \label{prop:hypercube2complex}
    The pair $(\widetilde{C}_{\gll_0}(\beta),d_\beta)$ is a chain complex of graded $\QQ$-vector space.
  \end{prop}

We now define the rectified version of this invariant:
\[
C_{\gll_0}(\beta) := \widetilde{C}_{\gll_0}\left(\overrightarrow{\beta}\right).
\]

  \begin{prop}
    The Euler characterisitic of the chain complex $({C}_{\gll_0}(\beta),d_\beta)$ is the Alexander polynomial of the link represented by $\beta$.
  \end{prop}
  \begin{proof}
    First, the Euler characteristic of the chain complex $({C}_{\gll_0}(\beta),d_\beta)$ is a link invariant. In order to prove this, it is enough to show that it is compatible with braid relations, Markov moves and with the moving the point. This follows from Section~\ref{sec:grad-dimens-sym_0g}. From the very definition of $C_{\gll_0}$, we obtain that $P:= \chi_q({C}_{\gll_0}(\bullet))$ satisfies the skein relation
    \[
P\left(
\NB{\tikz[scale =0.3]{

  \begin{scope}
    \draw[->] (+1,-1) -- (-1,+1);
    \fill[white] (0,0) circle (2mm);
    \draw[->] (-1,-1) -- (+1,+1);
  \end{scope}


}}
\right)
-
P\left(
\NB{\tikz[scale =0.3]{

  \begin{scope}
    \draw[->] (-1,-1) -- (+1,+1);
    \fill[white] (0,0) circle (2mm);
    \draw[->] (+1,-1) -- (-1,+1);
  \end{scope}


}}
\right)
=
(q- q^{-1})
P\left(
\NB{\tikz[scale =0.3]{

  \begin{scope}
    \draw[->] (+1,-1) .. controls + (-0.2, 0.2) and (-0.2, -0.2) .. (+1,+1);
     \draw[->] (-1,-1) .. controls + ( 0.2, 0.2) and ( 0.2, -0.2) .. (-1,+1);
  \end{scope}


}}
\right)
\quad \textrm{and} \quad
P\left(
\NB{\tikz[scale =0.4]{

  \begin{scope}
\draw[->] (0.5,0) arc (0:360:0.5);
  \end{scope}


}}
\right) = 1
    \]
    which characterizes the Alexander polynomial.
\end{proof}
    \begin{thm}\label{thm:main1}
    If a diagram $\beta$ represents a knot $K$, the homology of $(C_{\gll_0}(\beta), d_\beta)$ is a bigraded vector space which only depends on $K$ up to isomorphism.
  \end{thm}
The homology of $(C_{\gll_0}(\beta), d_\beta)$ is denoted by $\Hglo(\beta)$.
  \begin{proof}
    Thanks to Lemma~\ref{lem:rectified-markov}, we need to prove invariance of $\Hglo$ under:
    \begin{itemize}
    \item Isotopies for braid closure diagram with right base points. This is clear.
    \item Braid relations far from the base point. This follows from Proposition~\ref{prop:R2R3-inftyequivalence}.
    \item Stabilization as stated in the move (\ref{it:mv-stab}) of Proposition~\ref{prop:markov-base}. This is given by Lemma~\ref{lem:R1-far}.
    \item The local move \label{it:mv-2-R2} of Lemma \ref{lem:rectified-markov} (analog to Reidemeister II). This the purpose of Proposition~\ref{prop:move-base-point}.  
    \end{itemize}
  \end{proof}

  \begin{rmk}
    \label{rmk:tilde}
    Notice that except when $\beta$ is a braid closure diagram with right base point, we don't know if the two complexes $C_{\gll_0}(\beta)$ and $\widetilde{C}_{\gll_0}\left(\beta\right)$ have the same homology. We believe this is true. We denote the homology of $\widetilde{C}_{\gll_0}\left(\beta\right)$ by $\widetilde{H}_{\gll_0}\left(\beta\right)$. 
  \end{rmk}

\subsection{Invariance}
\label{sec:invariance}

This subsection is devoted to the proof of invariance of $H_{\gll_0}$ with respect to braid relations and the stabilization Markov move. When a braid relation is far from the base point, everything basically follow from the general construction and the proof of invariance of $\sll_N$-homology theory in the context of foams given by \cite{VAZPHD, MR2491657} see Proposition~\ref{prop:R2R3-inftyequivalence}.

\begin{lem}
  \label{lem:R1-far}
  \begin{enumerate}
\item \label{it:R2-1-1}
  Let $\beta_1$ and $\beta_2$ be two marked braid closure diagrams (see figure~\ref{fig:B1B2R1curl}) of level $k$ and $k+1$ respectively, which are related by a positive stabilization Markov move and such that the base point of $\beta_2$ is on its right-most strand (no matter where) and the base point of $\beta_1$ is on its right-most strands at the place where $\beta_2$ has a crossing. Then the bigraded complexes $C_{\gll_0} (\beta_1)$ and  $C_{\gll_0}(\beta_2)$ are isomorphic.
  \begin{figure}[ht]
    \centering
    \begin{tikzpicture}
      \input{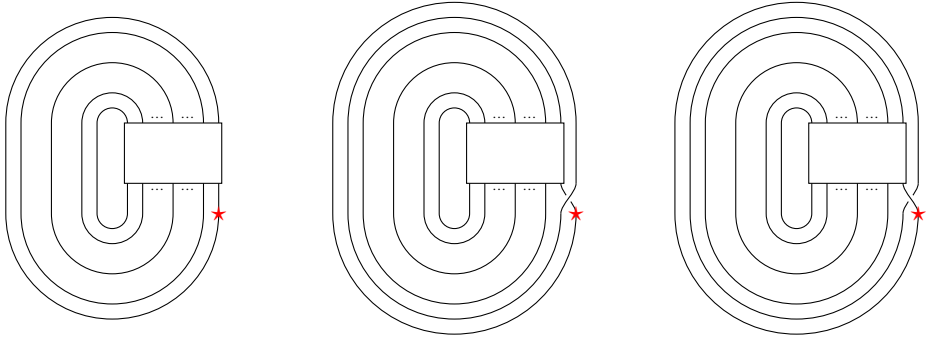}
    \end{tikzpicture}
    \caption{The braids closure diagrams $\beta_1$ (on the left), $\beta_2$ (on the middle) and $\beta_3$ of Lemma~\ref{lem:R1-far}.}
    \label{fig:B1B2R1curl}
  \end{figure}
\item \label{it:R2-1-2} Similarly, if $\beta_1$ and $\beta_3$ are 
two marked braid closure diagrams (see Figure~\ref{fig:B1B2R1curl}) of level $k$ and $k+1$ respectively, which are related by a negative stabilization Markov move and such that the base point of $\beta_3$ is on its right-most strand (no matter where) and the base point of $\beta_1$ is on its right-most strands at the place where $\beta_3$ has a crossing. Then the bigraded complexes $C_{\gll_0} (\beta_1)$ and  $C_{\gll_0}(\beta_3)$ are isomorphic as complexes.
\end{enumerate}
\end{lem}

\begin{proof} We only prove (\ref{it:R2-1-1}), the other case is similar. 
Recall that if a marked vinyl graph $\Gamma$ is not connected, then $\syf_0(\Gamma)= \{0\}$. Hence $C_{\gll_0}(\beta_2)$ is isomorphic to the complex associated with $\beta'_2$, where $\beta'_2$ is the knotted vinyl graph obtained from $\beta_2$ by replacing the rightmost crossing by its dumble resolution: 
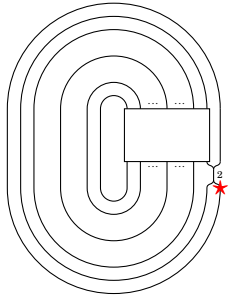
\[
\tikz{\input{\imagesfolder/alex_B2pR1}}
\]

We conclude by Corollary~\ref{cor:equivalence}: the isomorphism between $C_{\gll_0} (\beta'_2)$ and  $C_{\gll_0}(\beta_2)$ is given by the natural isomorphism $\psi$.
\end{proof}

\subsection{Moving the base point}
\label{sec:moving-base-point}

The aim of this subsection is to prove the following proposition:

\begin{prop}
  \label{prop:move-base-point}
  Let $\beta_1$, $\beta_2$ and $\beta_3$  be three marked braid closure diagrams which are identical except in a small ball containing the base point where they are given by the following triple of diagrams (the two strands are meant to be the rightmost ones):
\[\NB{
\begin{tikzpicture}[scale = 0.45] \input{\imagesfolder/alex_pairsB1B2-R2m}
\end{tikzpicture}}.
\] 
Then we have $\Hglo(\beta_1)=\wH(\beta_1) \simeq \wH(\beta_3) \simeq \wH(\beta_2) =\Hglo(\beta_2)$.
\end{prop}

The proof will use the follows lemma of homological algebra:

\begin{lem}
  \label{lem:homological-nonsense}
Let $(C_i, d_i)_{1\leq i \leq 4}$ be four complexes of graded $\QQ$-vector spaces. We consider $\pi_1: C_1 \to C_2$, $\pi_2: C_3 \to C_4$, $\iota_1 : C_1 \to C_3$ and $\iota_2: C_3 \to C_4$ four chain maps such that $\pi_1$ and $\pi_2$ are surjective and $\iota_1$ and $\iota_2$ are injective and that the square
\[
\ensuremath{\vcenter{\hbox{
\begin{tikzpicture}[yscale=-0.8, xscale =0.8, rotate =45]
  \node (A) at (0,2) {$C_1$};
  \node (B) at (2,2) {$C_3$};
  \node (C) at (0,0) {$C_2$};
  \node (D) at (2,0) {$C_4$};
  \draw[-to] (A)--(B) node [midway,left] {$\iota_1$};
  \draw[-to] (A)--(C) node [midway,left] {$\pi_1$};
  \draw[-to] (B)--(D) node [midway,right] {$\pi_2$};
  \draw[-to] (C)--(D) node [midway,right] {$\iota_2$};
\end{tikzpicture}
}}}
\]
anti-commutes. 
Then the homology of the total  complex 
\[(C', d') := 
\left(C_1\{-1\}\oplus C_2  \oplus C_3 \oplus C_4\{+1\},
  \begin{pmatrix}
    d_1     & 0   & 0   & 0 \\
    \pi_1   & d_2 & 0   & 0 \\
    \iota_1 &  0   & d_3 & 0 \\
    0       & \pi_2 & \iota_2   & d_4 
  \end{pmatrix} \right)
\] 
is isomorphic to the homology of  $\Ker \pi_2 / (\Im \iota_1 \cap \Ker \pi_2)$ seen as a subquotient of $C_3$.
\end{lem}
\begin{proof}

We consider  $C' =C_1\{-1\}\oplus C_2  \oplus C_3 \oplus C_4\{+1\}$ as a double complex. Let us denote by $d_H$ the differential $\pi_1 + \pi_2 + \iota_1 + \iota_2$ and $d_V$ the differential $d_1 + d_2 + d_3 + d_4$. We speak of horizontal and vertical degrees: these are the homological degrees with respect to $d_H$ and $d_V$. The homology $H(C', d_H)$ is concentrated in one horizontal degree. Therefore the spectral sequence with $H(C', d_H)$ on the first page and computing the homology of the total complex $C'$ converges on the second page and on the second page, there is no extra-differentials. Hence it enough to show that: 
\[\left(H(C', d_H), d_V\right)\simeq \left(\Ker \pi_2 / (\Im \iota_1 \cap \Ker \pi_2), d_V\right).\]
We will exhibit a pair of mutually inverse isomorphisms:
\[
\begin{array}{crcl}
\varphi  \colon\thinspace &  \Ker(\pi_2 + \iota_2)/ (\Im (\iota_1 + \pi_1)) & \to & \Ker \pi_2 / (\Im \iota_1 \cap \Ker \pi_2) \\
\psi  \colon\thinspace & \Ker \pi_2 / (\Im \iota_1 \cap \Ker \pi_2)   & \to & \Ker(\pi_2 + \iota_2)/ (\Im (\iota_1 + \pi_2)) \\
\end{array}
\]

Let $x_2 + x_3$ be an element of $\Ker(\pi_2 + \iota_2)$ (with $x_2$ in $C_2$ and $x_3$ in $C_3$). 
Let $x_1$ in $C_1$, such that $\pi_1(x_1) = x_2$.  We have $\pi_2(x_3 - \iota_1(x_1)) = 0$. We define $\varphi([x_2 + x_3]):= [x_3 - \iota(x_1)]$. Let $x_3$ be an element of $\Ker(\pi_2)$.  We define $\psi([x_3]):= [0_{C_2}+x_3]$. The square bracket represent equivalence classes; their meanings differ according to which quotient we consider.

We need to prove that:
\begin{itemize}
\item These maps are well-defined.
\item They are inverses from each other.
\item They are chain maps.
\end{itemize}

\emph{Well-definiteness.} The definition of $\varphi([x_2+ x_3])$ does not depend on the choice of $x_1$, because in the end we mod out by $\Im \iota_1$. It does not depends on the element $x_2 + x_3$ representing $[x_2 + x_3]$. Indeed, suppose $x_2 + x_3 = \pi_1(x_1) + \iota_1(x_1)$, then $\varphi([x_2 + x_3]) = [0]$.  

The definition of $\psi([x_3])$ does not depend on $x_3$ because if $x_3$ is in $\Im \iota_1 \cap \Ker \pi_2$, $x_3 + 0_{C_2}$ is in $\Im(\iota_1 + \pi_1)$. \\

\emph{Mutual inverses.} We clearly have: $\varphi \circ \psi= \Id$. The composition in the other direction follows from the fact that $[x_3 +x_2] = [x_3 - \iota_1(x_1)]$.\\

\emph{Chain maps.} We know that $\varphi$ and $\psi$ are mutually inverse linear maps, it is enough to show that $\psi$ is a chain map. This is clear. 
\end{proof}

\begin{lem}
  \label{lem:ker-singsaddle-digon}
  Let $\Gamma_1$ and $\Gamma_2$ be two marked vinyl graphs of depth $0$ which are identical except in a small ball containing the base point where they are given by the following pair of diagrams:
  
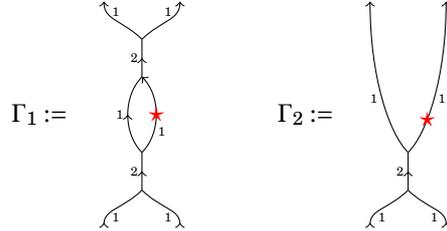
\[
    \tikz[scale = 1]{\input{\imagesfolder/alex_pairs-digon-saddle}}
\]
Let $s$ be the splitting singular saddle from $\Gamma_1$ to $\Gamma_2$, and let $f$ be the foam which adds a dot on the top right edge of $\Gamma_1$ (\ie{}multiplies the decoration of the corresponding facet by $X$). We have $\Ker \syf_0(s) = \Im \syf_0(f)$.
\end{lem}

\begin{proof} 
  First, $\Im \syf_0(f)$ is clearly a subspace of $\Ker \syf_0(s)$ (see for instance  Remark~\ref{rmk:Xk-is-zero-other-UC}
(\ref{it:too-high-deg})). 
Hence it is enough to prove that $\dim \Im \syf_0(f) \geq \dim \Ker \syf_0(s)$. The map $\syf_0(s):
\syf_0(\Gamma_1) \to \syf_0(\Gamma_2)$ is surjective (Remark~\ref{rmk:syf00}) and $\dim \syf_0 (\Gamma_1) = 2 \dim \syf_0(\Gamma_2)$ (see Propositions~\ref{prop:egd-q=1} and \ref{prop:egd-depth0}). This implies that $\dim \left(\Ker \syf_0(s) \right) = \dim \syf_0(\Gamma_2)$.  

We will construct a surjective map from $\Im(\syf_0(f))$ to $\syf_0( \Gamma_2')$ where $\Gamma_2'$ is the marked vinyl graph which is identical to $\Gamma_1$ except in a small ball where:
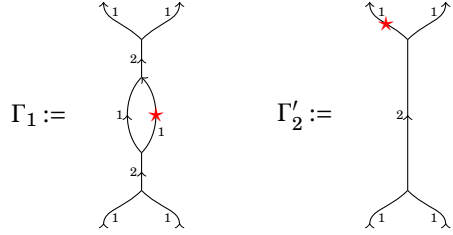
\[
    \tikz[scale = 1]{\input{\imagesfolder/alex_pairs-digon-saddle-2}}
\]
  
Since by Proposition~\ref{prop:gd-depth1} we have $\dim \syf_0(\Gamma_2') = \dim \syf(\Gamma_2)$, this will achieve the proof. 

The map we construct is induced by the foam $c: \Gamma_1 \to \Gamma_2'$ which is the identity everywhere except where $\Gamma_1$ and $\Gamma_2'$ differs where it is given by:
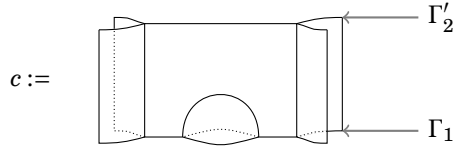
\[
\tikz[yscale= 0.8]{\input{\imagesfolder/alex_digon-cap}}
\]

This is not a morphism in the category of marked vinyl graphs. Hence, we need to go back to the definition of $\syf_0$. Recall from Lemma~\ref{lem:eq-dots-on-marked-distributed} that  $\syf_0(\Gamma_1)$ is the image by $\syf_1(X_r(\Gamma_1))$ of $\syf_1(\Gamma_1)$, where $X_r$ is the foam which adds a maximal elementary symmetric polynomial on each edge of $\Gamma_1$ intersecting a ray $r$ passing by the base point  except on the edge carrying the base point (see Figure~\ref{fig:rays}). The map $\syf_1(f) : \syf_1(\Gamma_1) \to \syf(\Gamma_2')$ clearly maps $\Im( \syf_1(f\circ X_r(\Gamma_1)))$ in $\Im( \syf_1( X_{r'}(\Gamma_2')))$. It is surjective, since any elements of $\Im( \syf_1( X_{r'}(\Gamma_2')))$ is a linear combination of disk-like foams and because of the local relation for the $\infty$-equivalence depicted in Figure~\ref{fig:bubble-removal}.
\begin{figure}[ht]
  \centering
    \tikz[scale = 0.5]{\input{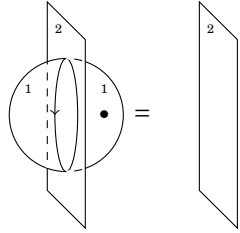}}
  \caption{Bubble removal: these two foams are $\infty$-equivalent.}
  \label{fig:bubble-removal}
\end{figure}
\begin{figure}[ht]
  \centering
  \tikz[scale = 1]{\input{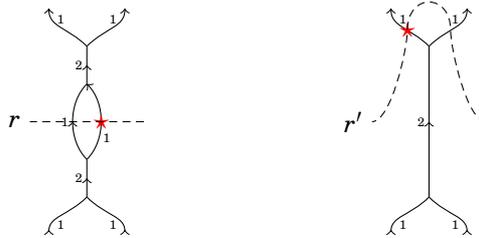}}
  \caption{Rays $r$ and $r'$ in $\Gamma_1$ and $\Gamma'_2$.}
  \label{fig:rays}
\end{figure}
\end{proof}

\begin{cor}
  \label{cor:special-maps}
  With the notations of Lemma~\ref{lem:ker-singsaddle-digon}, the foams 
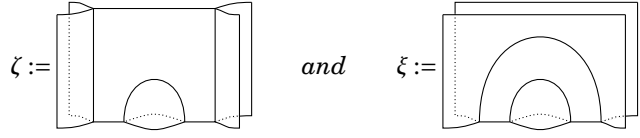
\[
\zeta:=\NB{\tikz[scale = 0.8]{\input{\imagesfolder/alex_digon-cap-2}}}
\quad \quad \textrm{and} \quad  \quad 
\xi:=\NB{\tikz[scale = 0.8]{\input{\imagesfolder/alex_digon-cap-saddle}}}
\] 
induce an isomorphism $\Ker \syf_0(s) \to \syf_0(\Gamma'_2)$ and a surjective map: $\Ker \syf_0(s) \to \syf_0(\Gamma''_2)$ with $\Gamma''_2$ locally defined by the following diagram:
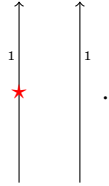
\[
\NB{\tikz[scale =0.4]{  \begin{scope}[xshift = 0cm]
    \draw[->] (-1, -3) -- (-1, 3) node[pos = 0.5, red] {${\star}$} node[pos = 0.7, left, scale =0.5] {$1$}; 
    \draw[->] (+1, -3) -- ( +1, 3) node[pos = 0.7, right, scale =0.5] {$1$};
  \end{scope}
}
}.
\] 
\end{cor}

\begin{proof}[Proof of Proposition~\ref{prop:move-base-point}]
  We only prove the statement for $\beta_1$ and $\beta_3$. The proof for $\beta_2$ and $\beta_3$ is similar. 
In the hypercube of resolutions of $C(\beta_1)$, we flatten all the differentials but the ones coming from the two crossings in the neighborhood of the base point. Thanks to Remark~\ref{rmk:syf00}, we are precisely in the hypothesis of Lemma~\ref{lem:homological-nonsense}:
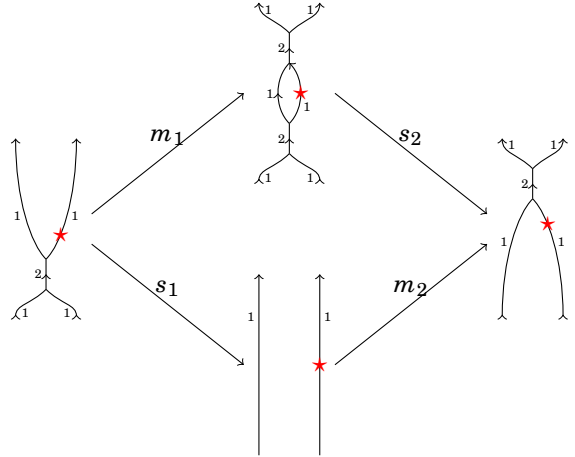
\[
\NB{\tikz[scale =0.8]{\input{\imagesfolder/alex_square-R2m-diagrams}}
},
\]
where $s_\bullet$ (\resp$m_\bullet$) denote the morphisms induced by the split (\resp{}merge) singular saddle on the bottom (\resp{}top) of the diagrams.

Hence $H_{\gll_0}(\beta_1)$ is isomorphic to the homology of $\Ker s_2 / (\Im m_1 \cap \Ker s_2)$. To conclude the proof we show that the complex $\Ker s_2 / (\Im m_1 \cap \Ker s_2)$ is isomorphic to $C(\beta_3)$ as a bi-graded complex.

For each resolution of the other crossings,
we can use the foam $\xi$ of Corollary~\ref{cor:special-maps}. Taking all the induced maps together gives a surjective map (still denoted $\xi$) from $\Ker s_2$ to $C(\beta_3)$. It is a chain map since the differentials and the maps are given by foams which commute since their supports are disjoint.   

It is clear that $\Im m_1 \cap \Ker s_2$ is in the kernel of $\xi$. Indeed an element $x$ of $\Im m_1 \cap \Ker s_2$, must be in the kernel of $s_1$, but at the foam level, $\xi \circ m_1 $ is isotopic to $s_2$ (see Figure~\ref{fig:isotopic-xim1-s2}).
\begin{figure}[ht]
  \centering
  \begin{tikzpicture}
    \input{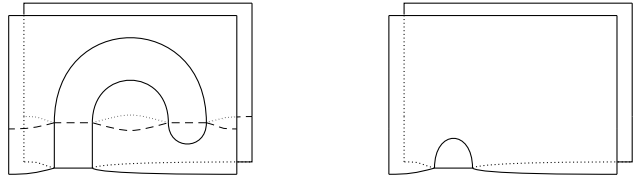}
  \end{tikzpicture}
  \caption{The foams $\xi \circ m_1$ and $s_2$.}
  \label{fig:isotopic-xim1-s2}
\end{figure}
Hence $\xi$ induces a surjective chain map from $\Ker s_2 / (\Im m_1 \cap \Ker s_2)$ to $C(\beta_3)$. We claim that for each resolution the corresponding space in  $\Ker s_2 / (\Im m_1 \cap \Ker s_2)$ and $C(\beta_3)$ have the same dimension. Hence $\xi$ induces an isomoprhism of complexes. The claim follows by computing (for each resolution) the Euler characteristic of the complex of length three whose homology is $\Ker s_2 / (\Im m_1 \cap \Ker s_2)$.
\end{proof}

\section{An algebraic approach}
\label{sec:an-algebr-appr}
In this section, we relate the construction of Section~\ref{sec:vinylgraphBP} to the Hochschild homology of (reduced) Soergel bimodules. We refer for instance to \cite{MR3447099, MR2339573, 2016arXiv160202769W}  for deeper description of the relationship between Soergel bimodules and link homologies.

\begin{notation}
  \label{not:polynomial-algebras}
  
  Let $\listk{k}= (k_1, \dots, k_l)$ be a finite sequence of
  positive
  integers which sum up to $k$. 
  The group $\prod_{i=1}^l {S}_{k_i}$ is denoted by
  ${{S}_{\listk{k}}}$.  We define the algebra $A_{\listk{k}}$: 
\[ A_{\listk{k}}:=\QQ[x_1, \dots,   x_k]^{{S}_{\listk{k}}}. \]
The indeterminates $x_\bullet$ are homogeneous of degree $2$. If $\listk{k}= (k)$ (that is if $\listk{k}$ has length $1$), we write $A_k$ instead of $A_{(k)}$. The list of length $k$ with only $1$s, is denoted $1^k$. In particular, we have $A_{1^k}= \QQ[x_1, \dots x_k]$.

 We denote by $A'_{1^k}$ the sub-algebra of $A_{1^k}$ generated by $y_i = x_i - x_k$ for $i\in\{1, \dots k-1\}$ and for any $\listk{k}$,  we set $A'_{\listk{k}} = A_{\listk{k}} \cap A'_{1^k}$.

\end{notation}

\begin{dfn}
  \label{dfn:soergel-bimodules}
  Let $\Gamma$ be a braid-like $(\listk{k_1},\listk{k_0})$-MOY graph. If $\Gamma$ has no trivalent vertices, we have $\listk{k_0}= \listk{k_1}$ and we define $\BS(\Gamma)$ (\resp$\BS'(\Gamma)$) to be equal to $A_{\listk{k_0}}$ (\resp$A'_{\listk{k_0}}$) as a $(A_{\listk{k_0}},A_{\listk{k_0}})$-module
(\resp$(A'_{\listk{k_0}},A'_{\listk{k_0}})$-module). If $\Gamma$ has only one trivalent vertex (which is supposed to be of type $(a,b,a+b)$), then:
  \begin{itemize}
  \item if the length of $\listk{k_1}$ is equal to the length of $\listk{k_0}$ plus 1, we define $\BS(\Gamma)$ (\resp$\BS'(\Gamma)$) to be $A_{\listk{k_1}}q^{-ab/2}$ (\resp$A'_{\listk{k_1}}q^{-ab/2}$) as a $(A_{\listk{k_1}},A_{\listk{k_0}})$-module;
  \item if the length of $\listk{k_0}$ is equal to the length of $\listk{k_1}$ plus 1, we define $\BS(\Gamma)$ to be $A_{\listk{k_0}}q^{-ab/2}$ as a $(A_{\listk{k_1}},A_{\listk{k_0}})$-module
(\resp{}as a $(A'_{\listk{k_1}},A'_{\listk{k_0}})$-module).
    
  \end{itemize}
If $\Gamma$ has more than one trivalent vertex, if necessary we perturb\footnote{In other words, we choose an ambient isotopy of $[0,1]^2$ such that the images of the trivalent vertices of $\Gamma$ have distinct $y$-coordinate} $\Gamma$ to see it as a composition:
\[
\Gamma = \Gamma_t \circ_{\listk{k^t}} \Gamma_{t-1} \circ_{\listk{k^{t-1}}} \cdots \circ_{\listk{k^2}}  \Gamma_1 \circ_{\listk{k^1}} \Gamma_0 
\]
where $\Gamma_i$ is a braid-like $(\listk{k^{i+1}},\listk{k^{i}})$-MOY graph with one trivalent vertex, for all $i$ in $\{0,\dots, t\}$. The symbols $\circ_{\listk{k^{i}}}$ means that  $\Gamma_i$ and $\Gamma_{i-1}$ are glued along $\listk{k^i}$. We have $\listk{k^{0}}= \listk{k_{0}}$ and $\listk{k^{t+1}} = \listk{k_{1}}$. We define
\begin{align*}
  \BS(\Gamma)&:= \BS(\Gamma_t) \otimes_{A_{\listk{k^t}}} \BS(\Gamma_{t-1}) \otimes_{A_{\listk{k^{t-1}}}} \cdots \otimes_{A_{\listk{k^2}}}  \BS(\Gamma_1) \otimes_{A_{\listk{k^1}}} \BS(\Gamma_0)\\
  (\textrm{\resp}\quad
  \BS'(\Gamma)&:= \BS'(\Gamma_t) \otimes_{A'_{\listk{k^t}}} \BS'(\Gamma_{t-1}) \otimes_{A'_{\listk{k^{t-1}}}} \cdots \otimes_{A'_{\listk{k^2}}}  \BS'(\Gamma_1) \otimes_{A'_{\listk{k^1}}} \BS'(\Gamma_0) )
\end{align*}
The space $\BS(\Gamma)$ (\resp$\BS'(\Gamma)$) has a natural structure of $(A_{\listk{k_1}},A_{\listk{k_0}})$-module (\resp$(A'_{\listk{k_1}},A'_{\listk{k_0}})$-module). It is called the \emph{Soergel bimodule associated with $\Gamma$} (\resp\emph{reduced Soergel bimodule associated with $\Gamma$}). 
\end{dfn}

\begin{rmk}\label{rmk:reduced-iso}
  For any braid-like MOY graph, we have a canonical isomorphism:
  \[
\BS(\Gamma) \simeq \BS'(\Gamma)\otimes_{\QQ}\QQ[E_1],
\]
where $E_1$ is the first symmetric elementary polynomial in the variables $x_\bullet$.
\end{rmk}

\begin{prop}[{\cite[Proposition 5.18]{RW2}}]\label{prop:HH-Finfty}
  Let $\Gamma$ be a braid-like $(\listk{k},\listk{k})$-MOY graph, then $\HH_0(A_{\listk{k}}, \BS(\Gamma))$ is canonically isomorphic to $\F_\infty(\widehat{\Gamma})$ (see proof of Proposition~\ref{prop:HHSoergel-with-deco} for notations).
\end{prop}

\begin{rmk}\label{rmk:reduced-syb}
  \begin{enumerate}
\item
  Because of Remark~\ref{rmk:reduced-iso} the space $\HH_0(A'_{\listk{k}}, \BS'(\Gamma))$ is naturally a subspace of $\HH_0(A_{\listk{k}}, \BS(\Gamma))$. 
  \item It follows from Proposition~\ref{prop:HH-Finfty} that there is a natural surjective map $\pi$ from $\HH_0(A_{\listk{k}}, \BS(\Gamma))$ to $\syf_1(\widehat{\Gamma})$. 
\end{enumerate}
\end{rmk}

\begin{lem}\label{lem:pi-surj}
  The restriction of $\pi$ to  $\HH_0(A'_{\listk{k}}, \BS'(\Gamma))$ is surjective. 
\end{lem}
\begin{proof}
  This follows from $\pi$ mapping $E_1$, the symmetric polynomial in the variables $x_\bullet$ onto $0$.
\end{proof}

On $A'_{1^k}= \QQ[y_1,\dots ,y_{k-1}]$, we consider the endomorphism
\[
  \begin{array}{crcl}
  D^0   \colon\thinspace & A'_{1^k} & \to & A'_{1^k} \\
  & P(y_1,\dots,y_{k-1})  &\mapsto & \sum_{i=1}^{k-1} \partial_{y_i}P(y_1, \dots, y_{k-1}).
  \end{array}
\]

If $\listk{k}$ is a list of length $\ell$ such that $k_\ell =1$, $D^0$ induces an endomorphism of $A'_{\listk{k}}$. 
As explained in \cite[Appendix B]{RW2}, we can use $D^0$ to define a differential $d_0$ on $\HH_\bullet(A'_{\listk{k}}, M)$ and endow this space with a structure of chain complex. This construction is functorial in $M$.

Let $\Gamma$ be a braid-like $(\listk{k},\listk{k'})$-MOY graph with only one vertex and suppose that both $\listk{k}$ and $\listk{k'}$ end with a $1$. If $M$ is a $(A'_{\listk{k'}},A'_{\listk{k}})$-module, then one can show (analog to \cite[Lemma B.6]{RW2} ) that the complexes \[\left(\HH_\bullet(A'_{\listk{k}}, \BS'(\Gamma)\otimes_{A'_{\listk{k'}}}M), d_0 \right)\quad
\textrm{and} \quad 
\left(\HH_\bullet(A'_{\listk{k'}}, M \otimes_{A'_{\listk{k}}}\BS'(\Gamma)), d_0 \right)\]
are homotopy equivalent.

This gives immediately:

\begin{lem}
 Let $\Gamma$ be a braid-like $(\listk{k},\listk{k})$-MOY graph with $\listk{k}$ a sequence finishing by $1$. Then the complex $(\HH(A'_{\listk{k}}, \BS'(\Gamma)), d_0)$ depends only on $\widehat{\Gamma'}$, the marked vinyl graph obtained from $\Gamma$ by putting a base point on its bottom right edge and closing up. 
\end{lem}

We denote the homology of $(\HH(A'_{\listk{k}}, \BS'(\Gamma)), d_0)$ by $H^{d_0}_\bullet(\widehat{\Gamma'})$. Because of Proposition~\ref{prop:HH-Finfty} and Remark~\ref{rmk:reduced-syb}, 
$H^{d_0}_0$ can be promoted to a functor from the category of marked vinyl foams to the category of graded vector space.

\begin{prop}
  Let $\Gamma$ be a marked vinyl graph with right base point, then for all $i> 0$, $H^{d_0}_i(\Gamma) =0$. Moreover the functors $H^{d_0}_0$ and  $\syf_0$ are isomorphic. 
\end{prop}

\begin{proof}
We first deal with the statement about $H^{d_0}_i$ for $i>0$. The same strategy as in proof of Lemma~\ref{lem:baddigon-egd} and further reduction of the components with no base points show that it is enough to check this on graphs of the form given in Figure \ref{fig:gamma-listk}.
  \begin{figure}[ht]
    \centering
    \NB{\tikz[scale =1.2]{\input{\imagesfolder/alex_Gamma-listk}}}
    
    \caption{The graph $\Gamma_{\listk{k}}$ (all $k_j\geq 1$ for all $j$).}
    \label{fig:gamma-listk}
  \end{figure}
If $l\geq 2$, the statement is obvious (and even $H_0^{d_0}(\Gamma_{\listk{k}})=0$). If $l= 1$, this is an easy computation.

  We need to define a natural transformation $(\psi_\Gamma)_{\Gamma \in \Foam_{k,0}}$. We start with one special marked graph. Let us consider
  \[
    \Gamma_k := \NB
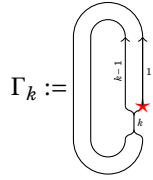
{\tikz[scale = 1.3]{\input{\imagesfolder/alex_Gamma0}}}
  \]
  One has $H_0^{d_0}(\Gamma_k) \simeq \QQ$ and $\syf_0(\Gamma_k)\simeq \QQ$. Let us fix once and for all an isomorphism $f$ between $H_0^{d_0}(\Gamma_0)$ and $\syf_0(\Gamma_0)$. Let $\Gamma$ be a right marked vinyl graph. Every elements of $H_0^{d_0}(\Gamma)$ is the image of an element $x$ of $H_0^{d_0}(\Gamma_k)$ by $H_0^{d_0}(F)$ for $F$ a linear combination of marked vinyl $(\Gamma,\Gamma_k)$-foam. We define $\psi_\Gamma(x)$ to be equal to $\syf_0(F)(f(x))$. The universal construction ensures that this map is well-defined and makes all diagrams commutative. We now should check that these maps are isomorphisms. Just like before, it is enough to  check this on graphs of the form given by Figure~\ref{fig:gamma-listk}. If $l \geq 2$. The spaces $H_0^{d_0}(\Gamma_{\listk{k}})$ and $\syf(\Gamma_{\listk{k}})$ are both equal to $0$. If $l =1$, then $\Gamma_{\listk{k}}= \Gamma_0$, and the result is trivial by construction.
\end{proof}

From the previous proposition, we deduce that we have another way to compute the homology $\Hglo(K)$ of a knot. On the one hand, the algebraic versions of zip $Z$ and unzip $U$ morphisms described at the end of Section \ref{sec:symgl1} are compatible with the functor $\syf_0$ and induce precisely $\syf_0(Z)$ and $\syf_0(U)$, respectively, once $\syf_0$ is applied. On the other hand, these morphisms can be seen as morphims of Soergel bimodules and are exactly the ones used to define the triply graded link homology of Khovanov--Rozansky \cite{MR2421131}. Moreover Khovanov explained in \cite{MR2339573} how they are compatible with the reduced Soergel bimodules and can also be used to compute the reduced triply graded link homology. Given a braid $\beta$ with right base point and whose closure is a knot $K$, denote by $\HH(\beta)$ the complex of bigraded $\QQ$-vector spaces obtained by taking the Hochschild homology of the reduced Soergel bimodules associated to the braid-like graphs on the vertices of the hypercube of resolution as described above. By abuse of notations,  still denote by $d_{\beta}$ the induced differential (by the zip's and unzip's) on this complex. Moreover $d_0$ also induces a differential on this complex which anti-commutes with $d_\beta$. We refer the reader for a more detailed discussion to \cite{RW2} and a very similar construction in the unreduced setting.  We have that $\HHH(K)=H(\HH(\beta), d_\beta)$. We can also consider $H(H(\HH(\beta),d_0),d_\beta)$ which since $d_0$ concentrates the homology in a single Hoschschild grading is nothing else than $H(\HH(\beta), d_0+d_\beta)$. By the previous proposition and the discussion on the zip and unzip morphisms we have also that $H(H(\HH(\beta),d_0),d_\beta)$ is exactly $\Hglo(K)$. We hence obtain the following theorem:

\begin{thm}
  \label{thm:spectral-seq}
  There exists a spectral sequence $E_{\HHH\to\gll_0}$ with the reduced triply graded homology of a knot $K$ on a first page and converging to the knot homology $\Hglo(K)$.
\end{thm}

\begin{rmk}
  \label{rmk:notclear}
  It is not clear that all the pages $E_{\HHH\to\gll_0}$ are knot invariants. This would require a proof of the independence of $H_{\gll_0}$ with respect to the base point compatible with $d_0$.   
\end{rmk}

\section{Generalization of the theory}
\label{sec:colored-version}
This section is concerned with possible generalizations of our construction. We described four possible directions. They should all be mutually compatible. The levels of difficulty of these generalizations differ.   

\subsection{Changing the coefficients}
\label{sec:chang-coeff}

In \cite{RW2}, we suggested that the construction of the symmetric $\gll_N$-homology ($N\geq 1$), might be carried out over the integers. However, the proof of invariance under the second Markov move (see \cite[Section 6.2]{RW2}) heavily depends on the description using Soergel bimodules and is inspired by an argument of Rouquier~\cite{1203.5065} which requires that $2$ is invertible. We believe that having a foamy proof of invariance would solve this technical issue. But such a proof eluded us so far.

The situation for the $\gll_0$-homology is different. Indeed, even if we mentioned some relations with reduced Soergel bimodules, no proof requires them. Hence, we believe that a definition of everything could be, in principle, carried over the integers. Some technical problems may appear anyway. This is because many arguments relies on the dimension of some spaces. If we work over $\ZZ$, the dimension (or rank) is not enough to characterize a space.  This might cause an issue for the proof of Proposition~\ref{prop:move-base-point}.

Let us mention that such an integral theory would be easier to compare with knot Floer homology, in particular with the approach developed by Ozsv\'ath--Szab\'o~\cite{MR2574747}, Manolescu~\cite{Manolescu} and Gilmore~\cite{Gilmore}.

\subsection{Equivariance}
\label{sec:equivariance}

In \cite{RW2}, the symmetric $\gll_N$-homology is defined in an equivariant setting. This means that the homology associated with a link is not a bigraded $\QQ$-vector space but rather a bigraded free $\QQ[T_1, \dots T_N]^{S_N}$-module. Of course, strictly speaking, when setting $N=0$, one would obtain only $\QQ$-vector space. However the $\gll_0$-homology is defined as a reduction or stabilization of the $\gll_1$-homology, hence we could reasonably hope for having a theory over $\QQ[T_1]$. The existence of minus version of Heegaard--Floer homology corroborates this hope.

Using the equivariant version of the symmetric $\gll_1$-homology, one could define an equivariant $\gll_0$-homology. The different equivalent definitions given in Section \ref{sec:gl0-evals} may not be equivalent in the equivariant context anymore. Presumably, one should use the definition given by the $\star$-evaluation (see Definition~\ref{df:star-eval}).

\subsection{Colored homology}
\label{sec:colored-homology}

In this paper, we are only concerned with uncolored invariant. This means that every strand of the braids we look at has label $1$. The vocabulary of vinyl graphs enables to speak about more colors. There is a standard way to deal with the braiding in the colored case (see for instance \cite{FunctorialitySLN}). One should then give the meaning of the base point on an edge with a label different from $1$. We believe that replacing a  base point on an edge of label $\ell$ by a small digon of type $(\ell, \ell-1, 1)$ with a base point as depicted on Figure~\ref{fig:marked-big} would give a reasonable definition.

\begin{figure}[ht]
  \centering
  \begin{tikzpicture}[scale= 1.2]
    \draw[->] (0,0) -- (0,1.2) node[pos= 0.3, right] {$\scriptstyle{\ell}$} node[red, pos =0.5] {$\star$};
    \node at (1.3, 0.6) {$\rightsquigarrow$};
    \draw[->-] (3,0) -- (3,0.3) node[pos = 0.5, right] {$\scriptstyle{\ell}$};
    \draw[->] (3,0.3).. controls +(0.2,0.2) and +(0.2, -0.2) .. (3,0.9) node[right, midway]  {$\scriptstyle{1}$} node[red, midway] {$\star$};
    \draw[->-] (3,0.3).. controls +(-0.2,0.2) and +(-0.2, -0.2) .. (3,0.9) node[left, midway]  {$\scriptstyle{\ell -1}$};
    \draw[->] (3,0.9) -- (3,1.2) node[pos = 0.5, right] {$\scriptstyle{\ell}$};
  \end{tikzpicture}
  \caption{Transforming a base point on an edge of label $\ell$.}
  \label{fig:marked-big}
\end{figure}
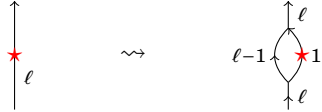

However, the proof of invariance under moving the base point would then fail. The proof relies on the fact that provided the base point is close enough from the right-most strand, then the dimension of the space associated with a marked vinyl graph is known not to depend on the base point. Such a statement is not known when the base point is arbitrary. This is essentially the content of Conjecture~\ref{cjc:expected-graded-dimension}.

Assume that such a construction would work. It is reasonable to wonder what it would categorify. From the very definition, we would have that it is the colored HOMFLY-PT polynomial specialized at $a=1$. The HOMFLY-PT polynomial at $a=1$ of a knot with color $n$ is equal to the Alexander polynomial at $q^n$. This follows from the work of Morton and Lukac~\cite{LukacMorton} and a classical formula of Torres~\cite{Torres} for the Alexander polynomial of cablings. This fact was explained to us by Hugh Morton.  

\begin{cjc}
  The construction described above yields a knot invariant categorifying the Alexander polynomial.
\end{cjc}

If this were true, the colored $\gll_0$-homology would give an infinite family of categorifications of the Alexander polynomial. Then it would be interesting to know if all these categorifications are equivalent or not.

\subsection{Links}
\label{sec:links}

The theory described in this paper is only concerned with knots: we need a base point and we can only move the base point along a given component. Hence it could be extended for free to links with a marked component. However, we believe that the construction is independent of the marked component.

\section{A few examples}
\label{sec:few-examples}
In this section, we provide some examples of computations.
We start with the trefoil and compute the homology of all $(2, 2n+1)$ torus knots.

We realize the right-handed trefoil as the closure of $\sigma_1^3$ in the braid group on two strands. The hypercube of resolution is given in Figure~\ref{fig:trefoil-cube}.
\begin{figure}[ht]
  \centering
  \tikz{\input{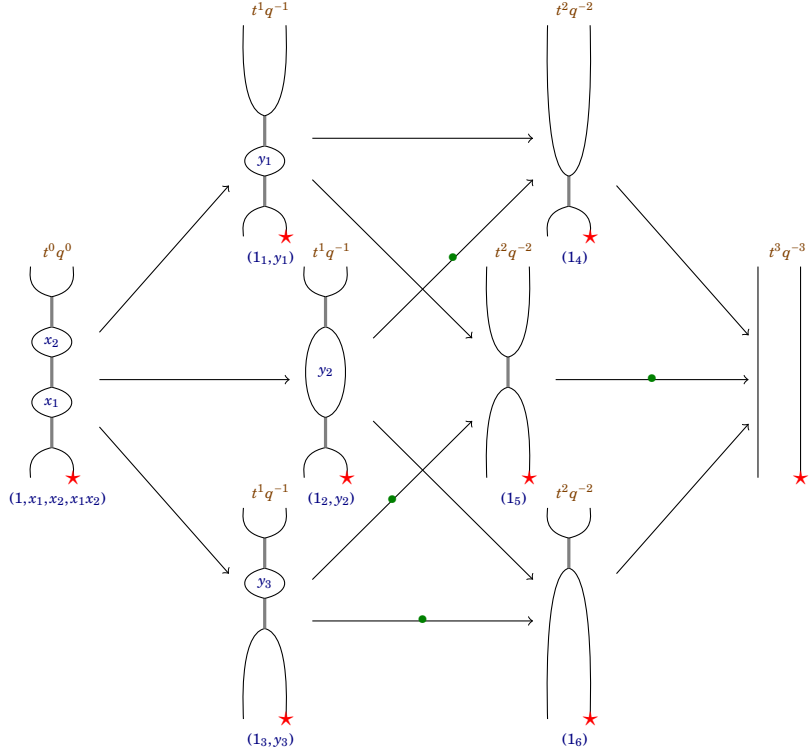}}
  \caption{Hypercube of resolutions for the trefoil. For simplicity we did not close the diagram in the picture. The green dots on the arrows indicate that the maps has to be multiplied by $-1$. Below each diagram, we give a basis of the space associated with this diagram. The $1_\bullet$ represent the trivially decorated tree-like foam. The $x_\bullet$ and $y_\bullet$ means decoration by a dot on the facet on the right-hand side of the corresponding digon.}
  \label{fig:trefoil-cube}
\end{figure}

Note that in the hypercube, the right-most diagram is not connected hence it is mapped to $0$ by $\syf_0$. 
We can write explicitly the matrices of the differential in the bases given in Figure~\ref{fig:trefoil-cube}:
 \[ M:=
    \begin{blockarray}{r c c c c}
                            &   1   &   x_1   &   x_2   &   x_1x_2   \\
    \begin{block}{r (c c c c)}
      1_1 & 1 & 0 & 0 & 0 \\
      1_2 & 1 & 0 & 0 & 0 \\
      1_3 & 1 & 0 & 0 & 0 \\
      y_1 & 0 & 1  & 0  & 0 \\
      y_2 & 0 & 1  & 1  & 0 \\
      y_3 & 0 & 0  & 1  & 0 \\
    \end{block}
  \end{blockarray}
  \quad \textrm{ and }
  \quad
  N:=
      \begin{blockarray}{r c c c c c c}
                            &   1_1   &   1_2   &   1_3   &   y_1 & y_2 & y_3   \\
    \begin{block}{r (c c c c c c)}
      1_4 & 1  & -1  &  & 0 & 0 & 0 \\
      1_5 & 1 &  0 & -1 & 0 & 0 & 0 \\
      1_6 & 0 & 1  & -1 & 0 & 0 & 0 \\
    \end{block}
  \end{blockarray}\,.
\]
From these matrices, we immediatly deduce that the $\gll_0$-homology of the right-handed trefoil is given by
\[
\Hglo(3_1)_{i,j} =
\begin{cases}
  \QQ & \textrm{if $(i,j) \in \{ (0, 2), (1,0), (2,-2),  \}$, } \\
  0 & \textrm{otherwise.}
\end{cases}
\]
where $i$ is the homological grading and $j$ is the quantum grading. This means that Poincaré polynomial of the $\gll_0$-homology of the right handed trefoil is given by:
\[
P(H_{\gll_0}(3_1)) (t,q) := q^2 + t + t^2 q ^{-2}
\]

We can deduce the value of the Poincaré polynomial of any 2-strands torus knot. Indeed one can show the following induction formula for $n \geq 1$:
\[
P(H_{\gll_0}(T(2, 2n+1))) (t,q) = q^{2n}+ t^{-1}q^{2n-2} +  t^2q^{-2} P(H_{\gll_0}(T(2, 2n-1))) (q,t).
\]
We obtain
\[
P(H_{\gll_0}(T(2, 2n+1)))(t,q) = \sum_{i=0}^{2n} (tq^{-2})^iq^{2n}.
\]

Note that it coincides with the Poincaré polynomial of knot Floer homology $\widehat{H\!F\!K}$.

\appendix

\section{Dimension of $\syf'_0(\Gamma_\star)$ for $\Gamma_\star$ of depth $1$. }
\label{sec:depth1}
The aim of this appendix is to show the following proposition:

\begin{prop}
\label{prop:gd-depth1}
  If $\Gamma_\star$ is a marked vinyl graph of depth $1$, then
  \[
\egd(\Gamma) = \gd(\Gamma_\star).
  \]
\end{prop}

We endow $\NN[q, q^{-1}]$ with a partial order: if $P$ and $Q$ are two Laurent polynomials in $q$ with positive integer coefficients, we write $P\leq Q$ if for every $i$ in $\ZZ$ the coefficient of $q^i$ of $P$ is lower than or equal to that of $Q$. 

\begin{obs}
  The proof of Proposition~\ref{prop:move-base-point} gives immediately that $\gd(\Gamma_\star)\leq \egd(\Gamma)$. 
  Indeed, we construct a surjective morphism from the homology of a certain complex which is concentrated in a single homological degree and has graded dimension $\egd(\Gamma)$ to $\syf'_0(\Gamma_\star)$. We  need Proposition~\ref{prop:gd-depth1} to prove that this morphism is an isomorphism.

  Note that the complete proof of Proposition~\ref{prop:move-base-point} relies on Proposition \ref{prop:gd-depth1} but the surjectivity of the morphism does not. 
\end{obs}

In this section $\skeinp(1,k)$ denotes the $\ZZ[q, q^{-1}]$ module generated by elementary marked vinyl graphs of level $k$ and of depth $1$ subject to relations~(\ref{eq:Hecke-diag}). Because of next lemma, for proving Proposition~\ref{prop:gd-depth1} we can assume $\Gamma_*$ to be in a small family of such marked vinyl graphs.

\begin{lem}
  \label{lem:skeinp1k}
   The $\ZZ[q, q^{-1}]$-module $\skeinp(1,k)$ is generated by marked vinyl graph 
  which are closure of $\Pi_{k,n}$ given by:
  \[
    \Pi_{k,n} := \NB
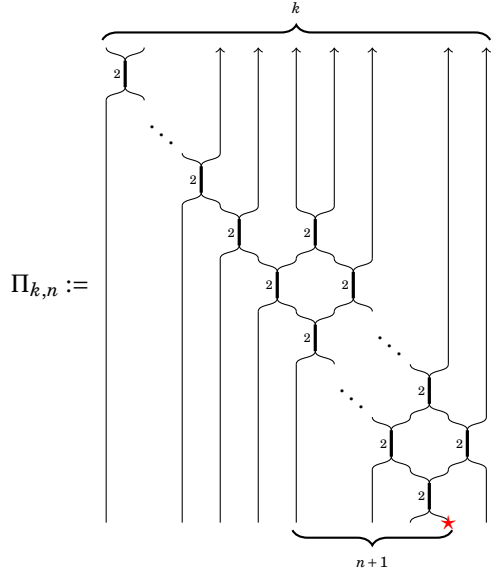
{
\tikz[scale=0.5]{\begin{scope}[yscale=0.7]
\newcommand{\dumble}[2]{
  \begin{scope}[xshift=#1cm, yshift =#2cm]
    \draw (-0.5,-1) .. controls +(0,0.3) and +(0,-0.3) .. (0, -0.5)--(0,0.5) node[midway, scale=0.5, left] {$2$} .. controls +(0,0.3) and +(0,-0.3) .. (-0.5, 1);
    \draw ( 0.5,-1) .. controls +(0,0.3) and +(0,-0.3) .. (0, -0.5)--(0,0.5)                                     .. controls +(0,0.3) and +(0,-0.3) .. ( 0.5, 1);
    \draw[very thick] (0, -0.5)--(0,0.5);
  \end{scope}
  }
  \dumble{0}{0}
  \dumble{-1}{-2}
  \dumble{-1}{2}
  \dumble{-2}{0}
  \dumble{-3}{6}
  \dumble{-4}{4}
  \dumble{-4}{8}
  \dumble{-5}{6}
  \dumble{-6}{8}
  \dumble{-7}{10}
  \dumble{-9}{14}
  \draw [<-] (-6.5, 15) -- (-6.5, 11);
  \draw [<-] (-5.5, 15) -- (-5.5, 9);
  \draw [<-] (-4.5, 15) -- (-4.5, 9);
  \draw [<-] (-3.5, 15) -- (-3.5, 9);
  \draw [<-] (-2.5, 15) -- (-2.5, 7);
  \draw [<-] (-0.5, 15) -- (-0.5, 3);
  \draw [<-] ( 0.5, 15) -- ( 0.5, 1);
  \draw  (-9.5, -3) -- (-9.5, 13);
  \draw  (-7.5, -3) -- (-7.5, 9);
  \draw  (-6.5, -3) -- (-6.5, 7);
  \draw  (-5.5, -3) -- (-5.5, 5);
  \draw  (-4.5, -3) -- (-4.5, 3);
  \draw  (-2.5, -3) -- (-2.5, -1);
  \draw  ( 0.5, -3) -- ( 0.5, -1); 
  \node at (-2  , 4) {$\ddots$};
  \node at (-3  , 2) {$\ddots$};
  \node at (-8  , 12) {$\ddots$};
  \draw [thick, decoration={ brace,  raise=0.1cm, amplitude = 0.2cm }, decorate ] (-9.6, 15) -- (0.6, 15);
    \node[font= \tiny] at (-4.5, 16.5) {$k$};
    \draw [thick, decoration={ brace,  raise=0.1cm, mirror, amplitude = 0.2cm }, decorate ] (-4.6, -3) -- (-0.4, -3);
    \node[font= \tiny] at (-2.5, -4.5) {$n+1$};
    \node[red] at (-0.5, -3) {$\star$};
  \end{scope}}
    }
  \]
  and by non-connected elementary vinyl graphs.
\end{lem}

\begin{proof}
  This is an easy consequence of Corollary~\ref{cor:base-induction} as in the proof of Lemma~\ref{lem:baddigon-egd}.
\end{proof}

If a marked vinyl graph $\Gamma_\star$ is not connected, we know that both $\egd(\Gamma_\star)$ and $\gd(\Gamma_\star)$ are zero. Hence we only need to focus on graphs of the form $\widehat{\Pi_{\bullet,\bullet}}$. For such graphs the expected graded dimension satisfies an induction formula. This induction formula defines the so-called quantum Pell numbers. 

\begin{dfn}
  \label{dfn:quantum-Pell-numbers}
  For $n \in \NN$, the $n$th quantum Pell number is the element $p_n$ of $\NN[q, q^{-1}]$ defined by:
  \begin{align*}
    p_0 &= 1, \qquad   p_1 = [2], \quad\text{and} \quad  p_{i+1} =  [2] p_i+ p_{i-1}  \quad\text{for $i \geq 2$}.
  \end{align*}
\end{dfn}

\begin{lem}\label{lem:egd-pell}
  For any integers $n \leq k-1$, $\egd(\widehat{\Pi_{k,n}}) = p_n$.
\end{lem}

\begin{proof}
  From Corollary ~\ref{cor:egd-baddigon}, it is enough to consider the graphs $\pi_n$ given by: 
  \[
    \pi_n:=\NB
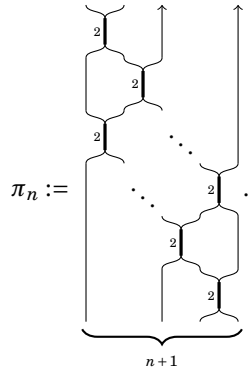
{    \tikz[scale=0.5]{\begin{scope}[yscale=0.7]
\newcommand{\dumble}[2]{
  \begin{scope}[xshift=#1cm, yshift =#2cm]
    \draw (-0.5,-1) .. controls +(0,0.3) and +(0,-0.3) .. (0, -0.5)--(0,0.5) node[midway, scale=0.5, left] {$2$} .. controls +(0,0.3) and +(0,-0.3) .. (-0.5, 1);
    \draw ( 0.5,-1) .. controls +(0,0.3) and +(0,-0.3) .. (0, -0.5)--(0,0.5)                                     .. controls +(0,0.3) and +(0,-0.3) .. ( 0.5, 1);
    \draw[very thick] (0, -0.5)--(0,0.5);
  \end{scope}
  }
  \dumble{-1}{-2}
  \dumble{-1}{2}
  \dumble{-2}{0}
  \dumble{-3}{6}
  \dumble{-4}{4}
  \dumble{-4}{8}
  \draw (-0.5, -1) -- (-0.5,1);
  \draw (-4.5, 5) -- (-4.5,7); 
  \draw [<-] (-2.5, 9) -- (-2.5, 7);
  \draw [<-] (-0.5, 9) -- (-0.5, 3);
  \draw  (-4.5, -3) -- (-4.5, 3);
  \draw  (-2.5, -3) -- (-2.5, -1);
  \node at (-2  , 4) {$\ddots$};
  \node at (-3  , 2) {$\ddots$};
    \draw [thick, decoration={ brace,  raise=0.1cm, mirror, amplitude = 0.2cm }, decorate ] (-4.6, -3) -- (-0.4, -3);
    \node[font= \tiny] at (-2.5, -4.5) {$n+1$};
  \end{scope}}}.
  \]
  We have $\egd(\pi_0)=1$  and $\egd(\pi_1) = [2]$.
  We use relations~(\ref{eq:Hecke-diag}) together with Corollary~\ref{cor:egd-baddigon}. This gives for $i \geq 2:$
  \[
    \egd(\pi_{i+1}) = [2] \egd(\pi_i) + \egd(\pi_{i-1}).
  \]
  Hence, we have $\egd(\Pi_{k,n}) = \egd(\pi_n) = p_n$.
\end{proof}

For what follows, it is convenient to have an alternative description of the quantum Pell numbers:

\begin{dfn}
  We consider a set of three dominos: $\domp$, $\domm$ and $\domoo$. The first two have one \emph{box}, the last one has two \emph{boxes}. Each of these dominos has a \emph{length} (a positive integer) denoted by $\ell$, a \emph{weight} (an integer) denoted by $w$ and a \emph{dual} (a domino) denoted by $\iota$: 
  \begin{align*}
    \ell (\domp) &= 1, &&& w(\domp)&= +1, &&& \iota(\domp) &= \domm \\
    \ell (\domm) &= 1, &&& w(\domm)&= -1,&&&  \iota(\domm) &= \domp\\
    \ell (\domoo) &= 2, &&& w(\domoo)&= 0, &&& \iota(\domoo) &= \domoo.
  \end{align*}
  Hence the length of a domino is its number of boxes.
  
  A \emph{domino configuration} is a finite sequence of dominos. Its \emph{length} is the sum of the lengths of the dominos it contains (\ie{}its number of boxes). Similarly, its \emph{weight} it the sum of the weights of the dominos it contains.
  The $i$th box of a configuration $s$ is denoted by $s_i$.
  The \emph{dual} of a configuration $s$ is the sequence consisting of duals of the dominos of $s$.   The set of all domino configurations of length $n$ is denoted by $\Dom_n$.
  The set $\Dom_n$ is endowed with a total order induced anti-lexicographically by
  $ \domm < \domoo <\domp$.
\end{dfn}
\begin{exa}
  Let $s= \NB{\tikz{
\begin{scope}[scale=0.15]
\draw (0, -1) -- (0, 1) -- (20,1) -- (20, -1) -- (0, -1); 
\draw (2, -1) -- ++(0,2);  
\draw (4, -1) -- ++(0,2);  
\draw (6, -1) -- ++(0,2);  
\draw (8, -1) -- ++(0,2); 
\draw (10, -1) -- ++(0,2);  
\draw (12, -1) -- ++(0,2);  
\draw (14, -1) -- ++(0,2);  
\draw (16, -1) -- ++(0,2);  
\draw (18, -1) -- ++(0,2);  
\node[scale =0.7] at (1,0) {$+$};
\node[scale =0.7] at (3,0) {$-$};
\node[scale =0.7] at (5,0) {$+$};
\node[scale =0.7] at (7,0) {$0_a$};
\node[scale =0.7] at (9,0) {$0_b$};
\node[scale =0.7] at (11,0) {$-$};
\node[scale =0.7] at (13,0) {$+$};
\node[scale =0.7] at (15,0) {$0_a$};
\node[scale =0.7] at (17,0) {$0_b$};
\node[scale =0.7] at (19,0) {$-$};
\end{scope}

}}$ and  $s'= \NB{\tikz{
\begin{scope}[scale=0.15]
\draw (0, -1) -- (0, 1) -- (20,1) -- (20, -1) -- (0, -1); 
\draw (2, -1) -- ++(0,2);  
\draw (4, -1) -- ++(0,2);  
\draw (6, -1) -- ++(0,2);  
\draw (8, -1) -- ++(0,2); 
\draw (10, -1) -- ++(0,2);  
\draw (12, -1) -- ++(0,2);  
\draw (14, -1) -- ++(0,2);  
\draw (16, -1) -- ++(0,2);  
\draw (18, -1) -- ++(0,2);  
\node[scale =0.7] at (1,0) {$+$};
\node[scale =0.7] at (3,0) {$-$};
\node[scale =0.7] at (5,0) {$+$};
\node[scale =0.7] at (7,0) {$-$};
\node[scale =0.7] at (9,0) {$0_a$};
\node[scale =0.7] at (11,0) {$0_b$};
\node[scale =0.7] at (13,0) {$+$};
\node[scale =0.7] at (15,0) {$0_a$};
\node[scale =0.7] at (17,0) {$0_b$};
\node[scale =0.7] at (19,0) {$-$};
\end{scope}

}}$. Both these configurations have length $10$ and weight $0$, and the first is lower than the second.  
\end{exa}

A straightforward induction gives the following lemma.

\begin{lem}
  For any integer $n\geq 1$, the following identity holds:
  \[
\sum_{s \in \Dom_n} q^{w(s)} = p_n.
  \] 
\end{lem}

Let $n<k$ be two non-negative integers. We will associate with every domino configuration $s$ of $\Dom_{n}$ an element $x(s)$ of $\syf'_0(\Pi_{k,n})$ of degree $w(s)$ and prove that the vectors $(x(s))_{s\in \Dom_n}$ are linearly independent.

All the graphs that we consider are elementary (their edges have labels $1$ or $2$) furthermore all edges with label 2 are non-circular. Such edges with label 2 are called \emph{dumbles}.

\begin{notation}
  The graph $\Pi_{k,n}$ has $(k-1)+ n$ dumbles. For later use it is convenient to give names to some of them. For $i \in \{1, \dots, n\}$ there are two dumbles in $\Pi_{k,n}$ between strands $i+k-n-2$ and $i+k-n-1$. We denote the the upper one $e_i^{\mathrm{top}}$, and the lower one $e_i^{\mathrm{bot}}$. This is illustrated in Figure~\ref{fig:Pikn-dumbles}.
\end{notation}

\begin{figure}[ht]
  \centering \tikz[scale=0.5]{\input{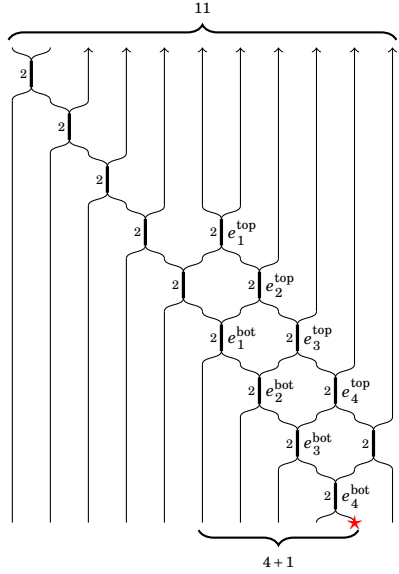}}
  \caption{The graph $\Pi_{11,4}$ with names of its dumbles.}
  \label{fig:Pikn-dumbles}
\end{figure}

\begin{dfn}
  Let $\Gamma_\star$ be an elementary marked vinyl graph and $e$ a dumble of $\Gamma$. We define two degree 2 endomorphisms $V_e$ and $S_e$ of $\syf_0(\Gamma_\star)$ which are given by (linear combination of) foams of the form $\Gamma_\star \times [0,1]$ with some decorations in the neighborhood of $e\times [0,1]$:
  \begin{align*}
    V_e &= \syf'_0\left(\NB{\tikz[scale=0.7]{\tdplotsetmaincoords{-125}{200}
\begin{scope}[tdplot_main_coords]
  \coordinate (A1B) at (0, 0,     0);
  \coordinate (A1T) at (0, 0,     2);  
  \coordinate (A2B) at (1, 0,     0);
  \coordinate (A2T) at (1, 0,     2);  
  \coordinate (C1B) at (0.5, 0.5, 0);
  \coordinate (C1T) at (0.5, 0.5, 2);  
  \coordinate (C2B) at (0.5, 1.5, 0);
  \coordinate (C2T) at (0.5, 1.5, 2);  
  \coordinate (B1B) at (0, 2,     0);
  \coordinate (B1T) at (0, 2,     2);  
  \coordinate (B2B) at (1, 2,     0);
  \coordinate (B2T) at (1, 2,     2);
  \coordinate (NW) at (1, 2,     2);  
  \filldraw [draw= black, fill =green, fill opacity =0.4] (A2B) --  (A2T) -- (C1T) -- (C1B) -- (A2B);
  \filldraw [draw= black, fill =green, fill opacity =0.4] (B2B) --  (B2T) -- (C2T) -- (C2B) -- (B2B);
  \filldraw [draw= black, fill =red, fill opacity =0.4]   (A1B) --  (A1T) -- (C1T) -- (C1B) -- (A1B);
  \filldraw [draw= black, fill =red, fill opacity =0.4]   (B1B) --  (B1T) -- (C2T) -- (C2B) -- (B1B);
  \filldraw [draw= black, thick, fill =blue, fill opacity =0.4]  (C1B) --  (C1T) -- (C2T) -- (C2B) -- (C1B);
  \node (NW) at  ($(B2B)!0.5!(C2T)$) {$\bullet$};
\end{scope}}}\right)
         - \syf'_0\left( \NB{\tikz[scale=0.7]{\tdplotsetmaincoords{-125}{200}
\begin{scope}[tdplot_main_coords]
  \coordinate (A1B) at (0, 0,     0);
  \coordinate (A1T) at (0, 0,     2);  
  \coordinate (A2B) at (1, 0,     0);
  \coordinate (A2T) at (1, 0,     2);  
  \coordinate (C1B) at (0.5, 0.5, 0);
  \coordinate (C1T) at (0.5, 0.5, 2);  
  \coordinate (C2B) at (0.5, 1.5, 0);
  \coordinate (C2T) at (0.5, 1.5, 2);  
  \coordinate (B1B) at (0, 2,     0);
  \coordinate (B1T) at (0, 2,     2);  
  \coordinate (B2B) at (1, 2,     0);
  \coordinate (B2T) at (1, 2,     2);
  \coordinate (NW) at (1, 2,     2);  
  \filldraw [draw= black, fill =green, fill opacity =0.4] (A2B) --  (A2T) -- (C1T) -- (C1B) -- (A2B);
  \filldraw [draw= black, fill =green, fill opacity =0.4] (B2B) --  (B2T) -- (C2T) -- (C2B) -- (B2B);
  \filldraw [draw= black, fill =red, fill opacity =0.4]   (A1B) --  (A1T) -- (C1T) -- (C1B) -- (A1B);
  \filldraw [draw= black, fill =red, fill opacity =0.4]   (B1B) --  (B1T) -- (C2T) -- (C2B) -- (B1B);
  \filldraw [draw= black, thick, fill =blue, fill opacity =0.4]  (C1B) --  (C1T) -- (C2T) -- (C2B) -- (C1B);
  \node (NW) at  ($(A2T)!0.25!(C1B)$) {$\bullet$};
\end{scope}}}\right) \\
    S_e &= \syf'_0\left(\NB{\tikz[scale=0.7]{}}\right)
         - \syf'_0\left( \NB{\tikz[scale=0.7]{\tdplotsetmaincoords{-125}{200}
\begin{scope}[tdplot_main_coords]
  \coordinate (A1B) at (0, 0,     0);
  \coordinate (A1T) at (0, 0,     2);  
  \coordinate (A2B) at (1, 0,     0);
  \coordinate (A2T) at (1, 0,     2);  
  \coordinate (C1B) at (0.5, 0.5, 0);
  \coordinate (C1T) at (0.5, 0.5, 2);  
  \coordinate (C2B) at (0.5, 1.5, 0);
  \coordinate (C2T) at (0.5, 1.5, 2);  
  \coordinate (B1B) at (0, 2,     0);
  \coordinate (B1T) at (0, 2,     2);  
  \coordinate (B2B) at (1, 2,     0);
  \coordinate (B2T) at (1, 2,     2);
  \coordinate (NW) at (1, 2,     2);  
  \filldraw [draw= black, fill =green, fill opacity =0.4] (A2B) --  (A2T) -- (C1T) -- (C1B) -- (A2B);
  \filldraw [draw= black, fill =green, fill opacity =0.4] (B2B) --  (B2T) -- (C2T) -- (C2B) -- (B2B);
  \filldraw [draw= black, fill =red, fill opacity =0.4]   (A1B) --  (A1T) -- (C1T) -- (C1B) -- (A1B);
  \filldraw [draw= black, fill =red, fill opacity =0.4]   (B1B) --  (B1T) -- (C2T) -- (C2B) -- (B1B);
  \filldraw [draw= black, thick, fill =blue, fill opacity =0.4]  (C1B) --  (C1T) -- (C2T) -- (C2B) -- (C1B);
  \node (NW) at  ($(A1B)!0.5!(C1T)$) {$\bullet$};
\end{scope}}}\right).
  \end{align*}
\end{dfn}

Let us describe $V_e$ and $S_e$ in the language of section~\ref{sec:1-dimens-appr}. A dumble has four edges of label $1$ adjacent to it, denoted by $SE(e)$, $SW(e)$, $NE(e)$ and $NW(e)$, where $E, N, S$ and $W$ stands for {\bf e}ast, {\bf n}orth, {\bf s}outh and {\bf w}est. The morphism $V_e$ is the difference of adding one dot on $NW(e)$ and adding one dot on $SW(e)$.
The morphism $S_e$ is the difference of adding one dot on $NW(e)$ and adding one dot on $SE(e)$.

\[
  V_e \leftrightsquigarrow
  \NB{\tikz[scale=0.7]{  \begin{scope}[xshift=0cm, yshift =0cm]
    \draw[->] (-0.5,-1) .. controls +(0,0.3) and +(0,-0.3) .. (0, -0.5) --(0,0.5) node[pos =0.3, scale=0.5, right]  {$2$} .. controls +(0,0.3) and +(0,-0.3) .. (-0.5, 1) node [midway] {$\bullet$};
    \draw[->] ( 0.5,-1) .. controls +(0,0.3) and +(0,-0.3) .. (0, -0.5) --(0,0.5) .. controls +(0,0.3) and +(0,-0.3) .. ( 0.5, 1);
  \end{scope}
}} -
    \NB{\tikz[scale=0.7]{  \begin{scope}[xshift=0cm, yshift =0cm]
    \draw[->] (-0.5,-1) .. controls +(0,0.3) and +(0,-0.3) .. (0, -0.5) node [midway] {$\bullet$} --(0,0.5) node[pos =0.3, scale=0.5, right]  {$2$} .. controls +(0,0.3) and +(0,-0.3) .. (-0.5, 1); 
    \draw[->] ( 0.5,-1) .. controls +(0,0.3) and +(0,-0.3) .. (0, -0.5) --(0,0.5) .. controls +(0,0.3) and +(0,-0.3) .. ( 0.5, 1);
  \end{scope}
}}\qquad\textrm{and}\qquad
  S_e \leftrightsquigarrow
  \NB{\tikz[scale=0.7]{}} -
    \NB{\tikz[scale=0.7]{  \begin{scope}[xshift=0cm, yshift =0cm]
    \draw[->] (-0.5,-1) .. controls +(0,0.3) and +(0,-0.3) .. (0, -0.5)  --(0,0.5) node[pos =0.3, scale=0.5, right]  {$2$} .. controls +(0,0.3) and +(0,-0.3) .. (-0.5, 1); 
    \draw[->] ( 0.5,-1) .. controls +(0,0.3) and +(0,-0.3) .. (0, -0.5) node [midway] {$\bullet$} --(0,0.5) .. controls +(0,0.3) and +(0,-0.3) .. ( 0.5, 1);
  \end{scope}
}}.
  \]

\begin{rmk}
  If $e$ and $e'$ are two different dumbles of a marked vinyl graph $\Gamma$, then $V_e$ and $S_e$ commute with $V_{e'}$ and $S_{e'}$. 
\end{rmk}

The letters $V$ and $S$ come from {\bf v}irtual and {\bf s}moothing as explained in the next lemma whose proof directly follows from the definition of $\kup{\bullet}_k$. 

\begin{lem}\label{lem:coloringVeSe}
  Let $\Gamma_\star$ be a marked vinyl graph of level $k$, $e$ be a dumble of $\Gamma$, $F$ a  tree-like $(\SS^k,\Gamma)$-foam and $G$ be a tree-like $\Gamma,\SS^k)$-foam. If $c$ is a coloring of $G\circ V_e \circ F$, which in the neighborhood of $e$ is locally given by
  \[ 
\NB
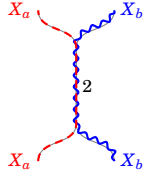
{\tikz{\input{\imagesfolder/alex_dumble-col-smooth}}}
 \]
  then $\kup{\mathrm{cl}(G\circ V_e \circ F),c}_k = 0$.
  Similarly, If $c$ is a coloring of $G\circ V_e \circ F$, which in the neighborhood of $e$ is locally given
  \[ 
\NB
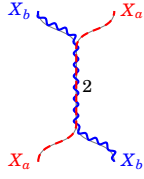
{\tikz{\input{\imagesfolder/alex_dumble-col-virtual}}}
 \]
  then $\kup{\mathrm{cl}(G\circ S_e \circ F),c}_k = 0$.
\end{lem}

\begin{rmk}\label{rmk:Ve_unzipzip}
  The morphism $V_e$ can be seen as the composition of an unzip and a zip. This follows either from \cite[relation (12)]{RW1} or from the $1$-dimensional point of view, see discussion after Proposition~\ref{prop:gl1-1dim-is-gl1}.
\end{rmk}

\begin{notation}
  \begin{enumerate}
  \item If $\Gamma_\star$ is a marked vinyl graph, we denote by
    $\dry_{\Gamma}$ the element of $\syf'_{0}(\Gamma_\star)$
    given by the dry (\ie without any decoration) tree-like $(\SS_k,\Gamma)$-foam.
  \item If $\Gamma_\star$ is marked vinyl graph and $e$ is a dumble of $\Gamma$, and $b$ is a box (\ie an element of $\{\domp, \domm, \dom{0_a}, \dom{0_b}\}$), the  we define
    \begin{align*}
      x(b,e) =
      \begin{cases}
        \id_{\Gamma}& \text{if $b = \domm$ or $b= \dom{0_a}$,}\\
        S_e& \text{if $b = \domp$,}\\
        V_e& \text{if $b = \dom{0_b}$}
      \end{cases}
    \end{align*}
  \end{enumerate} 
\end{notation}

\begin{dfn}
  Let $n$ be a non-negative integer, $k>n$ be a positive integer and $s$ a domino configuration of length $n$. We define $x(s)$ to be the element of $\syf'_{0}(\Pi_{k,n})$ which is the image of $\dry_{\Pi_{k,n}}$ by the morphism
  \[
    \prod_{i=1}^n x(s_i,e_i^{\mathrm{top}}).
  \]
  Similarly, we define $y(s)$ to be the element of $\syf'_{0}(\Pi_{k,n})$ which is the image of $\dry_{\Pi_{k,n}}$ by the morphism
  \[
    \prod_{i=1}^n x(s_i,e_i^{\mathrm{bot}}).
  \]
\end{dfn}

\begin{lem}\label{lem:ssprime-zero}
  For $s$ and $s'$ two  domino configuration of length $n$, then if $s'$ is strictly larger than the dual of $s$, then $(x(s), y(s'))_{\Pi_{k,n}} =0$ (see Proposition~\ref{prop:Xk-1-to-UC} for notations).
\end{lem}

 \begin{proof}
   First of all, for degree reasons, if $w(s)+w(s') \neq 0$ then $(x(s), y(s'))_{\Pi_{k,n}} =0$. We can suppose that $w(s) + w(s')=0$. Suppose that $s$ is stricly larger than the dual of $s'$. This implies that we can find $i \in \{1, \dots, n\}$ such that one of the following three situation happens:
  \begin{align*}
    \begin{cases}
      s_i = \domp,\\
      s'_i= \domp,
    \end{cases}\quad
    \begin{cases}
      s_i = \domp\\
      s'_i= \dom{0_b},
    \end{cases}\quad
    \begin{cases}
      s_i = \dom{0_b}\\
      s'_i= \domp,
    \end{cases}\quad
  \end{align*}
  In the first case, thanks to Remark~\ref{rmk:Ve_unzipzip}, we have:
  $(x(s), y(s'))_{\Pi_{k,n}} = (x', y')_{\Gamma_\star}$ with $\Gamma$ the disconnected marked vinyl graph obtained by smoothing the dumbles $e_i^{\textrm{top}}$ and $e_i^{\textrm{bot}}$ and $x'$ and $y'$ some foams. Hence, by Proposition~\ref{prop:syf0-0-split}, $(x(s), y(s'))_{\Pi_{k,n}}=0$.

  In the last two cases, we claim that for every coloring $c$ of the foam $\overline{y(s')}\circ X_e(\Pi_{k,n}) \circ x(s)$, $\kups{ \overline{y(s')}\circ X_e(\Pi_{k,n}) \circ x(s),c}_1=0$, where $e$ denote the edge of $\Gamma$ the base point of $\Pi_{k,n}$ belongs to and $X_e(\Pi_{k,n})$ is the foam $\Pi_{k,n}\times [0,1]$ with $k-1$ dots given in Notation~\ref{not:addadotatp}. This follows from Lemma~\ref{lem:coloringVeSe}. 
\end{proof}

\begin{prop}\label{prop:ssprime-nonzero}
Let $s$ be a domino configuration of length $n$ and denote byy $s'$ its dual. Then $(x(s), y(s'))_{\Pi_{k,n}} \neq 0$.  
\end{prop}

Before proving this proposition we introduce some combinatorial gadget which makes the computation easier.

\begin{dfn}
  Let $T= (V(T), E(T))$ be a rooted directed tree with $k$ vertices, and $\phi$ be a bijection  between $V(T)$ and $\{X_1, \dots X_k\}$. If we denote by $r \in V(T)$ the root of $T$, we define
  \[
    \kup{T, \phi} = \frac{\prod_{v\in V(T)\setminus\{r\}} \phi(v)}{\prod_{v_1 \to v_2 \in E(T)} (\phi(v_1)- \phi(v_2))}
  \]
  and
  \[
    \kup{T} = \sum_{\phi \in \mathrm{bij}(V(T), \{X_1, \dots, X_k\})} \kup{T, \phi}
  \]
\end{dfn}

\begin{lem}\label{lem:tree2non0}
  For any rooted directed tree with $k$ vertices, $\kup{T}$ is a non-zero integer.
\end{lem}

\begin{proof}
  First, $\kup{T}$ is a symmetric rational function of degree $0$ and the denominators are all products of $(X_i-X_j)$. In order to prove that it is an integer, it is enough to show that it is in fact a polynomial. This is more or less the same proof as \cite[Proposition 2.19]{RW1} or \cite[Lemma 2.14]{KR1}. Since it is symmetric, it is enough to prove that
  \[
    \kup{T} \in \ZZ[X_1, \dots, X_k]\left[(X_i-X_j)^{-1}_{\substack{1\leq i <j \leq k \\ (i,j)\neq (1,2)}}\right].
     \]
    The factor $(X_1-X_2)$ appears in the denominator of $\kup{T, \phi}$ if and only if $\phi^{-1}(X_1)$ and $\phi^{-1}(X_2)$ are neighbors. Such bijections come naturally in pairs which are related by swaping the pre-images of $X_1$ and $X_2$. For such a pair $(\phi_1, \phi_2)$, $\kup{T, \phi_1} + \kup{T, \phi_2}$ is symmetric in $X_1$ and $X_2$, hence when reducing to the same denominator, the numerator must be anti-symmetric in $X_1$ and $X_2$ and therefore divisible by $(X_1- X_2)$. Consequently we can simplify the denominator and obtain that 
    \[ \kup{T, \phi_1} + \kup{T, \phi_2} \in \ZZ[X_1, \dots, X_k]\left[(X_i-X_j)^{-1}_{\substack{1\leq i <j \leq k \\ (i,j)\neq (1,2)}}\right]
    \]
    and finally that
    \[ \kup{T} \in \ZZ[X_1, \dots, X_k]\left[(X_i-X_j)^{-1}_{\substack{1\leq i <j \leq k \\ (i,j)\neq (1,2)}}\right].
    \]
    
    We now need to prove that it is different from $0$. It is done by induction. If $T$ as only $1$ vertex (the root), this is clear.

    Suppose now $T$ has $k$ vertices with $k\geq 2$ and that the root $r$ is a leaf of $T$. Since, $\kup{T}$ is a polynomial, we can decide to evaluate the variable $X_k$ to zero in $\kup{T, \phi}$ without changing the value of $\kup{T}$. The only terms in the sum defining $\kup{T}$ will then be the ones which correspond to $\phi$ mapping $r$ to $X_k$. For such a bijection, $\kup{T, \phi}_{X_k=0} = \pm \kup{T', \phi'}$, for $T'$ the rooted directed tree given by removing $r$ from $T$ and setting the root of $T'$ to be the only neighbor $v$ of $r$ in $T$, and $\phi'$ is the map induced by $\phi$. The signs is given by the orientation of the edge $(v,r)$ in $T$. 

    If $r$ is not a leaf but an arbitrary vertex of $T$. We can argue exactly similarly, except that $T\setminus\{r\}$ is not longer a tree but a forest $T'_1\cup \dots \cup T'_l$ with say $k'_1, \dots, k'_l$ vertices with $k'_1 + \dots + k'_l = k-1$ vertices, and we have:
    \[\kup{T} =\pm
      \begin{pmatrix}
        k-1 \\  k'_1\,\,  \dots \,\, k'_l
      \end{pmatrix}
      \prod_{i=1}^l \kup{T'_i},
    \]
    where the binomial comes from the different way to separate the remaining $k-1$ variables in $l$ subsets of $k'_1, \dots, k'_l$ variables.
\end{proof}

\begin{proof}[Proof of Proposition~\ref{prop:ssprime-nonzero}]
  Following Proposition~\label{prop:Xk-1-to-UC}, we compute $(\overline{y(s')}, x(s))_{\widehat{\Pi_{k,n}}}$. We can restrict the coloring we consider: for the dumbles corresponding to boxes $\domp$ and $\dom{0_b}$ of $s$ and $s'$, this follows from Lemma~\ref{lem:coloringVeSe}. For the other dumbles, there is no choice, all colorings are locally given by:
  \[  \NB
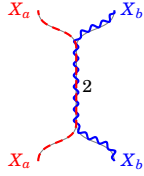
{\tikz{\input{\imagesfolder/alex_dumble-col-smooth}}} \]
  This means that up to an action of $S_k$, we have only one coloring. Every coloring induces an omnicoloring $(C_1, \dots, C_k)$ of the graph $\widehat{\Pi_{k,n}}$. Let us construct a directed graph $G$ from this omnicoloring: the vertices are $C_1, \dots, C_k$ and the edges are given by dumbles which either corresponds to a box of $s$ or $s'$ of type $\domm$ or $\dom{0_a}$ or do not correspond to any box: they joint the two cycles which passes through this dumble and are oriented by the following rule:
  \[
   \NB
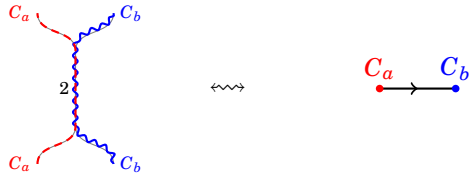
{\tikz{\input{\imagesfolder/alex_dumble-col-smooth-2}}}  
  \]
  The graph $G$ is a tree: indeed, first of all, there are exactly $k-1$ edges and if we number $C_1, \dots, C_k$, such that at the bottom of $\Pi_{k,n}$ we see from left to right $C_1, \dots, C_k$, any cycle $C_i$ for $i\geq 2$ is a neighbor of a $C_j$ for $j<i$ proving that the graph is connected. This last fact is easily deduced from the local behavior of the coloring. The difficult cases is depicted below (the other ones being trivial):
  \[
\NB
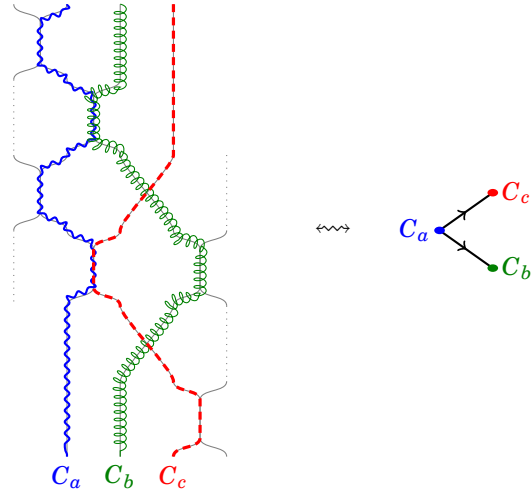
{\tikz[yscale=0.5, xscale =0.7]{  \begin{scope}[xshift=0cm, yshift =0cm]
    \coordinate  (A1) at (1,12);
    \coordinate (ab1) at (1.5,11.5);
    \coordinate (ba1) at (1.5,10.5);
    \coordinate  (B1) at (1,10);
    \coordinate (bc1) at (1.5,9.5);
    \coordinate (cb1) at (1.5,8.5);
    \coordinate  (C1) at (1,8);
    \coordinate (cd1) at (1.5,7.5);
    \coordinate (dc1) at (1.5,6.5);
    \coordinate  (D1) at (1,6);
    \coordinate (de1) at (1.5,5.5);
    \coordinate (ed1) at (1.5,4.5);
    \coordinate  (E1) at (1,4);
    \coordinate (ef1) at (1.5,3.5);
    \coordinate (fe1) at (1.5,2.5);
    \coordinate  (F1) at (1,2);
    \coordinate (fg1) at (1.5,1.5);
    \coordinate (gf1) at (1.5,0.5);
    \coordinate  (G1) at (1,0);

    \coordinate  (A2) at (2,12);
    \coordinate (ab2) at (2.5,11.5);
    \coordinate (ba2) at (2.5,10.5);
    \coordinate  (B2) at (2,10);
    \coordinate (bc2) at (2.5,9.5);
    \coordinate (cb2) at (2.5,8.5);
    \coordinate  (C2) at (2,8);
    \coordinate (cd2) at (2.5,7.5);
    \coordinate (dc2) at (2.5,6.5);
    \coordinate  (D2) at (2,6);
    \coordinate (de2) at (2.5,5.5);
    \coordinate (ed2) at (2.5,4.5);
    \coordinate  (E2) at (2,4);
    \coordinate (ef2) at (2.5,3.5);
    \coordinate (fe2) at (2.5,2.5);
    \coordinate  (F2) at (2,2);
    \coordinate (fg2) at (2.5,1.5);
    \coordinate (gf2) at (2.5,0.5);
    \coordinate  (G2) at (2,0);

    \coordinate  (A3) at (3,12);
    \coordinate (ab3) at (3.5,11.5);
    \coordinate (ba3) at (3.5,10.5);
    \coordinate  (B3) at (3,10);
    \coordinate (bc3) at (3.5,9.5);
    \coordinate (cb3) at (3.5,8.5);
    \coordinate  (C3) at (3,8);
    \coordinate (cd3) at (3.5,7.5);
    \coordinate (dc3) at (3.5,6.5);
    \coordinate  (D3) at (3,6);
    \coordinate (de3) at (3.5,5.5);
    \coordinate (ed3) at (3.5,4.5);
    \coordinate  (E3) at (3,4);
    \coordinate (ef3) at (3.5,3.5);
    \coordinate (fe3) at (3.5,2.5);
    \coordinate  (F3) at (3,2);
    \coordinate (fg3) at (3.5,1.5);
    \coordinate (gf3) at (3.5,0.5);
    \coordinate  (G3) at (3,0);

    \coordinate  (A4) at (4,12);
    \coordinate (ab4) at (4.5,11.5);
    \coordinate (ba4) at (4.5,10.5);
    \coordinate  (B4) at (4,10);
    \coordinate (bc4) at (4.5,9.5);
    \coordinate (cb4) at (4.5,8.5);
    \coordinate  (C4) at (4,8);
    \coordinate (cd4) at (4.5,7.5);
    \coordinate (dc4) at (4.5,6.5);
    \coordinate  (D4) at (4,6);
    \coordinate (de4) at (4.5,5.5);
    \coordinate (ed4) at (4.5,4.5);
    \coordinate  (E4) at (4,4);
    \coordinate (ef4) at (4.5,3.5);
    \coordinate (fe4) at (4.5,2.5);
    \coordinate  (F4) at (4,2);
    \coordinate (fg4) at (4.5,1.5);
    \coordinate (gf4) at (4.5,0.5);
    \coordinate  (G4) at (4,0);

    \coordinate  (A5) at (5,12);
    \coordinate (ab5) at (5.5,11.5);
    \coordinate (ba5) at (5.5,10.5);
    \coordinate  (B5) at (5,10);
    \coordinate (bc5) at (5.5,9.5);
    \coordinate (cb5) at (5.5,8.5);
    \coordinate  (C5) at (5,8);
    \coordinate (cd5) at (5.5,7.5);
    \coordinate (dc5) at (5.5,6.5);
    \coordinate  (D5) at (5,6);
    \coordinate (de5) at (5.5,5.5);
    \coordinate (ed5) at (5.5,4.5);
    \coordinate  (E5) at (5,4);
    \coordinate (ef5) at (5.5,3.5);
    \coordinate (fe5) at (5.5,2.5);
    \coordinate  (F5) at (5,2);
    \coordinate (fg5) at (5.5,1.5);
    \coordinate (gf5) at (5.5,0.5);
    \coordinate  (G5) at (5,0);

    \coordinate (up) at (0,0.3);
    \coordinate (down) at (0,-0.3);

     \begin{scope}[gray]
      \draw
      (A1) .. controls +(down) and +(up) ..
      (ab1) .. controls +(down) and +(up) ..
      (ba1) .. controls +(down) and +(up) ..
      (B1);
      \draw[dotted] (B1) -- (C1);
      \draw
      (C1) .. controls +(down) and +(up) ..
      (cd1) .. controls +(down) and +(up) ..
      (dc1) .. controls +(down) and +(up) ..
      (D1);
      \draw[dotted] (D1) -- (E1);


      \draw
      (A2)  .. controls +(down) and +(up) ..
      (ab1) .. controls +(down) and +(up) ..
      (ba1) .. controls +(down) and +(up) ..
      (B2)  .. controls +(down) and +(up) ..
      (bc2) .. controls +(down) and +(up) ..
      (cb2) .. controls +(down) and +(up) ..
      (C2)  .. controls +(down) and +(up) ..
      (cd1) .. controls +(down) and +(up) ..
      (dc1) .. controls +(down) and +(up) ..
      (D2)  .. controls +(down) and +(up) ..
      (de2) .. controls +(down) and +(up) ..
      (ed2) .. controls +(down) and +(up) ..
      (E2)  .. controls +(down) and +(up) ..
      (G2);
      
      \draw
      (A3)  .. controls +(down) and +(up) ..      
      (B3)  .. controls +(down) and +(up) ..      
      (bc2) .. controls +(down) and +(up) ..      
      (cb2) .. controls +(down) and +(up) ..      
      (C3)  .. controls +(down) and +(up) ..      
      (D4)  .. controls +(down) and +(up) ..      
      (de4) .. controls +(down) and +(up) ..      
      (ed4) .. controls +(down) and +(up) ..      
      (E4)  .. controls +(down) and +(up) ..      
      (F3)  .. controls +(down) and +(up) ..      
      (G3);

      \draw
      (A4)  .. controls +(down) and +(up) ..
      (C4)  .. controls +(down) and +(up) ..
      (D3)  .. controls +(down) and +(up) ..
      (de2) .. controls +(down) and +(up) ..
      (ed2) .. controls +(down) and +(up) ..
      (E3)  .. controls +(down) and +(up) ..
      (F4)  .. controls +(down) and +(up) ..
      (fg4)  .. controls +(down) and +(up) ..
      (gf4)  .. controls +(down) and +(up) ..
      (G4);
      
      \draw
      (D5)  .. controls +(down) and +(up) ..
      (de4) .. controls +(down) and +(up) ..
      (ed4) .. controls +(down) and +(up) ..
      (E5);
      \draw[dotted] (E5) -- (F5);
      \draw
      (F5)  .. controls +(down) and +(up) ..
      (fg4) .. controls +(down) and +(up) ..
      (gf4) .. controls +(down) and +(up) ..
      (G5);
      \draw[dotted] (C5) -- (D5);

    \end{scope}

     \draw[blue, decorate, decoration={snake, segment length=1.5mm, amplitude=0.3mm}, thick]
      (A2)  .. controls +(down) and +(up) ..
      (ab1) .. controls +(down) and +(up) ..
      (ba1) .. controls +(down) and +(up) ..
      (B2)  .. controls +(down) and +(up) ..
      (bc2) .. controls +(down) and +(up) ..
      (cb2) .. controls +(down) and +(up) ..
      (C2)  .. controls +(down) and +(up) ..
      (cd1) .. controls +(down) and +(up) ..
      (dc1) .. controls +(down) and +(up) ..
      (D2)  .. controls +(down) and +(up) ..
      (de2) .. controls +(down) and +(up) ..
      (ed2) .. controls +(down) and +(up) ..
      (E2)  .. controls +(down) and +(up) ..
      (G2) node [below, scale = 1] {$C_a$};
      
      \draw[green!50!black, decorate, decoration={coil, segment length=1mm, amplitude=0.8mm}]
      (A3)  .. controls +(down) and +(up) ..      
      (B3)  .. controls +(down) and +(up) ..      
      (bc2) .. controls +(down) and +(up) ..      
      (cb2) .. controls +(down) and +(up) ..      
      (C3)  .. controls +(down) and +(up) ..      
      (D4)  .. controls +(down) and +(up) ..      
      (de4) .. controls +(down) and +(up) ..      
      (ed4) .. controls +(down) and +(up) ..      
      (E4)  .. controls +(down) and +(up) ..      
      (F3)  .. controls +(down) and +(up) ..      
      (G3) node [below, scale = 1] {$C_b$};

      \draw[very thick, red, densely dashed]
      (A4)  .. controls +(down) and +(up) ..
      (C4)  .. controls +(down) and +(up) ..
      (D3)  .. controls +(down) and +(up) ..
      (de2) .. controls +(down) and +(up) ..
      (ed2) .. controls +(down) and +(up) ..
      (E3)  .. controls +(down) and +(up) ..
      (F4)  .. controls +(down) and +(up) ..
      (fg4)  .. controls +(down) and +(up) ..
      (gf4)  .. controls +(down) and +(up) ..
      (G4) node [below, scale = 1] {$C_c$};

     \node at (7 ,6) {$\leftrightsquigarrow$};
     \draw[thick, ->-] (9,6) -- (10,7);
     \draw[thick, ->-] (9,6) -- (10,5);
     \fill[blue] (9,6) circle (1mm) node[left] {$C_a$};
     \fill[green!50!black] (10,5) circle (1mm) node[right] {$C_b$};
     \fill[red] (10,7) circle (1mm) node[right] {$C_c$};
  \end{scope}
}}
  \]
Now if we set that the cycle on which the base point of $\Pi_{k,n}$ lies to be the root of $G$, we have:
\[
(\overline{y(s')},x(s))_{\widehat{\Pi_{k,n}}} = \kup{G}.
\]
This concludes since $\kup{G} \neq 0$ thanks to Lemma~\ref{lem:tree2non0}.
\end{proof}

\begin{proof}[Proof of Proposition~\ref{prop:gd-depth1}]
  From Proposition~\ref{prop:ssprime-nonzero} and Lemma~\ref{lem:ssprime-zero}, we deduce that the family $(x(s))_{s\in \Dom_n}$ of elements of $\syf'_0(\widehat{\Pi_{k,n}})$ is linearly independant. Indeed the pairing matrix $(x(s),x(s'))_{\Pi_{k,n}}$ is anti-triangular (with respect to the order of $\Dom_n$) with non-zero scalars on the anti-diagonal. Moreover, for every $s$, $x(s)$ is homogeneous of degree $w(s)$. Hence we have $p_n \leq \gd(\widehat{\Pi_{k,n}})$. We knew that $\gd(\widehat{\Pi_{k,n}})\leq  \egd(\widehat{\Pi_{k,n}})$ and that $\egd(\widehat{\Pi_{k,n}}) = p_n$ from Lemma \ref{lem:egd-pell}. Finally we have $\egd(\widehat{\Pi_{k,n}}) = \gd(\widehat{\Pi_{k,n}})$. We deduce that $\egd(\Gamma) = \gd(\Gamma_*)$ for all marked vinyl graph of depth $1$ thanks to Lemma~\ref{lem:skeinp1k}.
\end{proof}

\bibliographystyle{alphaurl}
\bibliography{biblio}

\end{document}